\documentclass[12pt]{amsart}
\usepackage[utf8]{inputenc}
\usepackage[T1]{fontenc}
\usepackage{amsfonts}
\usepackage[leqno]{amsmath}
\usepackage{accents}
\usepackage{amssymb, comment, float}
\usepackage[dvipsnames]{xcolor}
\usepackage[foot]{amsaddr}
\usepackage[shortlabels]{enumitem}
\setlist[itemize]{leftmargin=11mm}
\setlist[enumerate]{leftmargin=11mm}
\usepackage{mathdots}
\usepackage[a4paper, footskip=0.5in,  headheight = 0.5in, top=1.25in, bottom=1.25in,  right=1in,  left=1in]{geometry}
\usepackage{hyperref}
\hypersetup{
colorlinks,
linkcolor=black,
urlcolor=Blue,
citecolor=black,}
\usepackage{cleveref}
\usepackage[center]{caption}
\usepackage{eufrak}
\usepackage{marvosym}     
\usepackage{latexsym}
\usepackage[nodayofweek]{datetime}
\usepackage{tikz}
\usepackage{subfig}
\usepackage{thm-restate}
\usepackage{verbatim}
\usepackage[none]{hyphenat} 

\newtheorem{statement}{statement}[section]
\newtheorem{theorem}[statement]{Theorem}
\newtheorem{lemma}[statement]{Lemma}
\newtheorem{conjecture}[statement]{Conjecture}
\newtheorem{corollary}[statement]{Corollary}
\newtheorem{question}[statement]{Question}

  \DeclareMathOperator{\lng}{{\text{\sf{L}}}}

\def\dd{\hbox{-}}   

\newcommand{\ol}{\overline}

\newcommand{\mca}{\mathcal}

\newcommand{\poi}{\mathbb{N}} 
\newcommand{\lp}{$^+$}
\newcommand{\lm}{$^-$}

\newcounter{tbox}
\newcommand{\sta}[1]{%
  \par\addvspace{0.5\baselineskip}
  \refstepcounter{tbox}%
  \noindent(\thetbox)\,{\em #1}\par%
  \addvspace{0.3\baselineskip}%
}

\newcommand{\squeezeline}[1]{%
  \noindent\hbox to \linewidth{#1\hfil}%
}

\title[Asymmetric induced saturation]{Asymmetric induced saturation}

\author{Xinyue Fan$^{\dagger}$}

\author{Sahab Hajebi$^{\dagger}$}

\author{Sepehr Hajebi$^{\dagger}$}

\author{Sophie Spirkl$^{\dagger\ast}$}

\thanks{$^{\dagger}$ Department of Combinatorics and Optimization, University of Waterloo, Waterloo, Ontario, Canada.}
\thanks{$^{\ast}$ We acknowledge the support of the Natural Sciences and Engineering Research Council of Canada (NSERC), [funding reference number RGPIN-2020-03912].
Cette recherche a \'et\'e financ\'ee par le Conseil de recherches en sciences naturelles et en g\'enie du Canada (CRSNG), [num\'ero de r\'ef\'erence RGPIN-2020-03912]. This project was funded in part by the Government of Ontario. This research was conducted while Spirkl was an Alfred P. Sloan Fellow. This research was undertaken, in part, thanks to funding from the Canada Research Chairs Program.}

\date{\today}

\begin{document}

\maketitle



\begin{abstract}
For which graphs $H$ does there exist a graph $G$ with at least one edge and no induced subgraph isomorphic to $H$, such that deleting any edge of $G$ creates an induced copy of $H$? We call such a graph \emph{$H$-deletion-saturated}. This version of the well-studied notion of \emph{$H$-induced-saturated} graphs -- where both adding and deleting any edge creates an induced copy of $H$ -- appears more tractable. For example, while it remains wide open whether $H$-induced-saturated graphs exist for every even cycle $H$, we proved recently that deletion-saturated graphs exist for all even cycles. In fact, apart from complete graphs, no graph $H$ is known for which $H$-deletion-saturated graphs do not exist. We conjecture that $H$-deletion-saturated graphs exist for every non-complete graph $H$, and prove this conjecture for several types of graphs, including: complete bipartite graphs with parts of unequal size, triangle-free graphs with one cycle, graphs with two leaves at distance at most three, and line graphs of trees. In fact, in all cases, we prove the conjecture for substantially more general families. We also verify our conjecture for every graph $H$ on at most six vertices.
\end{abstract}

\section{Introduction}

\subsection{Background and Motivation}

We explore the ``asymmetric'' relaxation of the induced saturation problem on graphs. Given a graph $H$, a graph $G$ is \textit{$H$-free} if $G$ has no induced subgraph isomorphic to $H$. A graph $G$ is \emph{$H$-induced-saturated} if $G$ is $H$-free, for every $e\in E(G)$, $G-e$ is not $H$-free; and for every $e\in E(\ol{G})$, $G+e$ is not $H$-free. Induced saturation originated in a 2012 work by Martin and Smith \cite{martin} (see also \cite{smiththesis}), where they defined this notion as a property of ``trigraphs'' and studied its extremal aspects. (This in turn was motivated by the heavily investigated topic of ``non-induced'' saturation in extremal combinatorics; see \cite{survey} for a survey.) Since then, the literature on this topic has evolved around the study of graphs $H$ for which $H$-induced-saturated graphs exist in the first place. It is readily seen that if a graph $H$ is complete or empty (except when $|V(H)|=2$), then there is no $H$-induced-saturated graph. Bonamy et al.\ recently proved \cite{carlaetal} that for all other graphs $H$, an ``infinite $H$-induced-saturated graph'' always exists.

\begin{theorem}[Bonamy, Groenland, Johnston, Morrison, Scott \cite{carlaetal}]\label{thm:carlaetal}
   For every finite graph $H$ that is neither complete nor empty (except when $|V(H)|=2$), there exists a countably infinite $H$-induced-saturated graph.
\end{theorem}

(In fact, they prove a stronger result, namely that there is a countably infinite $H$-free graph $G$ where any ``bounded-degree perturbation'' of $G$ creates an induced copy of $H$.)

\medskip

Our focus is on graphs $H$ for which there exists a finite $H$-induced-saturated graph. We call such graphs \textit{normal}. (From here on, all graphs in this paper have finite vertex sets, no loops, and no parallel edges.) It follows that complete graphs and edgeless graphs (unless on two vertices) are not normal. Martin and Smith \cite{martin} also proved that:

\begin{theorem}[Martin and Smith \cite{martin}]\label{thm:p4}
    The four-vertex path $P_4$ is not normal.
\end{theorem}

Complete graphs, edgeless graphs, and the four-vertex path remain the only graphs known not to be normal, and even for structurally rather simple graphs $H$, deciding whether $H$ is normal appears to be quite difficult. For example, it remains wide open whether all even cycles are normal.

\begin{question}\label{q:evennormal}
    Let $t\geq 2$ be an integer. Is it true that the $2t$-vertex cycle is normal?
\end{question}

\begin{figure}
    \centering
    \includegraphics[scale=0.7]{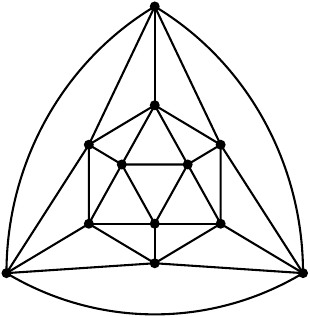}
    \caption{The icosahedron.}
    \label{fig:icosa}
\end{figure}

The answer to \Cref{q:evennormal} is only known for finitely many choices of $t$, and it is affirmative in those cases. (We are aware of constructions for $t\in \{2,3,4,5\}$ \cite{behrens,evencycle}; notably, the icosahedron from Figure~\ref{fig:icosa} is $C_4$-induced-saturated \cite{behrens}. We have also been told by Carla Groenland \cite{carlaprivate} that constructions for some larger values of $t$ have been found.) This is particularly curious because all odd cycles (except for the $3$-cycle, which is a complete graph) are known to be normal, and the proof is surprisingly simple.

\begin{theorem}[Behrens, Erbes, Santana, Yager, Yeager \cite{behrens}]\label{thm:odd}
  All odd cycles except for the $3$-cycle are normal: for every $t\geq 2$, the line graph of $K_{t+1,t+1}$ is $C_{2t+1}$-induced-saturated.
\end{theorem}

Another family of structurally simple graphs $H$ for which a proof of normality is surprisingly difficult is that of paths. For small paths, the solution is immediate: $P_1$ is normal (because the null graph is $P_1$-induced-saturated), $P_2$ is normal (because edgeless graphs on two or more vertices are $P_2$-induced-saturated), $P_3$ is normal (because complete graphs on three or more vertices are $P_3$-induced-saturated), and by \Cref{thm:p4}, $P_4$ is not normal. Moreover, Bonamy et al.\ \cite{bonamy,carlaetal} showed that the complement of the icosahedron is $P_5$-induced-saturated. Recently, Dvo\v{r}\'{a}k \cite{ptdvorak} found, for every integer $t\geq 6$, an elegant construction of a $P_t$-induced-saturated graph, hence completely resolving the case of paths: all paths except the one on four vertices are normal.

In addition, some other types of graphs $H$ were shown to be normal in subsequent work by Behrens et al.\ \cite{behrens} and by Axenovich and Csik\'{o}s \cite{axenovich}. Beyond that, a full characterization of all normal graphs remains wide open.

\subsection{Our take}

We focus on graphs in which $H$-freeness is only required to be sensitive to one of the two operations of deleting or adding edges. Given a graph $H$, we say that a graph $G$ is \textit{$H$-deletion-saturated} if $G$ is $H$-free with $E(G)\neq \varnothing$, and for every $e\in E(G)$, $G-e$ is not $H$-free. Similarly, we say that a graph $G$ is \textit{$H$-addition-saturated} if $G$ is $H$-free with $E(\ol{G})\neq \varnothing$, and for every $e\in E(\ol{G})$, $G+e$ is not $H$-free. We say that $H$ is \textit{deletion-normal} if there exists a graph $G$ that is $H$-deletion-saturated, and that $H$ is \textit{addition-normal} if there exists a graph $G$ that is $H$-addition-saturated.

These notions appear to be more tractable than symmetric normality. For example, it is known that all even cycles are both addition-normal \cite{Tennenhouse} and deletion-normal \cite{evencycle} (although the proof of deletion-normality is rather involved). The case of addition was proved in 2016 by Tennenhouse \cite{Tennenhouse}:

\begin{theorem}[Tennenhouse \cite{Tennenhouse}]\label{thm:Tennenhouse}
   All even cycles are addition-normal.
\end{theorem}

Recently, we also proved the deletion case \cite{evencycle}:

\begin{theorem}[Fan, Hajebi, Hajebi, Spirkl \cite{evencycle}]\label{thm:maineven}
    All even cycles are deletion-normal.
\end{theorem}

Furthermore, complete graphs on two or more vertices are the only graphs known not to be deletion-normal, and edgeless graphs on two or more vertices are the only graphs known not to be addition-normal. (For instance, in contrast to \Cref{thm:p4}, $P_4$ is readily seen to be both deletion- and addition-normal: the five-cycle is $P_4$-deletion-saturated, and the two-edge matching is $P_4$-addition-saturated.)
\medskip

We think that the converse should also be true:

\begin{conjecture}\label{conj:maindeletion}
Every non-complete graph $H$ is deletion-normal. Equivalently, every non-edgeless graph $H$ is addition-normal.
\end{conjecture}

We prove \Cref{conj:maindeletion} for a variety of graphs $H$, which we believe provides considerable evidence for the conjecture:

\begin{theorem}\label{thm:firstmain}
    The following graphs, if not complete, are deletion-normal.
    \begin{enumerate}[{\rm (a)}]
        \item\label{thm:firstmain_a} Every complete multipartite graph with exactly one largest part.
        \item\label{thm:firstmain_b} Every complete multipartite graph with all parts of cardinality $t$, where we have $t\in \left\{\binom{2r-1}{r-1}+1, \binom{2r}{r}+1\right\}$ for some $r\in \poi$.
        \item\label{thm:firstmain_c} Every graph obtained by iteratively adding vertices of degree at most one to a $2$-connected deletion-normal graph.
        \item\label{thm:firstmain_d} Every graph with two degree-one vertices at distance at most three from each other.
        \item\label{thm:firstmain_e} Line graphs of all trees.
    \end{enumerate}
\end{theorem}

In fact, we prove stronger results than outcomes \ref{thm:firstmain_c}, \ref{thm:firstmain_d}, and \ref{thm:firstmain_e} of \Cref{thm:firstmain}. Moreover, we verify our conjecture for all graph on at most six vertices:

\begin{theorem}\label{thm:six}
    Every non-complete graph on at most six vertices is deletion-normal.
\end{theorem}


The remainder of the paper is organized as follows. In \Cref{sec:prelim}, we introduce some notation and terminology used throughout. In \Cref{sec:multi}, we prove \ref{thm:firstmain}~\ref{thm:firstmain_a} and \ref{thm:firstmain_b}. In \Cref{sec:block}, we prove \ref{thm:firstmain}~\ref{thm:firstmain_c}; in fact, the main result therein (see \Cref{thm:maincentered}) is fairly stronger than \ref{thm:firstmain}~\ref{thm:firstmain_c}. In \Cref{sec:severed}, we prove \ref{thm:firstmain}~\ref{thm:firstmain_d} and \ref{thm:firstmain_e}; again, the main result of \Cref{sec:severed} (see \Cref{thm:mainsevered}) substantially strengthens both. The proof \Cref{thm:six} is long and technical, so we include it in Appendix~\ref{appendix}. (The proof also appears in the second author's thesis \cite{sahabthesis}.)

\section{Preliminaries}\label{sec:prelim}

\subsection{Graphs}
For standard graph-theoretic terminology, the reader is referred to \cite{diestel}. Graphs in this paper have finite vertex sets, no loops, and no parallel edges (unless specified otherwise). Let $G$ be a graph with vertex set $V(G)$ and edge set $E(G)$. For $X\subseteq V(G)$, we use $X$ to denote both the set $X$ of vertices and the induced subgraph $G[X]$ of $G$ with vertex set $X$, also called \textit{the subgraph of $G$ induced by $X$}. For $v\in V(G)$, we write $N_G(v)$ for the set of all neighbors of $v$ in $G$. A \textit{leaf} of $G$ is a degree-one vertex. We write $\ol{G}$ for the complement of $G$. For $e\in E(G)$, we denote by $G-e$ the graph obtained from $G$ by removing the edge $e$. For $e\in E(\ol{G})$, we denote by $G+e$ the graph obtained from $G$ by adding the edge $e$.

We say that two subsets $X,Y$ of $V(G)$ are \textit{complete in $G$} (or that \textit{$X$ is complete to $Y$ in $G$}) if $X\cap Y=\varnothing$ and every vertex in $X$ is adjacent to every vertex in $Y$. We also say that $X,Y$ are \textit{anticomplete in $G$} (or \textit{$X$ is anticomplete to $Y$ in $G$}) if $X\cap Y=\varnothing$ and every vertex in $X$ is nonadjacent to every vertex in $Y$. We say that a vertex \textit{$x\in V(G)$ is complete} (\textit{anticomplete}) \textit{to $Y$ in $G$} if ${x}$ is complete (anticomplete) to $Y$ in $G$.

A \textit{clique} in $G$ is a set of pairwise adjacent vertices, and a \textit{stable set} is a set of pairwise nonadjacent vertices. The \textit{chromatic number} of $G$, denoted $\chi(G)$, is the minimum number of stable sets into which $V(G)$ can be partitioned. 

The set of positive integers is denoted by $\poi$, and for integers $k,k'$, we write $\{k,\ldots, k'\}$ for the set of all integers at least $k$ and at most $k'$ (so $\{k,\ldots, k'\}\neq\varnothing$ if and only if $k\leq k'$). For $k\in \poi$ and a graph $H$, we denote by $kH$ the graph obtained from the disjoint union of $k$ isomorphic copies of $H$. For graphs $H_1,\ldots, H_k$, we denote by $H_1+\cdots+H_k$ the graph obtained from the disjoint union of isomorphic copies of $H_1,\ldots, H_k$.

\subsection{Special graphs}

We say that a graph $G$ is \textit{complete} if $E(G)$ is the set of all $2$-subsets of $V(G)$, and that $G$ is \textit{empty} if $E(G)=\varnothing$. For every $n\in \poi$, we denote the $n$-vertex complete graph by $K_n$.

For $m,n\in \poi\cup \{0\}$, we denote by $K_{m,n}$ the complete bipartite graph with parts of cardinality $m$ and $n$. The complete bipartite graph $K_{1,n}$, for $n\in \poi\cup \{0\}$, is also called a \textit{star}. In general, for $k\in \poi$ with $k\geq 2$ and $n_1,\ldots, n_k\in \poi\cup \{0\}$, we denote by $K_{n_1,\ldots, n_k}$ the complete $k$-partite graph with parts of cardinality $n_1,\ldots, n_k$.

Let $P$ be a path. We write $P=p_1\dd \cdots \dd p_k$ to indicate that $V(P)=\{p_1,\ldots, p_k\}$, for some $k\in \poi$, and that $E(P)=\{p_ip_{i+1}:i\in {1,\ldots, k-1}\}$. We call $p_1$ and $p_k$ the \textit{ends} of $P$. The \textit{length} of $P$ is the number of edges in $P$. For every $k\in \poi$, we denote the unique $k$-vertex path (that is, the path of length $k-1$) up to isomorphism by $P_k$. Given a graph $G$, a \textit{path in $G$} is an induced subgraph of $G$ that is a path.

Let $C$ be a cycle. We write $C=c_1\dd \cdots \dd c_k\dd c_1$ to mean that $V(C)=\{c_1,\ldots, c_k\}$, for some $k\in \poi$, and that $E(C)=\{c_ic_{i+1}:i\in {1,\ldots, k-1}\}\cup \{c_kc_1\}$. The \textit{length} of $C$ is the number of edges (or vertices) in $C$, and a cycle of length $k$ is also called a \textit{$k$-cycle}. For  $k\in \poi$, we denote the $k$-cycle up to isomorphism by $C_k$. The \textit{girth} of a graph $G$ is the minimum length of a cycle that is an (induced) subgraph of $G$.

\subsection{Line graphs}

Given a graph $H$, the \textit{line graph of $H$}, denoted $\lng(H)$, is the graph with vertex set $E(H)$ such that distinct edges $e,e'\in E(H)$ of $H$ are adjacent as vertices in $\lng(H)$ if and only if $e$ and $e'$ share an end in $H$. We say that a graph $G$ is a \textit{line graph} if there exists a graph $H$ such that $G=\lng(H)$.

\subsection{$k$-graphs}

Let $k\in \poi$. A \textit{$k$-graph} (also called a \textit{$k$-uniform hypergraph}) is a pair $\Gamma=(V(\Gamma),E(\Gamma))$ where $V(\Gamma)$ is a finite set of \textit{vertices} and $E(\Gamma)$ is a set of $k$-subsets of $V(\Gamma)$, called \textit{hyperedges}. For $n\in \poi$, we denote by $K^k_n$ the $n$-vertex \textit{complete} $k$-graph; that is, the unique $k$-graph (up to isomorphism) in which every $k$-subset of the vertex set is a hyperedge. Note that, with this terminology, graphs are exactly the $2$-graphs, and $K_n=K^2_n$ for every $n\in \poi$. Let $\Gamma$ be a $k$-graph for some $k\geq 2$. The \textit{chromatic number} of $\Gamma$ is the smallest $c\in \poi$ for which there is a partition $(S_1,\ldots, S_c)$ of $V(\Gamma)$ such that no hyperedge of $\Gamma$ is contained entirely in one of the sets $S_1,\ldots, S_c$. The \textit{girth} of $\Gamma$ is the smallest $g\in \poi$ for which there are $g$ pairwise distinct hyperedges of $\Gamma$ whose union has cardinality at most $g(k-1)$. For $v\in V(\Gamma)$, the \textit{degree of $v$ in $\Gamma$} is the number of hyperedges of $\Gamma$ containing $v$. The \textit{line graph} of a $k$-graph $\Gamma$, denoted $\lng(\Gamma)$, is the graph (or $2$-graph) with vertex set $E(\Gamma)$ such that distinct hyperedges $f_1,f_2\in E(\Gamma)$ are adjacent as vertices in $\lng(\Gamma)$ if and only if $f_1\cap f_2\neq\varnothing$. Note that $\ol{\lng(K_n^k)}$ for $n\geq 2k$ are exactly the \textit{Kneser graphs} \cite{kneser}. Observe that in the case $k=2$, the above notions match the analogous ones for graphs. Also, if the girth of $\Gamma$ is larger than $2$, then $\Gamma$ is \textit{linear}; that is, every two hyperedges in $\Gamma$ share at most one vertex.

\section{Complete multipartite graphs}\label{sec:multi}
Here, we prove outcomes \ref{thm:firstmain_a} and \ref{thm:firstmain_b} of \Cref{thm:firstmain}. We need the following from \cite{polychi}:

\begin{theorem}[Bria\'{n}ski, Davies, Walczak \cite{polychi}]\label{thm:AlexcF}
For all integers $c\geq t\geq 2$, there is a $K_{t+1}$-free graph $G$ with $\chi(G)=c$ such that $\chi(G')<27t^3$ for every $K_t$-free induced subgraph $G'$ of $G$.
\end{theorem}

We now give a proof of \Cref{thm:firstmain}\ref{thm:firstmain_a}, which we restate as follows:

\begin{theorem}\label{thm:compmulti}
Let $H$ be a complete multipartite graph with $|V(H)|\geq 2$ such that $H$ has a unique largest part. Then $H$ is deletion-normal.
\end{theorem}

\begin{proof}
We show that $\ol{H}$ is addition-normal. Let $s\in \poi\cup \{0\}$ be such that $H$ has $s+1$ parts. Then $\ol{H}$ has $s+1$ components, each of which is a clique. In particular, if $\ol{H}$ is connected, then $\ol{H}$ is a complete graph on two or more vertices, which is addition-normal. Therefore, since $H$ has exactly one largest part, we may assume that for some $s,t\in \poi$, $\ol{H}$ has exactly $s+1$ components, one of which is isomorphic to $K_{t+1}$, and each of the remaining $s$ components is a complete graph on $t$ or fewer vertices. In other words,

\sta{\label{st:strongercompbip}$K_{t+1}$ is isomorphic to an induced subgraph of $\ol{H}$, and $\ol{H}$ is isomorphic to an induced subgraph of $\ol{K_{t+1, \underbrace{t,\ldots, t}_{s}}}$.}

By \eqref{st:strongercompbip}, it suffices to show that there exists a $K_{t+1}$-free graph $G$ such that for every $xy\in E(\ol{G})$, $G+xy$ has an induced subgraph isomorphic to $\ol{K_{t+1, \underbrace{t,\ldots, t}_{s}}}$.

For $t=1$, the edgeless graph on $s+2$ vertices has this property. Assume that $t\geq 2$. Let
$c=27(st+2)t^3$. Then $c\geq t\geq 2$, and we can apply Theorem~\ref{thm:AlexcF} to $c$ and $t$. It follows that there is a $K_{t+1}$-free graph $G^-$ with $\chi(G^-)=27t^3(st+2)$ such that for every $K_t$-free induced subgraph $G'$ of $G^-$, we have $\chi(G')<27t^3$. Let $G$ be a graph containing $G^-$ as a spanning subgraph such that $G$ is $K_{t+1}$-free, and subject to this property, $E(G)$ is maximal with respect to inclusion. Our goal is to show that $G$ is the desired graph.
\medskip

Let $\gamma\in \poi\cup \{0\}$ and let $X\subseteq V(G)$. By an \textit{$X$-chain of length $\gamma$ in $G$} we mean a $(\gamma+1)$-tuple $(X_0,\ldots,X_{\gamma})$ of subsets of $V(G)$ such that:
\begin{itemize}
\item $X_0=X$;
\item $X_i$ is a $t$-vertex clique in $G$ for every $i\in \{1,\ldots, \gamma\}$; and
\item $X_0,\ldots, X_{\gamma}$ are pairwise anticomplete in $G$.
\end{itemize}
It follows that there is always an $X$-chain of length zero in $G$. For every $X\subseteq V(G)$, let $\gamma_X\in \poi \cup \{0\}$ be the maximum length of an $X$-chain in $G$. We claim that:

\sta{\label{st:Xsmallchi} Let $X\subseteq V(G)$ be such that $G[X]$ has no isolated vertex. Then we have $$\chi(G^-)<27t^3(|X|+t\gamma_X+1).$$}

Suppose not. Let $\gamma=\gamma_X$. Let $(X_0,\ldots,X_{\gamma})$ be an $X$-chain of length $\gamma$ in $G$. Let $U=X_0\cup \cdots \cup X_{\gamma}$, let $N=\bigcup_{v\in U}N_G(v)$, and let $M=V(G)\setminus N$. Note that for each $i\in \{0,1,\ldots, \gamma\}$, the graph $G[X_i]$ has no isolated vertex (for $i=0$, this follows from the assumption of \eqref{st:Xsmallchi}, and for $i\in \{1,\ldots, \gamma\}$, this follows from the fact that $X_i$ is a $t$-vertex clique and $t\geq 2$). We deduce that in the graph $G$, every vertex in $U$ has a neighbor in $U$; in particular, $U\subseteq N$, and $U$ and $M$ are anticomplete in $G$. On the other hand, since $G$ is $K_{t+1}$-free, it follows that for every $v\in U$, the graph $G[N_G(v)]$ is $K_t$-free, and so $G^-[N_G(v)]$ is $K_t$-free as well (because $G^-$ is a subgraph of $G$). Thus, by the choice of $G^-$, we have $\chi(G^-[N_G(v)])<27t^3$ for every $v\in U$. Since $|U|=|X|+t\gamma$, it follows that $\chi(G^-[N])<27t^3(|X|+t\gamma)$. Therefore, since $\chi(G^-)\geq 27t^3(|X|+t\gamma_X+1)$ and $(M,N)$ is a partition of $V(G)$, it follows that $\chi(G^-[M])>27t^3$, and so by the choice of $G^-$, there is a $t$-vertex clique $Y$ in $G^-$ such that $Y\subseteq M$. In particular, $Y$ is a $t$-vertex clique in $G$ because $G^-$ is a subgraph of $G$. But now, since $U$ and $M$ are anticomplete in $G$, it follows that $(X_0,\ldots, X_{\gamma},Y)$ is an $X$-chain of length $\gamma+1$, a contradiction. This proves \eqref{st:Xsmallchi}.
\medskip

Let $xy\in E(\ol{G})$ and let $G^+=G+xy$. By the choice of $G$, there is a $(t+1)$-vertex clique $X$ in $G^+$ such that $x,y\in X$ are the only two vertices in $X$ that are non-adjacent in $G$. Since $t\geq 2$, it follows that $G[X]$ has no isolated vertex. Recall also that
$$\chi(G^-)=27t^3(st+2)=27t^3(t+1+(s-1)t+1)=27t^3(|X|+(s-1)t+1).$$
By \eqref{st:Xsmallchi}, there is an $X$-chain $(X_0,\ldots, X_s)$ of length $s$ in $G$, and so $G^+[X_0\cup \cdots \cup X_s]$ is isomorphic to $\ol{K_{t+1, \underbrace{t,\ldots, t}_{s}}}$. This completes the proof of Theorem~\ref{thm:compmulti}.
\end{proof}

The proof of \Cref{thm:firstmain}\ref{thm:firstmain_b} is in three steps. First, we handle the case
$r=1$ and $t=\binom{2r-1}{r-1}+1=2$. Indeed, in this case, we prove something stronger:

\begin{theorem}\label{thm:Ktts0}
For all $s,s'\in \poi\cup \{0\}$, the graph $K_{\underbrace{2,\ldots, 2}_{s+1},\underbrace{1,\ldots, 1}_{s'}}$ is normal.
\end{theorem}

\begin{proof}
It suffices to show that
$$H=\ol{K_{\underbrace{2,\ldots, 2}_{s+1},\underbrace{1,\ldots, 1}_{s'}}}$$
is normal. For $s=0$, the edgeless graph on $s'+2$ vertices is $H$-induced-saturated. So we may assume that $s\geq 1$. Let $G$ be a graph with exactly $s+s'+2$ components such that $G$ has $s$ components $W_1,\ldots, W_s$, each of which is isomorphic to the complement of the icosahedron (see \Cref{fig:icosa}), and the remaining $s'+2$ components of $G$ are isolated vertices; say $z_1,\ldots, z_{s'+2}$. Our goal is to prove that $G$ is $H$-induced-saturated.
\begin{figure}[t!]
    \centering   \includegraphics[scale=0.5]{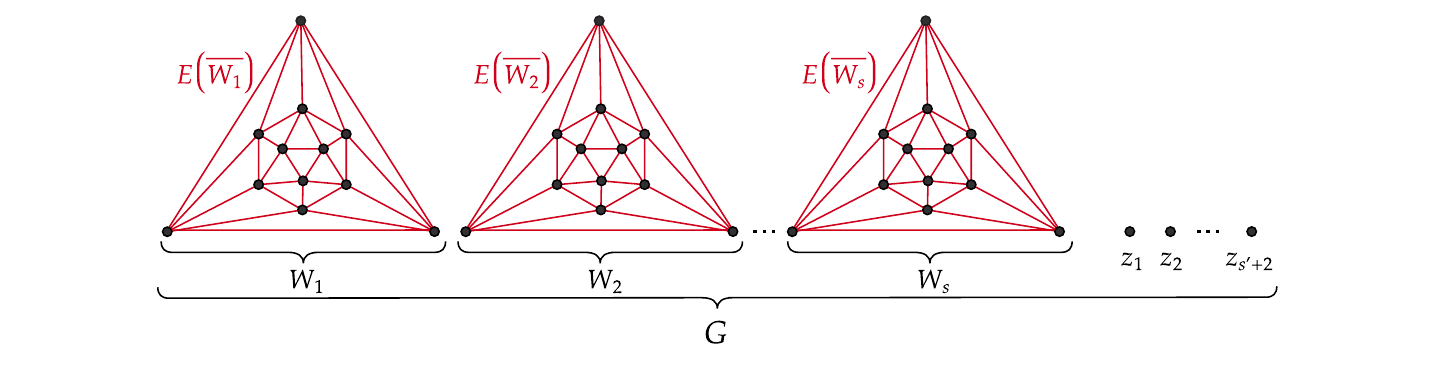}
    \caption{Proof of \Cref{thm:Ktts0}.}
    \label{fig:K_(222111)}
\end{figure}

Recall that the icosahedron is $C_4$-induced-saturated \cite{behrens}. Since $K_{2,2}$ is the $4$-cycle, each of $W_1,\ldots, W_s$ is $\ol{K_{2,2}}$-induced-saturated, and in particular $\ol{K_{2,2}}$-free. Now,
\begin{itemize}
\item For each $i\in \{1,\ldots, s\}$, fix an edge $u_iv_i\in E(W_i)$.
\item For every $w\in V(G)\setminus \{z_1,\ldots, z_{s'+2}\}$, say $w\in V(W_i)$ for some $i\in \{1,\ldots, s\}$, fix an edge $u_wv_w$ of $W_i$ such that $u_ww,v_ww\notin E(G)$.
\end{itemize}
Note that the choice of $u_wv_w$ in the second bullet is possible because the non-neighbors of $w$ in $W_i$ induce a $5$-cycle (which in turn holds because the neighbors of every vertex in the icosahedron induce a $5$-cycle).

First, we show that $G$ is $H$-free. Suppose not. Then, in particular, there are $s+1$ pairwise anticomplete two-vertex cliques $\{a_1,b_1\}, \ldots, \{a_{s+1},b_{s+1}\}$ in $G$. In particular, since $W_1,\ldots, W_s$ are the components of $G$ on two or more vertices, it follows that there exist $i\in \{1,\ldots, s\}$ and $j,j'\in \{1,\ldots, s+1\}$ such that $\{a_j,b_j\},\{a_{j'},b_{j'}\}\subseteq V(W_i)$. But now $W_i[\{a_j,b_j,a_{j'},b_{j'}\}]$ is isomorphic to $\ol{K_{2,2}}$, contrary to the fact that $W_i$ is $\ol{K_{2,2}}$-free.

Next, we show that for every $xy\in E(G)$, the graph $G-xy$ has an induced subgraph isomorphic to $H$. Since $W_1,\ldots, W_s$ are pairwise distinct components of $G$, and the only ones on two or more vertices, it follows that there exists exactly one $i\in \{1,\ldots, s\}$ for which $x,y\in V(W_i)$, and thus $xy\in E(W_i)$. Since $W_i$ is $\ol{K_{2,2}}$-induced-saturated, it follows that $W_i-xy$ has an induced subgraph $M$ isomorphic to $\ol{K_{2,2}}$. Therefore,
$$\left(\bigcup_{j\in \{1,\ldots, s\}\setminus \{i\}}\{u_j,v_j\}\right)\cup V(M)\cup \{z_1,\ldots, z_{s'}\}$$
induces a subgraph of $G-xy$ that is isomorphic to $H$, as desired. 

It remains to prove that for every $xy\in E(\ol{G})$, the graph $G+xy$ has an induced subgraph isomorphic to $H$. First, consider the case where there are $i,j\in \{1,\ldots, s\}$ for which $x\in V(W_i)$ and $y\in V(W_j)$. Assume that $i=j$. Then $xy\in E(\ol{W_i})$. Since $W_i$ is $\ol{K_{2,2}}$-induced-saturated, it follows that $W_i+xy$ has an induced subgraph $M$ isomorphic to $\ol{K_{2,2}}$. Thus, the subgraph of $G+xy$ induced by
$$\left(\bigcup_{k\in \{1,\ldots, s\}\setminus \{i\}}\{u_k,v_k\}\right)\cup V(M)\cup \{z_1,\ldots, z_{s'}\}$$
is isomorphic to $H$, as required. Consequently, we may assume that $i\neq j$. But then the subgraph of $G+xy$ induced by
$$\left(\bigcup_{k\in \{1,\ldots, s\}\setminus \{i,j\}}\{u_k,v_k\}\right)\cup \{u_x,v_x\}\cup \{x,y\}\cup \{u_y,v_y\}\cup \{z_1,\ldots, z_{s'}\}$$
is isomorphic to $H$. Next, consider the case where there are $i\in \{1,\ldots, s\}$ and $j\in \{1,\ldots, s'+2\}$ such that $x\in V(W_i)$ and $y=z_j$. Choose an $s'$-subset $S$ of $\{z_1,\ldots, z_{s'+1}\}\setminus \{y\}$. Then the subgraph of $G+xy$ induced by
$$\left(\bigcup_{k\in \{1,\ldots, s\}\setminus \{i\}}\{u_k,v_k\}\right)\cup \{u_x,v_x\}\cup \{x,y\}\cup S$$
is isomorphic to $H$. Finally, consider the case where there are $i,j\in \{1,\ldots, s'+2\}$ for which $x=z_i$ and $y=z_j$. Then the subgraph of $G+xy$ induced by
$$\left(\bigcup_{k\in \{1,\ldots, s\}}\{u_k,v_k\}\right)\cup \{z_1,\ldots, z_{s'+2}\}$$
is isomorphic to $H$. This completes the proof of \Cref{thm:Ktts0}.
\end{proof}

Second, we handle the case
$r\geq 2$ and $t=\binom{2r-1}{r-1}+1$. The proof will rely on the Erd\H os--Ko--Rado Theorem \cite{EKR}.

\begin{theorem}[Erd\H os, Ko, Rado \cite{EKR}]\label{thm:EKR}
    For all $n,r\in \poi$ with $n\geq 2r$, the maximum cardinality of a clique in $\lng(K_{n}^r)$ is $\binom{n-1}{r-1}$.
\end{theorem}

\begin{theorem}\label{thm:Ktts1}
   Let $r,s\in \poi$ with $r\geq 2$. Let
   $$t=\binom{2r-1}{r-1}+1$$
   and let
   $n=(2r+1)(s+1)-1$.
   Then
   $\ol{\lng(K_{n}^r)}$
   is $K_{\underbrace{t,\ldots,t}_{s+1}}$-deletion-saturated.
\end{theorem}

\begin{proof}
Let
$L=\lng(K_{n}^r)$ and let $$H=\ol{K_{\underbrace{t,\ldots,t}_{s+1}}}.$$
Our goal is to prove that $L$ is $H$-addition-saturated. First, we show that $L$ is $H$-free. Suppose not. Then there are $s+1$ pairwise anticomplete cliques $W_1,\ldots,W_{s+1}$ in $L$, each of cardinality $t$; in particular, $W_1,\ldots,W_{s+1}\subseteq E(K_{n}^r)$. For each $i\in \{1,\ldots,s+1\}$, let $U_i=\bigcup_{f\in W_i}f\subseteq V(K_{n}^r)$. Since $W_1,\ldots,W_{s+1}$ are pairwise anticomplete in $L$, it follows that $U_1,\ldots,U_{s+1}$ are pairwise disjoint. Also, since $|W_1|=\cdots=|W_{s+1}|=t>\binom{2r-1}{r-1}$, it follows from \Cref{thm:EKR} that $|U_i|\geq 2r+1$ for every $i\in \{1,\ldots,s+1\}$. But then
$$|V(K_{n}^r)|\geq |U_1\cup \cdots \cup U_{s+1}|\geq (2r+1)(s+1)=n+1,$$
a contradiction.

Next, we prove that for every $f_1f_2\in E(\ol{L})$, the graph $L+f_1f_2$ is not $H$-free. By choice, $f_1$ and $f_2$ are disjoint hyperedges of $K_{n}^r$. Let $U=f_1\cup f_2$ and let $U'=V(K_{n}^r)\setminus U$. Then $|U|=2r$ and $|U'|=n-2r=(2r+1)s$. Fix a partition $(U'_1,\ldots,U'_s)$ of $U'$ into $s$ subsets with $|U'_1|=\cdots=|U'_s|=2r+1$. Let $u\in f_1\subseteq U$ be fixed, let $C$ be the set of all $r$-subsets of $U$ containing $u$, and let $D=C\cup \{f_2\}$. Then $D\setminus \{f_2\}=C$ is a clique of cardinality $\binom{2r-1}{r-1}$ in $L$, and $f_1\in C$. Since $U$ is the disjoint union of $f_1$ and $f_2$, and since $|f_1|=|f_2|=r$, it follows that, other than $f_1$, every $r$-subset of $U$ containing $u$ intersects $f_2$. In particular, $D\setminus \{f_1\}=(C\setminus \{f_1\})\cup \{f_2\}$ is also a clique of cardinality $\binom{2r-1}{r-1}$ in $L$. Therefore, $D$ is a clique of cardinality $\binom{2r-1}{r-1}+1=t$ in $L+f_1f_2$.

For each $i\in \{1,\ldots,s\}$, fix a vertex $u'_i\in U'_i$. Since $|U'_i|=2r+1$ and $r\geq 2$, the number of all $r$-subsets of $U'_i$ containing $u'_i$ is
$$\binom{2r}{r-1}=\frac{2r}{r+1}\binom{2r-1}{r-1}>\binom{2r-1}{r-1}.$$
It follows that there is a set $D'_i$ of $r$-subsets of $U'_i$ containing $u'_i$ such that $|D'_i|=\binom{2r-1}{r-1}+1=t$. In particular, $D'_i$ is a clique of cardinality $t$ in $L$, and hence in $L+f_1f_2$.

Moreover, for all $(f,f'_1,\ldots,f'_s)\in D\times D'_1\times \cdots \times D'_s$, the hyperedges $f,f'_1,\ldots,f'_s$ of $K_n^r$ are pairwise disjoint because $f\subseteq U$, $f'_i\subseteq U'_i$ for each $i\in \{1,\ldots,s\}$, and $U,U'_1,\ldots,U'_s$ are pairwise disjoint. Consequently, $D,D'_1,\ldots,D'_s$ are pairwise anticomplete in $L$; in particular, $D,D'_1,\ldots,D'_s$ are pairwise disjoint. Since $f_1,f_2\in D$, we deduce that $D,D'_1,\ldots,D'_s$ are pairwise anticomplete in $L+f_1f_2$.

In conclusion, $D,D'_1,\ldots,D'_s$ are $s+1$ pairwise anticomplete cliques in $L+f_1f_2$, each of cardinality $t$. Hence, the subgraph of $L+f_1f_2$ induced by $D\cup D'_1\cup \cdots\cup D'_s$ is isomorphic to $H$. This completes the proof of \Cref{thm:Ktts1}.
\end{proof}

Finally, we handle the case 
$t=\binom{2r}{r}+1$. The proof here will use a ``product'' extension of the Erd\H os-Ko-Rado Theorem due to Frankl \cite{frankel}. In order to state that result, we need a definition. 
For $k\in \poi$ and graphs $G_1,\ldots, G_k$, let $G_1\Join \cdots \Join G_k$ be the graph with vertex set $V(G_1)\times \cdots \times V(G_k)$ such that \textit{distinct} $k$-tuples $(u_1,\ldots, u_k)$ and $(v_1,\ldots, v_k)$ are adjacent in $G_1\Join \cdots \Join G_k$ if and only if for some $i\in \{1,\ldots, k\}$, either $u_i=v_i$ or $u_iv_i\in E(G_i)$.

The ``product'' extension of the Erd\H os-Ko-Rado Theorem is as follows. (For $k=1$, this is identical to the original version of Erd\H os-Ko-Rado, namely \Cref{thm:EKR}.)

\begin{theorem}[Frankl \cite{frankel}]\label{thm:EKRproduct}
    Let $k\in \poi$ and let $n_1,r_1,\ldots, n_k,r_k\in \poi$ such that $n_i\geq r_i$ for all $i\in \{1,\ldots, k\}$. Then the maximum cardinality of a clique in $\lng(K_{n_1}^{r_1})\Join \cdots \Join \lng(K_{n_k}^{r_k})$ is 
$$\max\left\{\dfrac{r_1}{n_1},\ldots,\dfrac{r_k}{n_k}\right\}\times \binom{n_1}{r_1}\times \cdots \times \binom{n_k}{r_k}.$$
\end{theorem}

\begin{theorem}\label{thm:Ktts2}
Let $r,s\in \poi$. Let
$$t=\binom{2r}{r}+1$$
and let
$n=(2r+1)(s+1)-1$.
Then $\ol{\lng\left(K_{n}^r\right)\Join\lng\left(K^1_{s+2}\right)}$ is $K_{\underbrace{t,\ldots, t}_{s+1}}$-deletion-saturated. 
\end{theorem}

\begin{proof}
Recall that $\lng(K^1_{s+2})$ is just an edgeless graph on $s+2$ vertices (that is, isomorphic to $(s+2)K_1$). Let
$L=\lng\left(K_{n}^r\right)\Join\lng\left(K^1_{s+2}\right)$ and let $$H=\ol{K_{\underbrace{t,\ldots, t}_{s+1}}}.$$
Then our goal is to prove that $L$ is $H$-addition-saturated.

\sta{\label{st:productHfree} $L$ is $H$-free.}

Suppose not. Then there are $s+1$ pairwise anticomplete cliques $W_1, \ldots, W_{s+1}$ in $L$, each of cardinality $t$; in particular, 
$W_1, \ldots, W_{s+1}\subseteq E\left(K^r_{n}\right)\times E\left(K^1_{s+2}\right)$.
For each $i\in \{1,\ldots, s+1\}$ and every $j\in \{0,1\}$, let 
$$U_{i,j}=\bigcup_{(f_0,f_1)\in W_i}f_j$$
and let $n_{i,j}=|U_{i,j}|\geq 1$.
In particular, we have $U_{i,0}\subseteq V(K_n^r)$ and $U_{i,1}\subseteq V(K_{s+2}^1)$. Since $W_1, \ldots, W_{s+1}$ are pairwise anticomplete in $L$, it follows that for every $j\in \{0,1\}$, $U_{1,j}, \ldots, U_{s+1,j}$ are pairwise disjoint. In particular, we have:

\begin{itemize}
\item $\displaystyle\sum_{i=1}^{s+1}n_{i,0}\leq n$; and
\item $\displaystyle\sum_{i=1}^{s+1}n_{i,1}\leq s+2$. 
\end{itemize}
 
Since $n <(2r+1)(s+1)$, by the first bullet, there exists $i\in \{1,\ldots, s+1\}$ for which $n_{i,0}\leq 2r$. Also, by the second bullet, we have $n_{i,1}\leq 2$. Since $|W_i|=t$, it follows that
$\lng\left(K_{2r}^r\right)\Join\lng\left(K^1_{2}\right)$ has a clique of cardinality $t$. Thus, by \Cref{thm:EKRproduct},
$$t\leq \max\left\{\dfrac{r}{2r},\dfrac{1}{2}\right\}\times \binom{2r}{r}\times \binom{2}{1}=\binom{2r}{r},$$
which is a contradiction. This proves \eqref{st:productHfree}.

\sta{\label{st:saturatedproduct} For every $xy\in E(\ol{L})$, the graph $L+xy$ is not $H$-free.}

We may write $x=(e_0,e_1)$ and $y=(f_0,f_1)$ such that $e_0$ and $f_0$ are disjoint hyperedges of $K_{n}^r$, and $e_1$ and $f_1$ are disjoint hyperedges of $K_{s+2}^1$. Let $U=e_0\cup f_0$. Then $|U|=2r$. Let $u\in e_0\subseteq U$ be fixed, and let $C$ be the set of all $r$-subsets of $U$ containing $u$; in particular, $e_0\in C$ and $f_0\notin C$. Let
$$\mca{C}=C\times \{e_1,f_1\}$$
and let
$$\mca{D}=\mca{C}\cup \{y\}.$$
Note that $\mca{D}\setminus \{y\}=\mca{C}$ is a clique in $L$ with $x\in \mca{C}$. Since $U$ is the disjoint union of $e_0$ and $f_0$, and since $|e_0|=|f_0|=r$, it follows that, other than $e_0$, every $r$-subset of $U$ intersects $f_0$. Consequently, for every $z=(g_0,g_1)\in \mca{C}\setminus \{x\}$, we have $zy\in E(L)$, because either $g_0\neq e_0$, in which case $g_0\cap f_0\neq \varnothing$, or $g_0=e_0$ and $g_1=f_1$ (since $z\neq x$). In other words, $\mca{D}\setminus \{x\}=(\mca{C}\setminus \{x\})\cup \{y\}$ is also a clique in $L$. We deduce that $\mca{D}$ is a clique in $L+xy$. Moreover,
$$|\mca{D}|=|\mca{C}|+1=2\binom{2r-1}{r-1}+1=\binom{2r}{r}+1=t.$$

Let $U'=V(K_{n}^r)\setminus U$. Then $|U'|=n-2r=(2r+1)s$. Fix a partition $(U'_1,\ldots, U'_s)$ of $U'$ into $s$ subsets with $|U'_1|=\cdots =|U'_s|=2r+1$. Also, let $E(K_{s+2}^1)\setminus \{e_1,f_1\}=\{h_{1},\ldots, h_{s}\}$. For each $i\in \{1,\ldots, s\}$, since $|U'_i|=2r+1$, the number of all $r$-subsets of $U'_i$ is 
$$\binom{2r+1}{r}=\dfrac{2r+1}{r+1}\binom{2r}{r}>\binom{2r}{r}.$$
So we can choose a set $D'_i$ of $r$-subsets of $U'_i$ such that $|D'_i|=\binom{2r}{r}+1=t$. In particular, 
$$\mca{D}'_i=D'_i\times \{h_i\}$$
is a clique of cardinality $t$ in $L$, and so in $L+xy$.

Moreover, for all $((g,h),(g_1,h_1),\ldots, (g_s,h_s))\in \mca{D}\times \mca{D}'_1\times \cdots \times \mca{D}'_s$, 
\begin{itemize}
    \item the hyperedges $g,g_1,\ldots, g_s$ of $K^r_n$ are pairwise disjoint because $g\subseteq U$, $g_i\subseteq U'_i$ for each $i\in \{1,\ldots, s\}$, and $U,U'_1,\ldots, U'_s$ are pairwise disjoint; and
    \item $h,h_1,\ldots, h_s$ are pairwise distinct hyperedges of $K^1_{s+2}$ because $h\in \{e_1,f_1\}=E(K^1_{s+2})\setminus \{h_1,\ldots, h_s\}$, and thus $h,h_1,\ldots, h_s$ are pairwise disjoint.
\end{itemize}
Consequently, $\mca{D},\mca{D}'_1,\ldots, \mca{D}'_s$ are pairwise anticomplete in $L$; in particular, $\mca{D},\mca{D}'_1,\ldots, \mca{D}'_s$ are pairwise disjoint. Since $x,y\in \mca{D}$, we deduce that $\mca{D},\mca{D}'_1,\ldots, \mca{D}'_s$ are pairwise anticomplete in $L+xy$. 

In conclusion, we have shown that $$\mca{D},\mca{D}'_1,\ldots, \mca{D}'_s$$ are $s+1$ pairwise anticomplete cliques in $L+xy$, each of cardinality $t$. Hence, the subgraph of $L+xy$ induced by $\mca{D}\cup \mca{D}'_1\cup \cdots\cup \mca{D}'_s$ is isomorphic to $H$. This proves \eqref{st:saturatedproduct}.
\medskip

The result now follows from \eqref{st:productHfree} and \eqref{st:saturatedproduct}. This completes the proof of \Cref{thm:Ktts2}.
\end{proof}

\section{Governing blocks}\label{sec:block}

Here, we prove outcome \ref{thm:firstmain_c} of \Cref{thm:firstmain}. In fact, we prove a substantially stronger statement, which requires several definitions. Let $G$ be a graph. Recall that a \textit{block} in $G$ is a maximal $2$-connected induced subgraph of $G$ (we consider complete graphs on one or two vertices to be $2$-connected). We denote by $\mca{B}(G)$ the set of all blocks of $G$. We say that a graph $H$ is \textit{omnipresent in $G$} if for every $u\in V(H)$ and every $v\in V(G)$, there is an injection $\psi:V(H)\rightarrow V(G)$ such that $\psi(u)=v$ and $\psi$ is an isomorphism between $H$ and $G[\psi(V(H))]$. For instance, $K_1$ is omnipresent in $G$ if and only if $G$ is non-null, and $K_2$ is omnipresent in $G$ if and only if $G$ has no isolated vertices. Given a graph $H$, we say that a block $B$ of $H$ is \textit{governing} if there is a $B$-deletion-saturated graph $W$ such that every block $B'\in \mca{B}(H)\setminus \{B\}$ is omnipresent in $W$.  

We will prove the following strengthening of \Cref{thm:firstmain}\ref{thm:firstmain_c}:

\begin{theorem}\label{thm:maincentered}
Every graph with a governing block is deletion-normal.
\end{theorem}
\begin{figure}[t!]
    \centering   \includegraphics[scale=0.55]{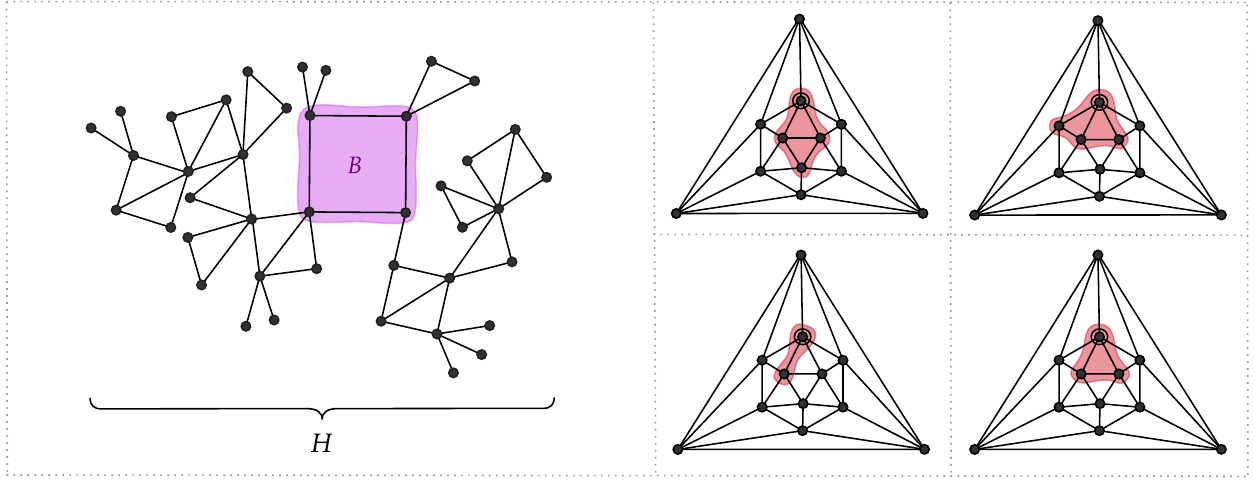}
    \caption{Left: A graph $H$ with a governing block $B$. Right: All blocks of $H$ except $B$ are omnipresent in the icosahedron.}
    \label{fig:governingblock}
\end{figure}

For example, the graph $H$ in \Cref{fig:governingblock} has a block $B$ which is a $4$-cycle; the other blocks of $H$ are all isomorphic to $K_2$, $K_3$, or the \textit{diamond} (the unique four-vertex graph with five edges). Therefore, since the icosahedron is $C_4$-induced-saturated (and thus $C_4$-deletion-saturated), and since $K_2$, $K_3$, and the diamond are all omnipresent in the icosahedron, it follows that the graph $H$ in \Cref{fig:governingblock} is deletion-normal.

Note that \Cref{thm:maincentered} implies \Cref{thm:firstmain}\ref{thm:firstmain_d}:

\begin{corollary}\label{cor:blockplusleaves}
Every graph obtained by iteratively adding vertices of degree at most one to a $2$-connected deletion-normal graph is deletion-normal.
\end{corollary}

\begin{proof}
    Let $H$ be a graph obtained by iteratively adding vertices of degree at most one to a $2$-connected deletion-normal graph $B$. Then $B$ is deletion-normal and every block in $\mca{B}(H)\setminus \{B\}$ is isomorphic to $K_1$ or $K_2$. In particular, since $B$ is connected, we may assume that there is a $B$-deletion-saturated graph $W$ which is also connected, and so both $K_1$ and $K_2$ are omnipresent in $W$. It follows that $B$ is a governing block of $H$, which along with \Cref{thm:maincentered} implies that $H$ is deletion-normal, as desired.
\end{proof}

In particular, every graph with exactly one cycle is obtained from that cycle by iteratively adding vertices of degree at most one. Recall also that, by \Cref{thm:odd,thm:maineven}, all cycles on four or more vertices are deletion-normal. So, from \Cref{cor:blockplusleaves}, it follows that every graph with exactly one cycle, where that cycle is not a triangle, is deletion-normal.

\begin{corollary}\label{cor:unicyclic}
Every triangle-free graph with exactly one cycle is deletion-normal.
\end{corollary}

Another consequence of \Cref{cor:blockplusleaves} is that for every $2$-connected deletion-normal graph $B$, the disjoint union of $B$ and any forest is deletion-normal. This is somewhat curious because deletion-normality remains open for most forests.

\begin{corollary}\label{cor:forestplus}
The disjoint union of a forest and a $2$-connected deletion-normal graph is deletion-normal.
\end{corollary}

The proof of \Cref{thm:maincentered} uses the existence of $k$-graphs with arbitrarily large girth and chromatic number, a result first proved by Erd\H{o}s and Hajnal \cite{erdosgirth}. Several other constructions were subsequently found; see, for example, \cite{alongirth} and \cite{lovaszgirth}.

\begin{theorem}[Erd\H{o}s and Hajnal, Corollary 13.4 in \cite{erdosgirth}]\label{thm:hypergirthchi}
For all $c,g,k\in \poi$, there is a $k$-graph of girth larger than $g$ and chromatic number larger than $c$.
\end{theorem}

From \Cref{thm:hypergirthchi}, it follows that there are $k$-graphs of arbitrarily large girth and minimum degree. (The proof is easy, but we include it for completeness.)

\begin{lemma}\label{lem:hypgirthdeg}
    For all $d,g,k\in \poi$, there is a $k$-graph of girth larger than $g$ in which every vertex has degree larger than $d$.
\end{lemma}
\begin{proof}
 By \Cref{thm:hypergirthchi}, there exists a $k$-graph $\Gamma$ with girth larger than $g$ and chromatic number larger than $d+1$. Let $\Gamma'$ be a $k$-graph with $V(\Gamma')\subseteq V(\Gamma)$ and $E(\Gamma')\subseteq E(\Gamma)$ such that $\Gamma'$ has chromatic number larger than $d+1$, and subject to this property, $V(\Gamma')\cup E(\Gamma')$ is minimal with respect to inclusion. 
 
 We claim that $\Gamma'$ is the $k$-graph desired in \Cref{lem:hypgirthdeg}.  Note that $\Gamma'$ is $k$-uniform because $\Gamma$ is. Also, every cycle in $\Gamma'$ is also a cycle in $\Gamma$, which in turn implies that $\Gamma'$ has girth larger than $g$ because $\Gamma$ does. It remains to show that every vertex in $\Gamma'$ has degree larger than $d$. Suppose for a contradiction that some vertex $v\in V(\Gamma')$ has degree at most $d$. Let $\mca{E}$ be the set of all hyperedges of $\Gamma'$ that contain $v$; then $|\mca{E}|\leq d$. Let $\Gamma''$ be the $k$-graph with $V(\Gamma'')=V(\Gamma')\setminus \{v\}$ and $E(\Gamma'')=E(\Gamma')\setminus \mca{E}$. Then $\Gamma''$ is a $k$-graph with $V(\Gamma'')\cup E(\Gamma'')$ strictly contained in $V(\Gamma')\cup E(\Gamma')$. From this and the choice of $\Gamma'$, it follows that $\Gamma''$ has chromatic number at most $d+1$. This means there is a partition $(S_1,\ldots,S_{d+1})$ of $V(\Gamma'')$ such that no hyperedge of $\Gamma''$ is entirely contained in one of $S_1,\ldots,S_{d+1}$. Moreover, since $|\mca{E}|\leq d$, it follows that for some $j\in \{1,\ldots,d+1\}$, none of the sets $(f\setminus \{v\}:f\in \mca{E})$ is entirely contained in $S_j$. Let $S'_j=S_j\cup \{v\}$ and let $S'_i=S_i$ for each $i\in \{1,\ldots,d+1\}\setminus \{j\}$. Then no hyperedge in $E(\Gamma'')\cup \mca{E}=E(\Gamma')$ is entirely contained in one of $S'_1,\ldots,S'_{d+1}$. But now $\Gamma'$ has chromatic number at most $d+1$, a contradiction. This completes the proof of \Cref{lem:hypgirthdeg}.
\end{proof}

We also need a somewhat more ``tangible'' certificate of large girth in $k$-graphs. Let $k\geq 2$, let $\Gamma$ be a $k$-graph, and let $m\geq 2$ be an integer. By a \textit{ring in $\Gamma$} we mean an ordered $(m+1)$-tuple $R=(f_1,\ldots,f_{m+1})$ where:
\begin{itemize}
    \item $f_1,\ldots,f_m,f_{m+1}$ are hyperedges of $\Gamma$ such that $f_1=f_{m+1}$ but there are at least two distinct hyperedges among $f_1,\ldots,f_m$; and
    \item there are $m$ pairwise distinct vertices $v_1,\ldots,v_m$ of $\Gamma$ such that for every $i\in \{1,\ldots,m\}$, we have $v_i\in f_i\cap f_{i+1}$.
\end{itemize}
The \textit{length} of $R$ is the number $m$ (which is at least two).

\begin{lemma}\label{lem:ring}
Let $g,k\in \poi$ and let $\Gamma$ be a $k$-graph. Assume that there is a ring of length at most $g$ in $\Gamma$. Then $\Gamma$ has girth at most $g$.
\end{lemma}
\begin{proof}
Let $(f_1,\ldots,f_{m+1})$ be a ring of length $m$ in $\Gamma$ with $m$ as small as possible. Then we have $2\leq m\leq g$. By definition, there are $m$ pairwise distinct vertices $v_1,\ldots,v_m$ of $\Gamma$ such that for every $i\in \{1,\ldots,m\}$, we have $v_i\in f_i\cap f_{i+1}$. Since every hyperedge of $\Gamma$ has cardinality exactly $k$, it follows that:

   \sta{\label{st:unionsmall} We have $\displaystyle\left|\bigcup_{i=1}^m f_i\right|\leq m(k-1)$.}

    Moreover, we claim that:

     \sta{\label{st:distincthyp} The hyperedges $f_1,\ldots,f_m$ are pairwise distinct.}

     Suppose not. Choose $i,\mu\in \{1,\ldots,m-1\}$ with $\mu$ as small as possible such that $f_i=f_{i+\mu}$. Since there are at least two distinct hyperedges among $f_1,\ldots,f_m$, it follows that $m\geq 3$. If $\mu=1$, then $(f_1,\ldots,f_i,f_{i+2},\ldots,f_{m+1})$ is a ring of length $m-1$ in $\Gamma$, a contradiction to the minimality of $m$. So we may assume that $2\leq \mu\leq m-1$. Note that, by the minimality of $\mu$, there are at least two distinct hyperedges among $f_i,\ldots,f_{i+\mu-1}$. But now $(f_i,\ldots,f_{i+\mu})$ is a ring in $\Gamma$ of length $\mu\leq m-1$, again a contradiction to the minimality of $m$. This proves \eqref{st:distincthyp}.
     \medskip

     From \eqref{st:unionsmall} and \eqref{st:distincthyp}, we deduce that $f_1,\ldots,f_m$ are $m$ pairwise distinct hyperedges of $\Gamma$ whose union contains at most $m(k-1)$ vertices. Hence, $\Gamma$ has girth at most $m\leq g$. This completes the proof of \Cref{lem:ring}.
\end{proof}

We use the above tools to prove the following lemma, which is the key step in the proof of \Cref{thm:maincentered}.

\begin{lemma}\label{lem:maincentered}
For every graph $H$ with a governing block $B$, there is a graph $G$ such that:
\begin{enumerate}[{\rm (a)}]
    \item $G$ is $B$-free; and
    \item for every $e\in E(G)$ and every connected induced subgraph $H'$ of $H$ with $V(B)\subseteq V(H')$, $G-e$ has an induced subgraph isomorphic to $H'$.
\end{enumerate}
\end{lemma}
\begin{proof}
By definition, since $B$ is governing, there is a $B$-deletion-saturated graph $W$ such that every block $B'\in \mca{B}(H)\setminus \{B\}$ is omnipresent in $W$. Let $|V(W)|=k\in \poi$ and let $h=|V(H)|$. By \Cref{lem:hypgirthdeg}, there exists a $k$-graph $\Gamma$ of girth larger than $h$ in which every vertex has degree larger than $h$. Also, since $B$ is $2$-connected and non-complete (because $B$ is deletion-normal), it follows that $h=|V(H)|\geq |V(B)|>2$, and so $\Gamma$ has girth larger than $2$. In particular,

\sta{\label{st:atmostone} For all distinct $f,f'\in E(\Gamma)$, we have $|f\cap f'|\leq 1$.}

For every hyperedge $f\in E(\Gamma)$, let $W_f$ be a graph with $V(W_f)=f$ that is isomorphic to $W$, and let $\eta_f:V(W_f)\rightarrow V(W)$ be an isomorphism between $W_f$ and $W$. Let
$$G=\bigcup_{f\in E(\Gamma)}W_f.$$

From \eqref{st:atmostone} and the construction of $G$, it follows that there is a map $\varphi:E(G)\rightarrow E(\Gamma)$ such that for every $e\in E(G)$, $\varphi(e)$ is the only hyperedge of $\Gamma$ that contains both ends of $e$. In particular, $\varphi(e)$ is the only hyperedge of $\Gamma$ for which we have $e\in E(W_{\varphi(e)})$. Moreover, we claim that:

\sta{\label{st:shortcycle} Let $f\in E(\Gamma)$ and let $u,v\in f$ be distinct. Let $P$ be a path of length at most $h$ from $u$ to $v$ which is a subgraph of $G$. Then for every $e\in E(P)$, we have $\varphi(e)=f$.}

Let $P=v_1\dd\cdots\dd v_m$ where $v_1=u$, $v_m=v$, and $2\leq m\leq h$. For each $i\in \{1,\ldots,m-1\}$, let $f_i=\varphi(v_iv_{i+1})$, and let $f_m=f$. We need to show that $f_1=\cdots=f_{m-1}=f_m$. Suppose not. Define $f_{m+1}=\varphi(v_1v_2)$. Then $f_1,\ldots,f_m,f_{m+1}$ are hyperedges of $\Gamma$ such that $f_1=f_{m+1}$ but there are at least two distinct hyperedges among $f_1,\ldots,f_m$, and $v_1,\ldots,v_m$ are $m$ pairwise distinct vertices of $\Gamma$ such that for every $i\in \{1,\ldots,m\}$, we have $v_i\in f_i\cap f_{i+1}$. Therefore, $(f_1,\ldots,f_{m+1})$ is a ring of length $m\leq h$ in $\Gamma$. But now by \Cref{lem:ring}, $\Gamma$ has girth at most $h$, a contradiction. This proves \eqref{st:shortcycle}.
\medskip

Our goal is now to show that $G$ satisfies \Cref{lem:maincentered}. First, we prove that:

\sta{\label{st:wfree}$G$ is $B$-free.}

Suppose not. Let $B_0$ be an induced subgraph of $G$ isomorphic to $B$. Since $B_0$ is $2$-connected, it follows that $E(B_0)\neq \varnothing$. Choose an edge $e\in E(B_0)$ and let $f=\varphi(e)$. Let $M$ be the component of $B_0[V(B_0)\cap f]$ that contains both ends of $e$. Since $W_f$ is isomorphic to $W$, since $B_0$ is isomorphic to $B$, and since $W$ is $B$-free, it follows that $V(B_0)\not\subseteq V(W_f)=f$. Consequently, since $B_0$ is connected, it follows that there is an edge $vw\in E(B_0)$ such that $v\in V(M)$ and $w\in V(G)\setminus f$. Let $f'=\varphi(vw)$. Then $f'\neq f$, and by \eqref{st:atmostone}, we have $f\cap f'=\{v\}$. Since $M$ is connected with at least two vertices, it follows that $v$ has a neighbor $u$ in $M$; in particular, we have $u,v\in f$ distinct. On the other hand, since $B_0$ is $2$-connected, it follows that there is a path $Q$ in $B_0\setminus \{v\}$ from $u$ to $w$. But now $P=u\dd Q\dd w\dd v$ is a path of length at most $|V(B_0)|=|V(B)|\leq h$ between distinct vertices $u,v\in f$ which is a subgraph of $B_0$ (and so of $G$), and $vw$ is an edge of $P$ such that $\varphi(vw)=f'\neq f$, a contradiction to \eqref{st:shortcycle}. This proves \eqref{st:wfree}.
\medskip

In view of \eqref{st:wfree}, it remains to show that for every $e\in E(G)$ and every connected induced subgraph $H'$ of $H$ with $V(B)\subseteq V(H')$, $G-e$ has an induced subgraph isomorphic to $H'$. The proof is by induction on $|V(H')|$ for fixed $e=xy\in E(G)$. Note that $B$ is a block of $H'$. In fact, since every block of $H'$ is an induced subgraph of a block of $H$, it follows that $B$ is a governing block of $H'$.

Let $f_0=\varphi(e)$ and let $G^- = G-e$. Then $x,y\in f_0$, $e\in E(G[f_0])=E(W_{f_0})$, and
$$G^-[f_0]=G[f_0]-e=W_{f_0}-e.$$
Assume that $|V(H')|=|V(B)|$. Then $H'=B$. Since $\eta_{f_0}$ is an isomorphism between $W_{f_0}$ and $W$, it follows that $\eta_{f_0}(x)\eta_{f_0}(y)\in E(W)$ and $W_{f_0}-e$ is isomorphic to $W-\eta_{f_0}(x)\eta_{f_0}(y)$. On the other hand, since $W$ is $B$-deletion-saturated, it follows that $W-\eta_{f_0}(x)\eta_{f_0}(y)$ has an induced subgraph isomorphic to $B$, and so isomorphic to $H'$. Therefore, $G^-[f_0]$ has an induced subgraph isomorphic to $H'$, as desired.

Now assume that $|V(H')|>|V(B)|$. Since $H'$ is connected and $B\in \mca{B}(H')$, it follows that $\mca{B}(H')\setminus \{B\}\neq \varnothing$, and in particular, the block tree of $H'$ has at least two leaves. This means that there is a block $B'\in \mca{B}(H')\setminus \{B\}$ such that $V(B')$ contains exactly one cut vertex of $H'$. Let $u\in V(B')$ be the unique vertex of $B'$ that is a cut vertex of $H'$. Let 
$$H''=H'\setminus (V(B')\setminus \{u\}).$$ Then $H''$ is connected and $|V(H'')|<|V(H')|$. Also, since $\mca{B}(H'')=\mca{B}(H')\setminus \{B'\}$, it follows that $B$ is a block of $H''$ as well. So by the inductive hypothesis, $G^-$ has an induced subgraph $J$ isomorphic to $H''$. Let $\theta:V(H'')\rightarrow V(J)$ be an isomorphism between $H''$ and $J$, and let $v=\theta(u)\in V(J)\subseteq V(G)$. Let $\mca{E}=\{\varphi(vv'):v'\in N_{J}(v)\}$. Then every hyperedge in $\mca{E}$ contains $v$, but we have $|\mca{E}|<|V(H'')|<h$. Since each vertex in $\Gamma$ has degree more than $h$, there is a hyperedge $f\in E(\Gamma)\setminus \mca{E}$ with $v\in f$. We further show that:

\sta{\label{st:antihyp} $f\setminus \{v\}$ and $V(J)\setminus \{v\}$ are anticomplete in $G$.}

Suppose not. Then there exists $u\in f\setminus \{v\}$ such that either $u\in V(J)\setminus \{v\}$ or $u$ has a neighbor in $V(J)\setminus \{v\}$ in $G$. By the definition of $\mca{E}$ and the choice of $f$, we have $N_{J}(v)\cap (f\setminus \{v\})=\varnothing$, and so $u\notin N_{J}(v)$. Therefore, since $J$ is connected, it follows that there is a path $Q$ in $J$ from $v$ to a vertex $w\in V(J)\setminus \{v\}$ such that $uw\in E(G^-)$. In particular, the unique neighbor $v'$ of $v$ in $Q$ belongs to $N_{J}(v)$, and so $\varphi(vv')\neq f$. But now $P=u\dd w\dd Q\dd v$ is a path of length at most $|V(Q)|\leq |V(J)|<h$ between distinct vertices $u,v\in f$ which is a subgraph of $G^-$ (and so of $G$), and $vv'$ is an edge of $P$ such that $\varphi(vv')\neq f$, a contradiction to \eqref{st:shortcycle}. This proves \eqref{st:antihyp}.
\medskip

We can now finish the proof. By \eqref{st:antihyp}, we have $f_0\neq f$. Since $B$ is a governing block of $H'$ and since $B'\in \mca{B}(H')\setminus \{B\}$, it follows that $B'$ is omnipresent in $W$, and so in $W_f$. Recall also that $u\in V(B')$ and $v\in f=V(W_f)$. Thus, there is an injection $\psi:V(B')\rightarrow V(W_f)$ such that $\psi(u)=v$ and $\psi$ is an isomorphism between $B'$ and $B''=W_f[\psi(V(B'))]$. Since $f_0\neq f$, it follows that $B''$ is an induced subgraph of $G^-$. Hence, by \eqref{st:antihyp}, $J\cup B''$ is an induced subgraph of $G^-$ isomorphic to $H'$. This completes the proof of \Cref{lem:maincentered}.
\end{proof}

We are now in a position to prove \Cref{thm:maincentered}:

\begin{proof}[Proof of \Cref{thm:maincentered}]
Let $H$ be a graph with a governing block $B$. By definition, there is a $B$-deletion-saturated graph $W$ such that every block $B'\in \mca{B}(H)\setminus \{B\}$ is omnipresent in $W$. Since $B$ has no isolated vertices, we may assume without loss of generality that $W$ has no isolated vertices either. We begin by showing that:

\sta{\label{st:conn} There is a connected graph $\hat{H}$ such that $H$ is an induced subgraph of $\hat{H}$ and $B$ is a governing block of $\hat{H}$.}

If $H$ is connected, then we are done by choosing $\hat{H}=H$. Assume that $H$ is not connected. Let $\hat{H}$ be a graph obtained from $H$ by adding an extra vertex with exactly one neighbor in each component of $H$. Then $\hat{H}$ is connected and $H$ is an induced subgraph of $\hat{H}$. Moreover, we have $\mca{B}(H)\subseteq \mca{B}(\hat{H})$, and every block of $\hat{H}$ that is not a block of $H$ is isomorphic to $K_2$. Therefore, since $B$ is a governing block of $H$ and since $W$ has no isolated vertices, it follows that $B$ is a governing block of $\hat{H}$ as well. This proves \eqref{st:conn}.
\medskip

Let $\hat{H}$ be as given by \eqref{st:conn}. Then, by \Cref{lem:maincentered}, there is a $B$-free graph $G$ such that for every $e\in E(G)$, $G-e$ has an induced subgraph isomorphic to $\hat{H}$. But now, since $B$ is an induced subgraph of $H$, and $H$ is an induced subgraph of $\hat{H}$, we deduce that $G$ is $H$-deletion-saturated. This completes the proof of \Cref{thm:maincentered}.
\end{proof}

\section{Special pairs of leaves}\label{sec:severed}

In this last section, we prove \Cref{thm:firstmain}\ref{thm:firstmain_d} and \ref{thm:firstmain_e}. For \Cref{thm:firstmain}\ref{thm:firstmain_d}, our method also works for another kind of pair of leaves. (Recall that $\alpha(G)$ denotes the maximum cardinality of a stable set in $G$.)

\begin{theorem}\label{thm:specialleaves}
    Let $H$ be a graph of maximum degree $d\in \poi$, let $x_1,x_2$ be two leaves of $H$ and let $y_1,y_2$, respectively, be the unique neighbors of $x_1,x_2$ in $H$. Assume that either
    \begin{enumerate}[{\rm (a)}]
        \item\label{thm:specialleaves_a} $y_1=y_2$; or
        \item\label{thm:specialleaves_b} $y_1y_2\in E(H)$, that is, $x_1,x_2$ are at distance three in $H$; or
        \item\label{thm:specialleaves_c} $\alpha(N_H(y_1)) + \alpha(N_H(y_2)) \geq d+2$.
    \end{enumerate}
    Then $H$ is deletion-normal.
\end{theorem}

Instead of \ref{thm:firstmain}\ref{thm:firstmain_e}, we prove more generally that the line graph of nearly \textit{every} graph with two leaves at distance at least three is deletion-normal. Given a graph $H$, a \textit{claw component of $H$} is a component of $H$ that is isomorphic to $K_{1,3}$, and a \textit{triangle component of $H$} is a component of $H$ that is isomorphic to $K_{3}$. A pair $(x_1,x_2)$ of vertices in $H$ is said to be \textit{severed} if:
\begin{itemize}
\item $x_1,x_2$ are leaves of $H$ at distance at least three from each other (that is, $x_1,x_2$ are non-adjacent and the unique neighbors of $x_1,x_2$ in $H$ are distinct); and
    \item none of $x_1,x_2$ belongs to a claw component of $H$.
\end{itemize}
We show that:

\begin{theorem}\label{thm:mainsevered}
   Let $H$ be a graph that has a severed pair. Then $\lng(H)$ is deletion-normal.
\end{theorem}

For instance, if $H$ is a tree, then $H$ has a severed pair if and only if $H$ is not a star. Also, conveniently, if $H$ is a tree, then $\lng(H)$ is a complete graph if and only if $H$ is a star. Since complete graphs are not deletion-saturated, this deduces \ref{thm:firstmain}\ref{thm:firstmain_e} from \Cref{thm:mainsevered}:

\begin{corollary}
    For a tree $H$, $\lng(H)$ is deletion-normal if and only if $H$ is not a star.
\end{corollary}

The main ingredient in the proof of both Theorems~\ref{thm:specialleaves} and \ref{thm:mainsevered} is \Cref{thm:ktn-1} below. Before proving that result, we need some definitions. Let $\Gamma$ be a $k$-graph for some integer $k\geq 2$, and let $\ell\in \{1,\ldots, k-1\}$. We say that a hyperedge $e$ of $\Gamma$ is \textit{$\ell$-pendent} if there are at least $\ell$ vertices in $e$ whose degree in $\Gamma$ is one.

Let $k,n\in \poi$ with $k\geq 2$. We say that a graph $H$ is \textit{$(k,n)$-admissible} if:
\begin{itemize}
    \item $H$ is not isomorphic to the line graph of any $k$-graph on at most $n-1$ vertices; and
    \item for every $\ell\in \{1,\ldots, k-1\}$, there is a $k$-graph $\Gamma_{\ell}$ on at most $n+\ell-1$ vertices with two disjoint $\ell$-pendent hyperedges such that $H$ is isomorphic to $\lng(\Gamma_{\ell})$.
\end{itemize}

\begin{theorem}\label{thm:ktn-1}
  Let $k,n\in \poi$ with $k\geq 2$ and let $H$ be a $(k,n)$-admissible graph. Then $\lng(K_{n-1}^k)$ is $H$-deletion-saturated.
\end{theorem}
\begin{proof}
    Since $H$ is $(k,n)$-admissible, it follows that $H$ is not the line graph of any $k$-graph on at most $n-1$ vertices, and so $\lng(K_{n-1}^k)$ is $H$-free. It remains to show that for every edge $e_1e_2\in E(\lng(K_{n-1}^k))$ (where $e_1,e_2$ are distinct hyperedges of $K_{n-1}^k$), $H$ is isomorphic to an induced subgraph of $\lng(K_{n-1}^k)-e_1e_2$.
    
    Let $|e_1\cap e_2|=\ell$. Then $\ell\in \{1,\ldots, k-1\}$, and since $H$ is $(k,n)$-admissible, there is a $k$-graph $\Gamma_{\ell}$ on at most $n+\ell-1$ vertices with two disjoint $\ell$-pendent hyperedges $f_1,f_2$ such that $H$ is isomorphic to $\lng(\Gamma_{\ell})$. In particular, for each $i\in \{1,2\}$, there exists an $\ell$-subset $X_i$ of $f_i$ such that the vertices in $X_i$ belong to no hyperedge of $\Gamma_{\ell}$ except for $f_i$. 
    
    We construct a $k$-graph $\Gamma'$ as follows. Let $Y$ be a set (of vertices) of cardinality $\ell$ disjoint from $V(\Gamma_{\ell})$. For each $i\in \{1,2\}$, let $f'_i=(f_i\setminus X_i)\cup Y$. Let $\Gamma'$ be the $k$-graph with
    $$V(\Gamma')=(V(\Gamma_{\ell})\setminus (X_1\cup X_2))\cup Y;$$
    $$E(\Gamma')=(E(\Gamma_{\ell})\setminus \{f_1,f_2\})\cup \{f'_1,f'_2\}.$$
    Then we have:
    \begin{itemize}
    \item $|V(\Gamma')|= |V(\Gamma_{\ell})|-\ell\leq n-1$.
    \item $f'_1,f'_2\in E(\Gamma')$ with $f'_1\cap f'_2=Y$; in particular, $|f'_1\cap f'_2|=|Y|=\ell=|e_1\cap e_2|$.
        \item  $\lng(\Gamma_{\ell})$ is isomorphic to $\lng(\Gamma')-f'_1f'_2$; thus, $H$ is isomorphic to $\lng(\Gamma')-f'_1f'_2$.
    \end{itemize}
From the first two bullets, it follows that $\Gamma'$ is isomorphic to a sub-$k$-graph $\Gamma''$ of $K_{n-1}^k$ with $e_1,e_2 \in E(\Gamma'')$, via an isomorphism $\varphi : V(\Gamma') \rightarrow V(\Gamma'')$ such that $\{\varphi(v):v\in f'_i\}=e_i$ for each $i\in \{1,2\}$. In particular, the map $\Phi:E(\Gamma')\to E(\Gamma'')$ with $\Phi(f)=\{\varphi(v):v\in f\}$ is an isomorphism between $\lng(\Gamma')$ and the induced subgraph $\lng(\Gamma'')$ of $\lng(K_{n-1}^k)$ with $\Phi(f'_i)=e_i$ for every $i\in \{1,2\}$. It follows that $\lng(\Gamma')-f'_1f'_2$ is isomorphic to the induced subgraph $\lng(\Gamma'')-e_1e_2$ of $\lng(K_{n-1}^k)-e_1e_2$. But now by the third bullet above, $H=\lng(\Gamma_{\ell})$ is isomorphic to an induced subgraph of $\lng(K_{n-1}^k)-e_1e_2$. This completes the proof of \Cref{thm:ktn-1}.
\end{proof}

We also need the following:

\begin{lemma}\label{lem:existlinegraph}
   Let $H$ be a graph of maximum degree $d\in \poi$. Then there exists a $d$-graph $\Gamma$ such that $H$ is isomorphic to $\lng(\Gamma)$.
\end{lemma}
\begin{proof}
    Let $V(H)=\{v_1,\ldots, v_h\}$, where $h\in \poi$. For each $i\in \{1,\ldots, h\}$, let $d_i$ be the degree of $v_i$ in $H$. Let $U_1,\ldots, U_h$ be pairwise disjoint sets (of vertices), disjoint from $V(H)$, such that for each $i\in \{1,\ldots, h\}$, we have $|U_i|=d-d_i\geq 0$. Let $H'$ be the graph with
    $$V(H')=V(H)\cup \left(\bigcup_{i=1}^h U_i\right)$$
    and
    $$E(H')=E(H)\cup \left(\bigcup_{i=1}^h\{uv_i:u\in U_i\}\right).$$
  In other words, $H'$ is obtained from $H$ by adding, for each $i\in \{1,\ldots, h\}$, a set $U_i$ of $d-d_i$ exclusive leaves adjacent to $v_i$. For each $i\in \{1,\ldots, h\}$, let $f_i$ be the set of all edges in $E(H')$ that are incident to $v_i$ in $H'$. It follows that $|f_1|=\cdots=|f_h|=d$, and for all distinct $i,j\in \{1,\ldots, h\}$, if $v_iv_j\in E(H)$, then $f_i\cap f_j=\{v_iv_j\}$, and if $v_iv_j\notin E(H)$, then $f_i\cap f_j=\varnothing$. Let $\Gamma$ be the $d$-graph with $V(\Gamma)=E(H')$ and $E(\Gamma)=\{f_1,\ldots, f_h\}$. Then, for all distinct $i,j\in \{1,\ldots, h\}$, we have $v_iv_j\in E(H)$ if and only if $f_i\cap f_j\neq \varnothing$, which in turn implies that the map $\varphi:V(H)\to E(\Gamma)$ with $\varphi(v_i)=f_i$ is an isomorphism between $H$ and $\lng(\Gamma)$. This completes the proof of \Cref{lem:existlinegraph}.
\end{proof}

Let us now prove \Cref{thm:specialleaves}:

\begin{proof}[Proof of \Cref{thm:specialleaves}]

   By \Cref{lem:existlinegraph}, there is a $d$-graph whose line graph is isomorphic to $H$. Let $n\in \poi$ be minimum such that there exists a $d$-graph $\Gamma$ on $n$ vertices whose line graph is isomorphic to $H$. Let $\varphi:V(H)\to E(\Gamma)$ be an isomorphism between $H$ and $\lng(\Gamma)$. We show that $H$ is $(d,n)$-admissible. This, along with \Cref{thm:ktn-1}, implies that $H$ is deletion-normal.

   To begin with, we deduce from the minimality of $n$ that:

 \sta{\label{st:nohypergraph}There is no $d$-graph on at most $n-1$ vertices whose line graph is isomorphic to $H$.}

Note that for each $i\in \{1,2\}$, we have $\varphi(x_i)\cap \varphi(y_i)\neq\varnothing$, $\varphi(x_i)\setminus \varphi(y_i)\neq\varnothing$, and $\varphi(x_i)\cap e=\varnothing$ for all $e\in E(\Gamma)\setminus \{\varphi(y_i)\}$. For every $i\in \{1,2\}$ and $\ell\in \{1,\ldots, d-1\}$, let
   $$\delta_{i,\ell}=\max\{0,\ell-|\varphi(x_i)\setminus \varphi(y_i)|\}.$$

   We claim that:

\sta{\label{st:deltalimits} For each $i\in \{1,2\}$ and every $\ell\in \{1,\ldots, d-1\}$, we have

$$0\leq \delta_{i,\ell}\leq \min\{\ell,|\varphi(x_i)\cap \varphi(y_i)|\}-1.$$}

   Note that $\delta_{i,\ell}\geq 0$ is immediate from the definition. Since $\ell\geq 1$, we have $0\leq \ell-1$. Since $\varphi(x_i)\cap \varphi(y_i)\neq \varnothing$, it follows that $0\leq |\varphi(x_i)\cap \varphi(y_i)|-1$, and so $\ell-|\varphi(x_i)\setminus \varphi(y_i)|\leq \ell-1$. Since $\ell\leq d-1$, it follows that $\ell-|\varphi(x_i)\setminus \varphi(y_i)|\leq d-1-|\varphi(x_i)\setminus \varphi(y_i)|=|\varphi(x_i)\cap \varphi(y_i)|-1$. From the latter four inequalities, we deduce that
   $$\max\{0,\ell-|\varphi(x_i)\setminus \varphi(y_i)|\}\leq \min\{\ell-1,|\varphi(x_i)\cap \varphi(y_i)|-1\}.$$
   This proves \eqref{st:deltalimits}.

   \sta{\label{st:deltasum} For every $\ell\in \{1,\ldots, d-1\}$, either $\delta_{1,\ell}=\delta_{2,\ell}=0$ or
 $$\delta_{1,\ell}+\delta_{2,\ell}\leq \ell-1+|\varphi(x_1)\cap \varphi(y_1)|+|\varphi(x_2)\cap \varphi(y_2)|-d.$$}

Assume that either $\delta_{1,\ell}$ or $\delta_{2,\ell}$ is nonzero; say $\delta_{1,\ell}\neq 0$. Then $\delta_{1,\ell}=\ell-|\varphi(x_1)\setminus \varphi(y_1)|$. By \eqref{st:deltalimits} and since $\varphi(x_1)\cap \varphi(x_2)=\varnothing$, we have
    \begin{align*}
\delta_{1,\ell}+\delta_{2,\ell}&\leq \ell-|\varphi(x_1)\setminus \varphi(y_1)|+|\varphi(x_2)\cap \varphi(y_2)|-1\\
        &= \ell-(d-|\varphi(x_1)\cap \varphi(y_1)|)+|\varphi(x_2)\cap \varphi(y_2)|-1\\
&= \ell-1+|\varphi(x_1)\cap \varphi(y_1)|+|\varphi(x_2)\cap \varphi(y_2)|-d.
    \end{align*}
 This proves \eqref{st:deltasum}.

\medskip

From \eqref{st:deltalimits}, it follows that for each $i\in \{1,2\}$ and every $\ell\in \{1,\ldots, d-1\}$, there is a $\delta_{i,\ell}$-subset $D_{i,\ell}$ of $\varphi(x_i)\cap \varphi(y_i)$ such that $(\varphi(x_i)\cap \varphi(y_i))\setminus D_{i,\ell}\neq\varnothing$.

\sta{\label{st:dist2leaves} If $y_1=y_2$, then for every $\ell\in \{1,\ldots, d-1\}$,  there is a $d$-graph $\Gamma_{\ell}$ with $|V(\Gamma_{\ell})|< n+\ell$ and two disjoint $\ell$-pendent hyperedges such that $H$ is isomorphic to $\lng(\Gamma_{\ell})$.}

Write $y_1=y_2=y\in V(H)$. We construct the $d$-graph $\Gamma_{\ell}$ as follows. Let $C$ be a set of vertices, disjoint from $V(\Gamma)$, with $|C|=\delta_{1,\ell}+\delta_{2,\ell}$, and let $f=(\varphi(y)\setminus (D_{1,\ell}\cup D_{2,\ell}))\cup C$ (see \Cref{fig:mul19}). Then $|f|=d$. Let $\Gamma_{\ell}$ be the $d$-graph with
    $$V(\Gamma_{\ell})=V(\Gamma)\cup C;$$
    $$E(\Gamma_{\ell})=(E(\Gamma)\setminus \{\varphi(y)\})\cup \{f\}.$$
Now, note that since $\varphi(x_1)\cap \varphi(x_2)=\varnothing$, we have $|\varphi(x_1)\cap \varphi(y)|+|\varphi(x_2)\cap \varphi(y)|\leq |\varphi(y)|=d$, and so $\ell-1+|\varphi(x_1)\cap \varphi(y)|+|\varphi(x_2)\cap \varphi(y)|-d\leq \ell-1$. Therefore, by \eqref{st:deltasum} and since $\ell\geq 1$, it follows that $\delta_{1,\ell}+\delta_{2,\ell}\leq \ell-1$. We deduce that
$$|V(\Gamma_{\ell})|=|V(\Gamma)|+|C|=n+\delta_{1,\ell}+\delta_{2,\ell}\leq n+\ell-1.$$
Moreover, since $E(\Gamma_{\ell})\setminus \{f\}=E(\Gamma)\setminus \{\varphi(y)\}$, for each $i\in \{1,2\}$, we have $\varphi(x_i)\cap e=\varnothing$ for every $e\in E(\Gamma_{\ell})\setminus \{f\}$. Also, we have
$\varphi(x_i)\setminus f=(\varphi(x_i)\setminus \varphi(y))\cup D_{i,\ell}$, and so
 $$|\varphi(x_i)\setminus f|=|\varphi(x_i)\setminus \varphi(y)|+\delta_{i,\ell}
  \geq |\varphi(x_i)\setminus \varphi(y)|+(\ell-|\varphi(x_i)\setminus \varphi(y)|)
  =\ell.$$
Thus, $\varphi(x_1)$ and $\varphi(x_2)$ are two disjoint $\ell$-pendent hyperedges of $\Gamma_{\ell}$. In addition, it is straightforward to check that, since $\varphi$ is an isomorphism between $H$ and $\lng(\Gamma)$, the map $\varphi':V(H)\to E(\Gamma_{\ell})$ with $\varphi'(y)=f$ and $\varphi'(x)=\varphi(x)$ for all $x\in V(H)\setminus \{y\}$ is an isomorphism between $H$ and $\lng(\Gamma_{\ell})$. This proves \eqref{st:dist2leaves}.
\begin{figure}
    \centering
    \includegraphics[width=0.6\linewidth]{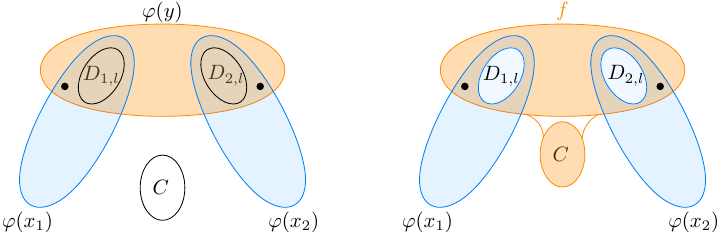}
    \caption{Proof of  \eqref{st:dist2leaves}. The hyperedge $\varphi(y)$ and sets $D_{i,\ell}$ (left) and the new hyperedge $f$ (right).}
    \label{fig:mul19}
\end{figure}

\sta{\label{st:dist3leaves} If $y_1y_2\in E(H)$, then for every $\ell\in \{1,\ldots, d-1\}$, there is a $d$-graph $\Gamma_{\ell}$ with $|V(\Gamma_{\ell})|< n+\ell$ and two disjoint $\ell$-pendent hyperedges such that $H$ is isomorphic to $\lng(\Gamma_{\ell})$.}

Note that $\varphi(y_1)\neq \varphi(y_2)$ but $\varphi(y_1)\cap \varphi(y_2)\neq \varnothing$. Assume, without loss of generality, that $\delta_{1,\ell}\leq \delta_{2,\ell}$. We construct the $d$-graph $\Gamma_{\ell}$ as follows. Let $C_{1,\ell}$ and $C_{2,\ell}$ be disjoint sets of vertices, also disjoint from $V(\Gamma)$, with $|C_{1,\ell}|=\delta_{1,\ell}$ and $|C_{2,\ell}|=\delta_{2,\ell}-\delta_{1,\ell}$. Let $f_1=(\varphi(y_1)\setminus D_{1,\ell})\cup C_{1,\ell}$ and let $f_2=(\varphi(y_2)\setminus D_{2,\ell})\cup C_{1,\ell}\cup C_{2,\ell}$ (see \Cref{fig:mul20}). Then $|f_1|=|f_2|=d$, and $f_1\cap f_2=(\varphi(y_1)\cap \varphi(y_2))\cup C_{1,\ell}$; in particular, we have $C_{2,\ell}\subseteq f_2\setminus f_1$. Let $\Gamma_{\ell}$ be the $d$-graph with
$$V(\Gamma_{\ell})=V(\Gamma)\cup C_{1,\ell}\cup C_{2,\ell};$$
$$E(\Gamma_{\ell})=(E(\Gamma)\setminus \{\varphi(y_1),\varphi(y_2)\})\cup \{f_1,f_2\}.$$
Now, by \eqref{st:deltalimits}, we have $\delta_{2,\ell}\leq \ell-1$, and so
$$|V(\Gamma_{\ell})|=|V(\Gamma)|+|C_{1,\ell}|+|C_{2,\ell}|=n+\delta_{2,\ell}\leq n+\ell-1.$$
Also, since $E(\Gamma_{\ell})\setminus \{f_1,f_2\}=E(\Gamma)\setminus \{\varphi(y_1),\varphi(y_2)\}$, and since $\varphi(x_1)\cap f_2=\varphi(x_2)\cap f_1=\varnothing$, it follows that for each $i\in \{1,2\}$, we have $\varphi(x_i)\cap e=\varnothing$ for every $e\in E(\Gamma_{\ell})\setminus \{f_i\}$. Moreover, note that $\varphi(x_i)\setminus f_i=(\varphi(x_i)\setminus \varphi(y_i))\cup D_{i,\ell}$, and so
 $$|\varphi(x_i)\setminus f_i|=|\varphi(x_i)\setminus \varphi(y_i)|+\delta_{i,\ell}
  \geq |\varphi(x_i)\setminus \varphi(y_i)|+(\ell-|\varphi(x_i)\setminus \varphi(y_i)|)
  =\ell.$$
It follows that $\varphi(x_1)$ and $\varphi(x_2)$ are two disjoint $\ell$-pendent hyperedges of $\Gamma_{\ell}$. In addition, it is straightforward to check that, since $\varphi$ is an isomorphism between $H$ and $\lng(\Gamma)$, the map $\varphi':V(H)\to E(\Gamma_{\ell})$ with $\varphi'(y_i)=f_i$ for each $i\in \{1,2\}$ and $\varphi'(x)=\varphi(x)$ for all $x\in V(H)\setminus \{y_1,y_2\}$ is an isomorphism between $H$ and $\lng(\Gamma_{\ell})$. This proves \eqref{st:dist3leaves}.
\begin{figure}
    \centering
    \includegraphics[width=0.8\linewidth]{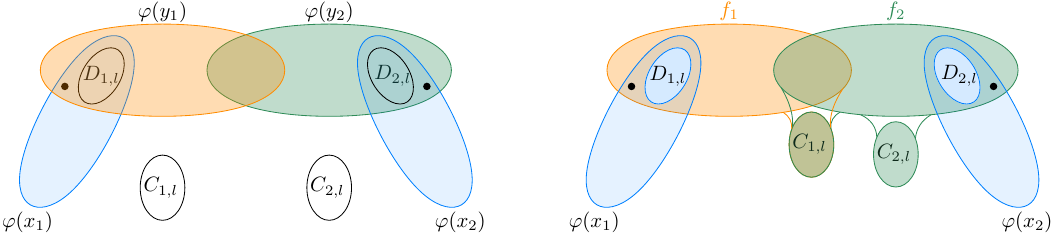}
    \caption{Proof of  \eqref{st:dist3leaves}. The hyperedges $\varphi(y_1)$, $\varphi(y_2)$ and sets $D_{i,\ell}$ (left) and the new hyperedges $f_1$ and $f_2$ (right).}
    \label{fig:mul20}
\end{figure}

\sta{\label{st:alphaleaves} If $\alpha(N_H(y_1))+\alpha(N_H(y_2))\geq d+2$, then for every $\ell\in \{1,\ldots, d-1\}$,  there is a $d$-graph $\Gamma_{\ell}$ with $|V(\Gamma_{\ell})|< n+\ell$ and two disjoint $\ell$-pendent hyperedges such that $H$ is isomorphic to $\lng(\Gamma_{\ell})$.}

We construct the $d$-graph $\Gamma_{\ell}$ as follows. Let $C_{1,\ell}$ and $C_{2,\ell}$ be disjoint sets of vertices, also disjoint from $V(\Gamma)$, with $|C_{1,\ell}|=\delta_{1,\ell}$ and $|C_{2,\ell}|=\delta_{2,\ell}$. Let $f_1=(\varphi(y_1)\setminus D_{1,\ell})\cup C_{1,\ell}$ and let $f_2=(\varphi(y_2)\setminus D_{2,\ell})\cup C_{2,\ell}$. Then $|f_1|=|f_2|=d$, and $f_1\cap f_2=\varphi(y_1)\cap \varphi(y_2)$; in particular, we have $C_{1,\ell}\subseteq f_1\setminus f_2$ and $C_{2,\ell}\subseteq f_2\setminus f_1$. Let $\Gamma_{\ell}$ be the $d$-graph with $$V(\Gamma_{\ell})=V(\Gamma)\cup C_{1,\ell}\cup C_{2,\ell};$$
    $$E(\Gamma_{\ell})=(E(\Gamma)\setminus \{\varphi(y_1),\varphi(y_2)\})\cup \{f_1,f_2\}.$$
Since $\alpha(N_H(y_1))+\alpha(N_H(y_2))\geq d+2$, it follows that
\begin{align*}
|\varphi(x_1)\cap \varphi(y_1)|+|\varphi(x_2)\cap \varphi(y_2)|&=2d-|\varphi(y_1)\setminus \varphi(x_1)|-|\varphi(y_2)\setminus \varphi(x_2)|\\
        &\leq 2d-(\alpha(N_H(y_1))-1)-(\alpha(N_H(y_2))-1)\\
&\leq d.
    \end{align*}
We deduce that
$\ell-1+|\varphi(x_1)\cap \varphi(y_1)|+|\varphi(x_2)\cap \varphi(y_2)|-d\leq \ell-1$.
Thus, by \eqref{st:deltasum} and since $\ell\geq 1$, we have $\delta_{1,\ell}+\delta_{2,\ell}\leq \ell-1$, and so
$$|V(\Gamma_{\ell})|=|V(\Gamma)|+|C_{1,\ell}|+|C_{2,\ell}|=n+\delta_{1,\ell}+\delta_{2,\ell}\leq n+\ell-1.$$
Since $E(\Gamma_{\ell})\setminus \{f_1,f_2\}=E(\Gamma)\setminus \{\varphi(y_1),\varphi(y_2)\}$, and since $\varphi(x_1)\cap f_2=\varphi(x_2)\cap f_1=\varnothing$, it follows that for each $i\in \{1,2\}$, we have $\varphi(x_i)\cap e=\varnothing$ for every $e\in E(\Gamma_{\ell})\setminus \{f_i\}$. Moreover, note that
$\varphi(x_i)\setminus f_i=(\varphi(x_i)\setminus \varphi(y_i))\cup D_{i,\ell}$, and so
 $$|\varphi(x_i)\setminus f_i|=|\varphi(x_i)\setminus \varphi(y_i)|+\delta_{i,\ell}
  \geq |\varphi(x_i)\setminus \varphi(y_i)|+(\ell-|\varphi(x_i)\setminus \varphi(y_i)|)
  =\ell.$$
It follows that $\varphi(x_1)$ and $\varphi(x_2)$ are two disjoint $\ell$-pendent hyperedges of $\Gamma_{\ell}$. Furthermore, it is straightforward to check that, since $\varphi$ is an isomorphism between $H$ and $\lng(\Gamma)$, the map $\varphi':V(H)\to E(\Gamma_{\ell})$ with $\varphi'(y_i)=f_i$ for each $i\in \{1,2\}$ and $\varphi'(x)=\varphi(x)$ for all $x\in V(H)\setminus \{y_1,y_2\}$ is an isomorphism between $H$ and $\lng(\Gamma_{\ell})$.
This proves \eqref{st:alphaleaves}.
\begin{figure}
    \centering
    \includegraphics[width=0.8\linewidth]{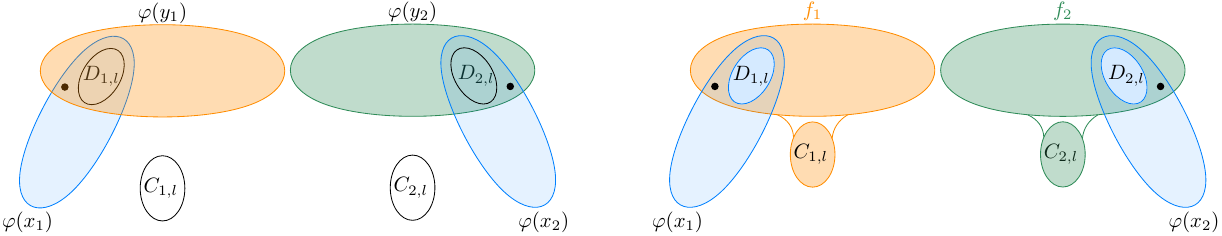}
    \caption{Proof of  \eqref{st:alphaleaves}. The hyperedges $\varphi(y_1)$, $\varphi(y_2)$ and sets $D_{i,\ell}$ (left) and the new hyperedges $f_1$ and $f_2$ (right).}
    \label{fig:mul21}
\end{figure}
\medskip

Now, \Cref{thm:specialleaves} follows from \eqref{st:nohypergraph}, \eqref{st:dist2leaves}, \eqref{st:dist3leaves} and \eqref{st:alphaleaves}.
\end{proof}

For the proof of \Cref{thm:mainsevered}, we use the following well-known result of Whitney \cite{whitneyline}:

\begin{theorem}[Whitney \cite{whitneyline}]\label{thm:whitneyline}
    Let $H_1$ and $H_2$ be connected graphs. Then $\lng(H_1)$ and $\lng(H_2)$ are isomorphic if and only if either $H_1$ and $H_2$ are isomorphic, or one of $H_1,H_2$ is isomorphic to $K_{1,3}$ and the other is isomorphic to $K_3$.
\end{theorem}

\begin{proof}[Proof of \Cref{thm:mainsevered}]
Let $s\geq 0$ be the number of isolated vertices of $H$, let $t\geq 0$ be the number of claw components of $H$ and let $n=|V(H)|-s-t$. Let $H'$ be the graph obtained from $H$ by first removing all isolated vertices of $H$ and then replacing every claw component of $H$ with a triangle component. Then $|V(H')|=n$, and it follows from \Cref{thm:whitneyline} that $H'$ is the unique graph with the minimum number of vertices whose line graph is isomorphic to $\lng(H)$. In particular, there is no graph on at most $n-1$ vertices whose line graph is isomorphic to $\lng(H)$. Also, $|V(H)|\leq n$, and since $H$ has a severed pair, it follows that, $H$ (viewed as a $2$-graph) has two disjoint $1$-pendent hyperedges. Now, $\lng(H)$ is $(2,n)$-admissible, and the result follows from \Cref{thm:ktn-1}.
\end{proof}

\bibliographystyle{abbrv}
\bibliography{ref}
\newpage

\appendix
\section{Graphs on at most six vertices}\label{appendix}

Here, we give a proof of \Cref{thm:six}. Up to isomorphism, there are \textbf{202} non-complete graphs on at most six vertices, all depicted in \Cref{fig:allsix} (see, for example, \cite{mckay}). The first six outcomes of \Cref{thm:firstmain} (or their strengthenings, to be more precise) establish the deletion-normality of many of them. (Some graphs could be handled in many different ways, in which case we have broken the tie arbitrarily.) There remains a total of \textbf{25} ``sporadic'' graphs on at most six vertices, each of which requires new ideas for the proof of their deletion-normality.

\subsection{Complete multipartite graphs}\label{subsec:multi} Among the 202 non-complete graphs on at most six vertices, there are exactly \textbf{23} complete multipartite graphs; see Figure~\ref{fig:K_(s,t,t)_Table}, where these graphs are depicted with the same labeling as in \Cref{fig:allsix}, and with a multipartition of their vertices given. In particular, the only entries which are complete multipartite with at least two parts of maximum cardinality are
\begin{itemize}
    \item 11C (isomorphic to $K_{3,3}$), which is handled by \ref{thm:Ktts2} with 
    $r=1$, $t=3$, and $s=1$.
    
     \item 13H (isomorphic to $K_{2,2,2}$), which is handled by \ref{thm:Ktts0} with $s=2$ and $s'=0$.
      \item 13J (isomorphic to $K_{2,2,1,1}$), which is handled by \ref{thm:Ktts0} with $s=1$ and $s'=2$.
      \item 13K (isomorphic to $K_{2,1,1,1,1}$), which is handled by \ref{thm:Ktts0} with $s=0$ and $s'=4$.
      \item 16H (isomorphic to $K_{2,2,1}$), which is handled by \ref{thm:Ktts0} with $s=s'=1$.
      \item 17I (isomorphic to $K_{2,2}$), which is handled by \ref{thm:Ktts0} with $s=1$ and $s'=0$.
\end{itemize}
The rest are all handled by \ref{thm:compmulti}.

\subsection{Adding vertices of degree zero or one}\label{subsec:0or1}
Let us first prove a lemma:

\begin{lemma}\label{lem:isolated}
Let $G, H$ be graphs. Let $u$ be an isolated vertex of $G$ and let $v$ be an isolated vertex of $H$. Assume that $G\setminus \{u\}$ is $(H\setminus \{v\})$-deletion-saturated. Then $G$ is $H$-deletion-saturated.    
\end{lemma}

\begin{proof}
    First, we show that $G$ is $H$-free. Suppose for a contradiction that $G$ has an induced subgraph $H'$ isomorphic to $H$. Let $f:V(H)\rightarrow V(H')$ be an isomorphism between $H$ and $H'$. Since $G\setminus \{u\}$ is $(H\setminus \{v\})$-free, it follows that $H'$ is not an induced subgraph of $G\setminus \{u\}$. In other words, we have $u\in V(H')$. Since $G\setminus \{u\}$ is $(H\setminus \{v\})$-free, it follows that $H'\setminus \{u\}$ is not isomorphic to $H\setminus \{v\}$, and so $f(v)\neq u$. In particular, $u$ and $f(v)$ are distinct isolated vertices of $H'$. It follows that $H'\setminus \{u\}$ and $H'\setminus \{f(v)\}$ are isomorphic, which in turn implies that $H'\setminus \{u\}$ and $H\setminus \{v\}$ are isomorphic. But then $H'\setminus \{u\}$ is an induced subgraph of $G\setminus \{u\}$ isomorphic to $H\setminus \{v\}$, a contradiction. We deduce that $G$ is $H$-free. Now, let $e\in E(G)$. Then $e\in E(G\setminus \{u\})$. Since $G\setminus \{u\}$ is $(H\setminus \{v\})$-deletion-saturated, it follows that $(G\setminus \{u\})-e$ has an induced subgraph $H''$ isomorphic to $H\setminus \{v\}$. Hence, the induced subgraph of $G-e$ with vertex set $V(H'')\cup \{u\}$ is isomorphic to $H$. This completes the proof of \Cref{lem:isolated}.
\end{proof}

Of the remaining 
$202-23=179$ graphs from Subsection~\ref{subsec:multi}, the \textbf{37} graphs in \Cref{fig:isolated} have isolated vertices, and removing their isolated vertices leaves a non-complete graph. By \Cref{lem:isolated}, we may assume that these graphs are handled by removing their isolated vertices and reducing to smaller graphs. Similarly, the \textbf{32} graphs in \Cref{fig:block} each have a unique non-complete block $B$ such that $B$ is not the whole graph and all other blocks are isomorphic to $K_2$. It follows that $B$ is governing (note that a $B$-deletion-saturated graph need not have isolated vertices). By \ref{thm:maincentered}, we may assume that these graphs are also handled by reducing to $B$, which is a smaller graph.

\subsection{Near-line-graphs and severed pairs}\label{subsec:nearline}
We say that a graph $H$ is a biline graph \textit{biline graph} if there exists a bipartite graph $F$ such that $H=\lng(F)$. There are well-known characterizations of both line graphs and biline graphs in terms of their forbidden induced subgraphs.

\begin{theorem}[Beineke \cite{BEINEKE1970129}]\label{thm:lineobst}
A graph $G$ is a line graph if and only if $G$ has no induced subgraph isomorphic to any of the nine graphs in \Cref{fig:lineobst}.
\end{theorem}

\begin{theorem}[See Peterson \cite{PETERSON2003223}]\label{thm:bilineobst}
A graph $G$ is a biline graph if and only if $G$ is claw-free, diamond-free, and odd-hole-free.
\end{theorem}

(Recall that the \textit{claw} is $K_{1,3}$, the \textit{diamond} is the graph obtained from $K_4$ by removing an edge, and an \textit{odd hole} is an induced cycle of odd length at least five. See \Cref{fig:bilineobst}.)

\begin{figure}[t!]
\centering
\includegraphics[scale=0.5]{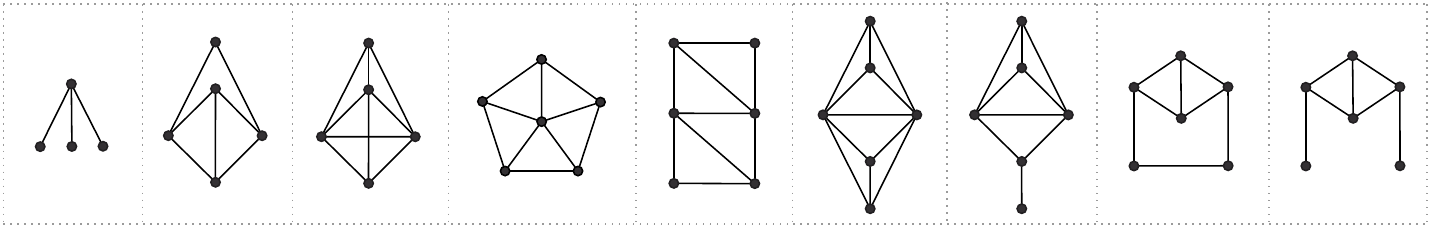}
\caption{Induced subgraph obstructions to line graphs.}
\label{fig:lineobst}
\end{figure}

\begin{figure}[t!]
\centering
\includegraphics[scale=0.5]{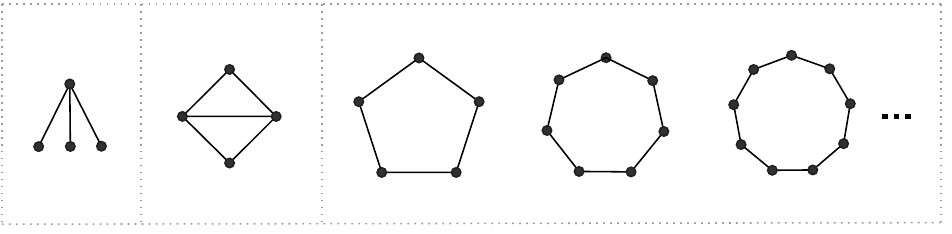}
\caption{Induced subgraph obstructions to biline graphs.}
\label{fig:bilineobst}
\end{figure}

Here, we are also interested in graphs that are ``an edge away'' from being a (bi)line graph. We say that a graph $G$ is a \textit{line\lm graph} if $G$ is not a line graph, but there exists $e\in E(\ol{G})$ such that $G+e$ is a line graph. We also say that $G$ is a \textit{line\lp graph} if $G$ is not a line graph, but there exists $e\in E(G)$ such that $G-e$ is a line graph. Intuitively, line\lm graphs are graphs that are ``an edge less than'' a line graph, and line\lp graphs are graphs that are ``an edge more than'' a line graph. 

Similarly, we say that $G$ is a \textit{biline\lm graph} if $G$ is not a biline graph, but there exists $e\in E(\ol{G})$ such that $G+e$ is a biline graph, and that $G$ is a \textit{biline\lp graph} if $G$ is not a biline graph, but there exists $e\in E(G)$ such that $G-e$ is a biline graph. Again, this roughly means that biline\lm graphs are graphs that are ``an edge less than'' a biline graph, and biline\lp graphs are graphs that are ``an edge more than'' a biline graph. 

We need the following lemma:

\begin{lemma}\label{lem:linegraphs}  Let $H$ be a graph on $h\in \poi$ vertices. Then the following hold:

\begin{enumerate}[{\rm (a)}]
    \item \label{lem:linegraphs_a}   If $H$ is a line\lm graph, then $\lng(K_{2h})$ is $H$-deletion-saturated.
     \item \label{lem:linegraphs_b}   If $\ol{H}$ is a line\lp graph, then $\ol{\lng(K_{2h})}$ is $H$-deletion-saturated.
      \item \label{lem:linegraphs_c}   If $H$ is a biline\lm graph, then $\lng(K_{h,h})$ is $H$-deletion-saturated.
     \item \label{lem:linegraphs_d}   If $\ol{H}$ is a biline\lp graph, then $\ol{\lng(K_{h,h})}$ is $H$-deletion-saturated.
\end{enumerate}
\end{lemma}

\begin{proof}
 First, we prove \ref{lem:linegraphs}\ref{lem:linegraphs_a}. Since $H$ is not a line graph, it follows that $\lng(K_{2h})$ is $H$-free. It remains to show that for every edge $ab\in E(\lng(K_{2h}))$ (where $a,b\in E(K_{2h})$), there is an induced subgraph of $\lng(K_{2h})-ab$ that is isomorphic to $H$. Let $a=uv$ and $b=uw$ where $u,v,w\in V(K_{2h})$ are distinct. Since $H$ is a line\lm graph, there exists $xy\in E(\ol{H})$ for which $H+xy$ is a line graph. Choose a graph $G'$ with no isolated vertex such that $\lng(G')=H+xy$. Since $x,y$ are adjacent in $H+xy$, it follows that there are distinct vertices $u',v',w'\in V(G')$ for which $x=u'v'\in E(G')$ and $y=u'w'\in E(G')$. Since $|E(G')|=|V(H+xy)|=|V(H)|=h$ and $G'$ has no isolated vertex, it follows that $G'$ has at most $2h$ vertices. Therefore, there is a subgraph $G$ of $K_{2h}$ isomorphic to $G'$ with $a=uv\in E(G)$ and $b=uw\in E(G)$, as well as an isomorphism $f:V(G)\rightarrow V(G')$ between $G$ and $G'$ such that $f(u)=u', f(v)=v'$, and $f(w)=w'$. Define the map $\varphi:E(G)\rightarrow E(G')$ as follows: for every edge $q_1q_2\in E(G)$, let $\varphi(q_1q_2)=f(q_1)f(q_2)\in E(G')$. Then $\varphi$ is an isomorphism between $\lng(G)$ and $\lng(G')$ where $\varphi(a)=x$ and $\varphi(b)=y$. Moreover, note that $\lng(G)$ is the induced subgraph of $\lng(K_{2h})$ with vertex set $E(G)$. Hence, $\lng(G)-ab$ is an induced subgraph of $\lng(K_{2h})-ab$ isomorphic to $\lng(G')-xy=H$. This proves \ref{lem:linegraphs}\ref{lem:linegraphs_a}.

  Second, we prove \ref{lem:linegraphs}\ref{lem:linegraphs_b}. Since $\ol{H}$ is not a line graph, it follows that $\lng(K_{2h})$ is $\ol{H}$-free, and so $\ol{\lng(K_{2h})}$ is $H$-free. It remains to show that for every edge $ab\in E(\ol{\lng(K_{2h})})$ (where $a,b\in E(K_{2h})$), there is an induced subgraph of $\ol{\lng(K_{2h})}-ab$ that is isomorphic to $H$. Let $a=uv$, and $b=wz$ where $u,v,w,z\in V(K_{2h})$ are distinct. Since $\ol{H}$ is a line\lp graph, there exists $xy\in E(\ol{H})$ for which $\ol{H}-xy$ is a line graph. Choose a graph $G'$ with no isolated vertex such that $\lng(G')=\ol{H}-xy$. Since $x,y$ are not adjacent in $\ol{H}-xy$, it follows that there are distinct vertices $u', v', w', z' \in V(G')$ such that $x=u'v'\in E(G')$ and $y=w'z'\in E(G')$. Since $|E(G')|=|V(\ol{H}-xy)|=|V(H)|=h$ and $G'$ has no isolated vertex, it follows that $G'$ has at most $2h$ vertices. Therefore, there is a subgraph $G$ of $K_{2h}$ isomorphic to $G'$ with $a=uv\in E(G)$ and $b=wz\in E(G)$, as well as an isomorphism $f:V(G)\rightarrow V(G')$ such that $f(u)=u', f(v)=v', f(w)=w'$, and $f(z)=z'$. Define the map $\varphi:E(G)\rightarrow E(G')$ as follows: for every edge $q_1q_2\in E(G)$, let $\varphi(q_1q_2)=f(q_1)f(q_2)\in E(G')$. Then $\varphi$ is an isomorphism between $\lng(G)$ and $\lng(G')$ where $\varphi(a)=x$ and $\varphi(b)=y$. Moreover, note that $\lng(G)$ is the induced subgraph of $\lng(K_{2h})$ with vertex set $E(G)$. Hence, $\lng(G)+ab$ is an induced subgraph of $\lng(K_{2h})+ab$ isomorphic to $\lng(G')+xy=\ol{H}$, which in turn implies that $\ol{\lng(G)+ab}=\ol{\lng(G)}-ab$ is an induced subgraph of $\ol{\lng(K_{2h})+ab}=\ol{\lng(K_{2h})}-ab$ isomorphic to $\ol{\lng(G')+xy}=H$. This proves \ref{lem:linegraphs}\ref{lem:linegraphs_b}.

   Third, we prove \ref{lem:linegraphs}\ref{lem:linegraphs_c}. Since $H$ is not the line graph of a bipartite graph, it follows that $\lng(K_{h,h})$ is $H$-free. It remains to show that for every edge $ab\in E(\lng(K_{h,h}))$ (where $a,b\in E(K_{h,h})$), there is an induced subgraph of $\lng(K_{h,h})-ab$ that is isomorphic to $H$. Let $a=uv$ and $b=uw$ where $u,v,w\in V(K_{h,h})$ are distinct. Since $H$ is a biline\lm graph, there exists $xy\in E(\ol{H})$ for which $H+xy$ is the line graph of a bipartite graph. Choose a bipartite graph $G'$ with no isolated vertex such that $\lng(G')=H+xy$. Since $x,y$ are adjacent in $H+xy$, it follows that there are distinct vertices $u',v',w'\in V(G')$ for which $x=u'v'\in E(G')$ and $y=u'w'\in E(G')$. Since $|E(G')|=|V(H+xy)|=|V(H)|=h$ and $G'$ has no isolated vertex, it follows that $G'$ admits a bipartition with at most $h$ vertices on either side. Therefore, there is a subgraph $G$ of $K_{h,h}$ isomorphic to $G'$ with $a=uv\in E(G)$ and $b=uw\in E(G)$, as well as an isomorphism $f:V(G)\rightarrow V(G')$ between $G$ and $G'$ such that $f(u)=u', f(v)=v'$, and $f(w)=w'$. Define the map $\varphi:E(G)\rightarrow E(G')$ as follows: for every edge $q_1q_2\in E(G)$, let $\varphi(q_1q_2)=f(q_1)f(q_2)\in E(G')$. Then $\varphi$ is an isomorphism between $\lng(G)$ and $\lng(G')$ where $\varphi(a)=x$ and $\varphi(b)=y$. Moreover, note that $\lng(G)$ is the induced subgraph of $\lng(K_{h,h})$ with vertex set $E(G)$. Hence, $\lng(G)-ab$ is an induced subgraph of $\lng(K_{h,h})-ab$ isomorphic to $\lng(G')-xy=H$. This proves \ref{lem:linegraphs}\ref{lem:linegraphs_c}.

 Finally, we prove \ref{lem:linegraphs}\ref{lem:linegraphs_d}. Since $\ol{H}$ is not the line graph of a bipartite graph, it follows that $\lng(K_{h,h})$ is $\ol{H}$-free, and so $\ol{\lng(K_{h,h})}$ is $H$-free. It remains to show that for every edge $ab\in E(\ol{\lng(K_{h,h})})$ (where $a,b\in E(K_{h,h})$), there is an induced subgraph of $\ol{\lng(K_{h,h})}-ab$ that is isomorphic to $H$. Let $a=uv$, and $b=wz$ where $u,v,w,z\in V(K_{h,h})$ are distinct. Since $\ol{H}$ is a biline\lp graph, there exists $xy\in E(\ol{H})$ for which $\ol{H}-xy$ is a line graph. Choose a bipartite graph $G'$ with no isolated vertex such that $\lng(G')=\ol{H}-xy$. Since $x,y$ are not adjacent in $\ol{H}-xy$, it follows that there are distinct vertices $u', v', w', z' \in V(G')$ such that $x=u'v'\in E(G')$ and $y=w'z'\in E(G')$. Since $|E(G')|=|V(\ol{H}-xy)|=|V(H)|=h$ and $G'$ has no isolated vertex, it follows that $G'$ admits a bipartition with at most $h$ vertices on either side. Therefore, there is a subgraph $G$ of $K_{h,h}$ isomorphic to $G'$ with $a=uv\in E(G)$ and $b=wz\in E(G)$, as well as an isomorphism $f:V(G)\rightarrow V(G')$ such that $f(u)=u', f(v)=v', f(w)=w'$, and $f(z)=z'$. Define the map $\varphi:E(G)\rightarrow E(G')$ as follows: for every edge $q_1q_2\in E(G)$, let $\varphi(q_1q_2)=f(q_1)f(q_2)\in E(G')$. Then $\varphi$ is an isomorphism between $\lng(G)$ and $\lng(G')$ where $\varphi(a)=x$ and $\varphi(b)=y$. Moreover, note that $\lng(G)$ is the induced subgraph of $\lng(K_{h,h})$ with vertex set $E(G)$. Hence, $\lng(G)+ab$ is an induced subgraph of $\lng(K_{h,h})+ab$ isomorphic to $\lng(G')+xy=\ol{H}$, which in turn implies that $\ol{\lng(G)+ab}=\ol{\lng(G)}-ab$ is an induced subgraph of $\ol{\lng(K_{h,h})+ab}=\ol{\lng(K_{h,h})}-ab$ isomorphic to $\ol{\lng(G')+xy}=H$. This completes the proof of \ref{lem:linegraphs}\ref{lem:linegraphs_d}.
\end{proof}

Now, of the 
$179-37-32=110$ graphs remaining from Subsections~\ref{subsec:multi} and \ref{subsec:0or1},

\begin{itemize}
    \item The \textbf{34} graphs in Figures~\ref{fig:lineminus1} and \ref{fig:lineminus2} are line\lm graphs, and handled by \Cref{lem:linegraphs}\ref{lem:linegraphs_a}.
    \item  The \textbf{8} graphs in \Cref{fig:lineplus} are complements of line\lp graphs, and so handled by \Cref{lem:linegraphs}\ref{lem:linegraphs_b}.
    \item The \textbf{2} graphs in \Cref{fig:bilineminus} are biline\lm graphs, and so handled by \Cref{lem:linegraphs}\ref{lem:linegraphs_c}.
    
    \item  The \textbf{3} graphs in \Cref{fig:bilineplus} are complements of biline\lp graphs, and so handled by \Cref{lem:linegraphs}\ref{lem:linegraphs_d}.
\end{itemize}

Moreover, the \textbf{38} graphs in \Cref{fig:severed} are line graphs of graphs with a severed pairs, and so by \Cref{thm:mainsevered}, they are all deletion-normal.

\subsection{Sporadic cases}\label{subsec:sporadicsix}

There are now $110-34-8-2-3-38=\textbf{25}$ graphs remaining from Subsection~\ref{subsec:multi}, \ref{subsec:0or1} and \ref{subsec:nearline} that need special treatment. These graphs are entries
$$\text{\rm 3B,\, 3F,\, 5E,\, 5F,\, 5H,\, 7E,\, 9F,\, 9L,\, 10E,\, 10K,\, 11B,\, 12B,\, 12D}$$
$$\text{\rm 12F,\, 12H,\, 12K,\, 13B,\, 13C,\, 13F,\, 13G,\, 13I,\, 15I,\, 15L,\, 16A,\, 16F}$$
in \Cref{fig:allsix}. We handle these 25 graphs in three groups:
\begin{itemize}
    \item \textbf{Group 1.} The \textbf{12} ``easy'' cases; namely:
    $$\text{\rm 3F,\, 5F,\, 9F,\, 11B,\, 12B,\, 12D,\, 13C,\, 13F,\, 13G,\, 13I,\, 15L,\, 16F.}$$
     \item \textbf{Group 2.} Then \textbf{5} cases that use the icosahedron; namely:
    $$\text{\rm 7E,\, 9L,\, 12F,\, 12K,\, 16A.}$$
      \item \textbf{Group 3.} The \textbf{8} ``hard'' cases; namely:
    $$\text{\rm 3B,\, 5E,\, 5H,\, 10E,\, 10K,\, 12H,\, 13B,\, 15I.}$$
\end{itemize}

First, we handle the graphs in \textbf{Group 1}:

\begin{theorem}\label{thm:13K&9F} Entries {\rm 3F}, {\rm 9F}, {\rm 12B} and {\rm 12D} in \Cref{fig:allsix} are deletion-normal.
\end{theorem}

\begin{figure}[t!]
    \centering   \includegraphics[scale=0.7]{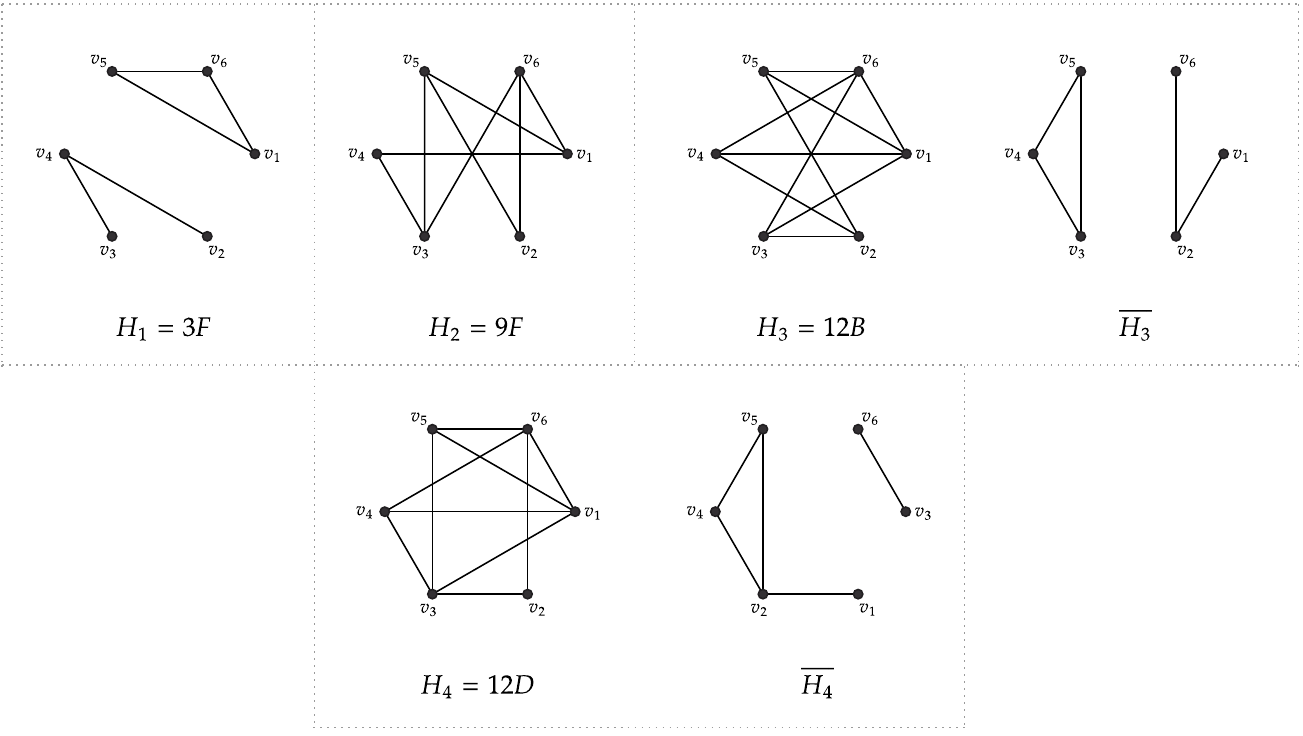}
    \caption{Proof of \Cref{thm:13K&9F}. Graphs $H_1,H_2,H_3, H_4$, along with the complement of $H_3$ and $H_4$.}
    \label{fig:3F9F12B12D}
\end{figure}

\begin{proof}
    Let $H_1$ be entry 3F in \Cref{fig:allsix}, let $H_2$ be entry 9F in \Cref{fig:allsix}, let $H_3$ be entry 12B in \Cref{fig:allsix}, and let $H_4$ be entry 12D in \Cref{fig:allsix}.

    Note that $H_1$ is isomorphic to the graph obtained from $2K_3$ by removing an edge. Since $2K_3$ is edge-transitive, it follows that $2K_{3}$ itself is $H_1$-deletion-saturated. 

   Note that $H_2$ is isomorphic to the graph obtained from $K_{3,3}$ by removing an edge. Since $K_{3,3}$ is edge-transitive, it follows that $K_{3,3}$ itself is $H_2$-deletion-saturated.

   Note also that $\ol{H_3}$ is isomorphic to $K_3+P_3$ and $\ol{H_4}$ is the disjoint union of a two-vertex path and a graph obtained from a triangle by adding a leaf. It is straightforward to observe that $3K_3$ is both $\ol{H_3}$-addition-saturated and $\ol{H_4}$-addition-saturated. Hence, $\ol{3K_3}=K_{3,3,3}$ is both $H_3$-deletion-saturated and $H_4$-deletion-saturated. This completes the proof of \Cref{thm:13K&9F}.
\end{proof}

\begin{theorem}\label{thm:cycles}
    Entries {\rm 5F}, {\rm 11B} and {\rm 15L} in \Cref{fig:allsix} are deletion-normal.
\end{theorem}
\begin{figure}[t!]
    \centering   \includegraphics[scale=0.85]{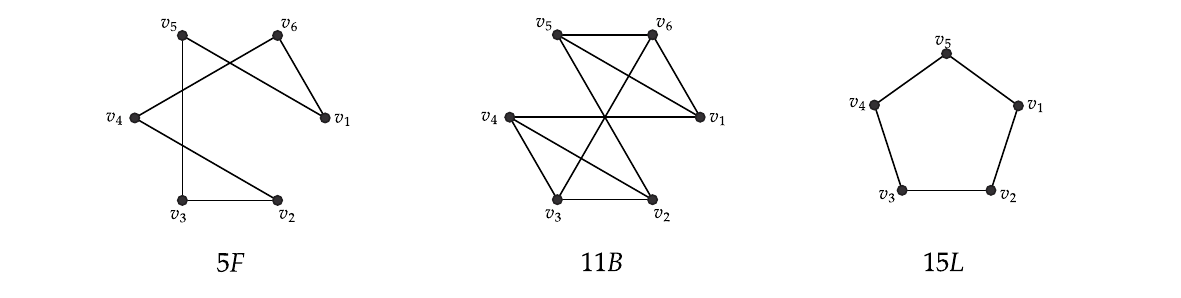}
    \caption{Proof of \Cref{thm:cycles}. The $6$-cycle, the complement of the $6$-cycle, and the $5$-cycles.}
    \label{Cycles6&5&4}
\end{figure}

\begin{proof}
    Note that entry 5F in \Cref{fig:allsix} is a $6$-cycle, entry 11B is the complement of a $6$-cycle, and entry 15L is a $5$-cycle. Recall that all cycles of length at most $10$ (except for the $3$-cycle) are normal (and thus both deletion-normal and addition-normal). This completes the proof of \Cref{thm:cycles}.
\end{proof}

\begin{theorem}\label{thm:13CFGI} Entries {\rm 13C}, {\rm 13F}, {\rm 13G}, {\rm 13I}, and {\rm 16F} in \Cref{fig:allsix} are deletion-normal.
\end{theorem}

\begin{proof}
Let $H_1,H_2,H_3,H_4,$ and $H_5$, respectively, be the graphs at entries 13C, 13F, 13G, 13I, and 16F in \Cref{fig:allsix}. Then,
\begin{itemize}
    \item $\ol{H_1}$ is isomorphic to $P_4+P_2$;
     \item $\ol{H_2}$ is isomorphic to $P_4+2P_1$;
      \item $\ol{H_3}$ is isomorphic to $P_3+P_2+P_1$;
       \item $\ol{H_4}$ is isomorphic to $P_3+3P_1$;
        \item $\ol{H_5}$ is isomorphic to $P_3+2P_1$;
\end{itemize}

Our goal is to show that for each $i\in \{1,\cdots,5\}$, the graph $5P_2$ is $\ol{H_i}$-addition-normal. Note, in particular, that $5P_2$ is $\ol{H_i}$-free because $\ol{H_i}$ has at least one vertex of degree two whereas $5P_2$ is $1$-regular. Also, for every edge $e\in E(\ol{5P_2})$, the graph $5P_2+e$ is isomorphic to $P_4+3P_2$, which is easily seen to contain an induced subgraph isomorphic to $\ol{H_i}$ for every $i\in \{1,\ldots, 5\}$. This completes the proof of \Cref{thm:13CFGI}.   
\end{proof}
\begin{figure}[t!]
    \centering   \includegraphics[width=\linewidth]{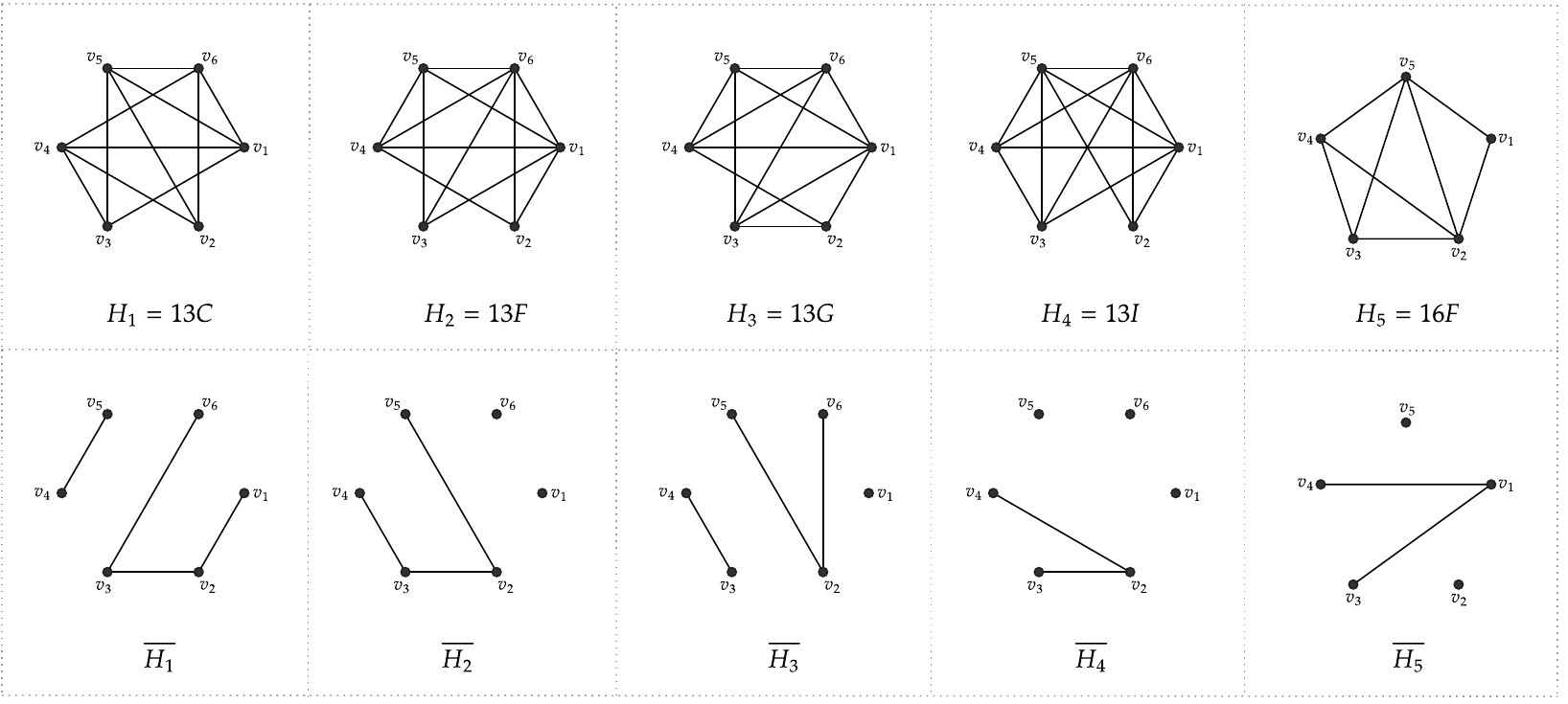}
    \caption{Proof of \Cref{thm:13CFGI}. From left to right, the graphs $H_1,\ldots, H_5$ (top) and their complements (bottom).}
    \label{fig:13CFGI16F}
\end{figure}

Next, we handle the 5 graphs in \textbf{Group 2}:

\begin{figure}[t!]
    \centering   \includegraphics[width=\linewidth]{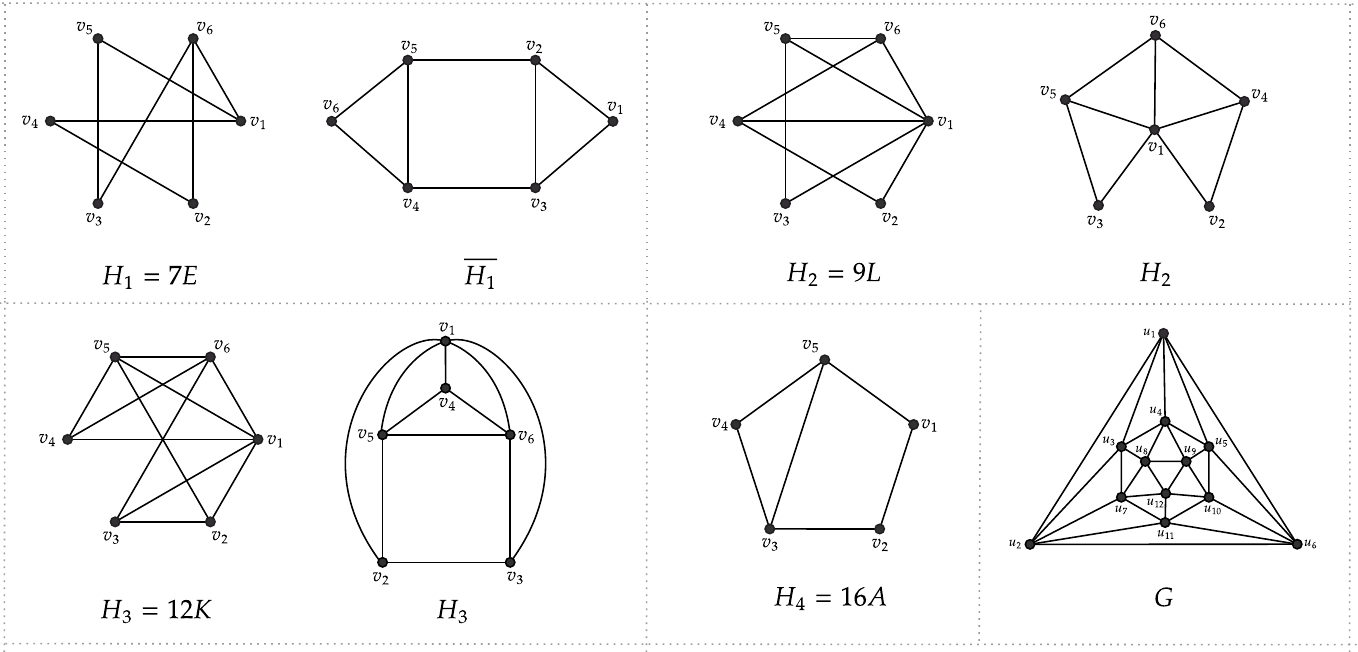}
    \caption{Proof of \Cref{thm:9L}: The graph $H_1$ and its complement, two drawings of $H_2$, two drawings of $H_3$, the graph $H_4$, and the graph $G$.}
    \label{fig:7E9L12K16A}
\end{figure}

\begin{theorem}\label{thm:9L}Entries {\rm 7E}, {\rm 9L}, {\rm 12K} and {\rm 16A} in \Cref{fig:allsix} are all deletion-normal. 
\end{theorem}

\begin{proof}
Let $H_1$ be entry 7E in \Cref{fig:allsix}, let $H_2$ be entry 9L, let $H_3$ be entry 12K, and let $H_4$ be entry 16A; see \Cref{fig:7E9L12K16A}. Let $G$ be an isomorphic copy of the icosahedron with $V(G)=\{u_1,\ldots, u_{12}\}$ and the adjacency as depicted in \Cref{fig:7E9L12K16A}, and let $G^+$ be the graph obtained from $G$ by adding a new vertex $u_{13}$ adjacent to all vertices in $V(G)$. First, we show that $G$ is $\ol{H_1}$-addition-saturated (which would in turn imply that $\ol{G}$ is $H_1$-deletion-saturated). Note that $G$ is $\ol{H_1}$-free because the icosahedron, and thus $G$, is $C_4$-free whereas $\ol{H_1}[v_2,v_3,v_4,v_5]$ is a $4$-cycle. It remains to show that for every edge $e\in E(\ol{G})$, the graph $G+e$ is not $\ol{H_1}$-free. There are now two cases up to symmetry:
\begin{itemize}
    \item $e=u_1u_{11}$. Then, for $S=\{u_1,u_3,u_6,u_7,u_8,u_{11}\}$, it is readily seen that the map $f:V(H_1)\to S$ with 
$$f(v_1)=u_6,\quad f(v_2)=u_1,\quad f(v_3)=u_{11},\quad f(v_4)=u_7,\quad f(v_5)=u_3,\quad f(v_6)=u_8$$
is an isomorphism between $\ol{H_1}$ and $G[S]+e$.
 \item $e=u_1u_{12}$. Then, for $S=\{u_1,u_3,u_4,u_7,u_{11},u_{12}\}$, it is readily seen that the map $f:V(H_1)\to S$ with 
$$f(v_1)=u_4,\quad f(v_2)=u_1,\quad f(v_3)=u_{3},\quad f(v_4)=u_7,\quad f(v_5)=u_{12},\quad f(v_6)=u_{11}$$
is an isomorphism between $\ol{H_1}$ and $G[S]+e$.
\end{itemize}

Second, we show that $G$ is $H_2$-deletion-saturated. Note that $G$ is $H_2$-free because the neighborhood of every vertex in the icosahedron, and so in $G$, induces a $5$-cycle, whereas $H_2[N_{H_2}(v_1)]$ is a five-vertex path. Since the icosahedron, and thus $G$, is edge-transitive, it remains to show that for some edge $e\in E(G)$, the graph $G-e$ is not $H_2$-free. We pick $e=u_2u_3$. Then, for $S=\{u_1,\ldots, u_6\}$, it is readily seen that the map $f:V(H_2)\to S$ with 
$$f(v_i)=u_i\quad \text{for all }i\in\{1,\ldots, 6\}$$
is an isomorphism between $H_2$ and $G[S]-e$.

 Third, we show that $G^+$ is $H_3$-deletion-saturated. Since $G$ is $C_4$-free, and since $C_4$ has no dominating vertex, it follows that $G^+$ is $C_4$-free, too. This implies that $G^+$ is $H_3$-free because $H_3[\{v_2,v_3,v_5,v_6\}]$ is a $4$-cycle. It remains to show that for every $e\in E(G^+)$, the graph $G^+-e$ is not $H_3$-free. Since the icosahedron, and thus $G$, is edge-transitive, there are now two cases up to symmetry:
 \begin{itemize}
      \item $e=u_1u_{3}$. Then, for $S=\{u_1,u_2,u_3,u_4,u_5,u_{13}\}$, it is readily seen that the map $f:V(H_3)\to S$ with 
$$f(v_1)=u_{13},\quad f(v_2)=u_3,\quad f(v_3)=u_{2},\quad f(v_4)=u_5,\quad f(v_5)=u_4,\quad f(v_6)=u_1$$
is an isomorphism between $H_3$ and $G[S]-e$.
  \item $e=u_1u_{13}$. Then, for $S=\{u_1,u_3,u_4,u_5,u_8,u_{13}\}$, it is readily seen that the map $f:V(H_3)\to S$ with 
$$f(v_1)=u_{4},\quad f(v_2)=u_1,\quad f(v_3)=u_{5},\quad f(v_4)=u_8,\quad f(v_5)=u_3,\quad f(v_6)=u_{13}$$
is an isomorphism between $H_3$ and $G[S]-e$.
 \end{itemize}
 Finally, note that $\ol{H_4}$ is isomorphic to $P_5$. By \Cref{thm:P5}, $\ol{G}$ is  $\ol{H_4}$-induced-saturated, and thus $G$ is $H_4$-induced-saturated. This completes the proof of \Cref{thm:9L}.
\end{proof}

\begin{theorem}\label{thm:12F}  Entry ${\rm 12F}$ in \Cref{fig:allsix} is deletion-normal.
\end{theorem}

\begin{proof}
  Let $H$ be entry 12F in \Cref{fig:allsix}. Construct a graph $G$ as follows. Let $G_1$ and $G_2$ be two disjoint and isomorphic copies of the icosahedron, such that $V(G_1)=\{u^1_i:i\in \{1,\ldots, 12\}\}$, $ V(G_2)=\{u^2_i:i\in \{1,\ldots, 12\}\}$, and $f: u^1_i\mapsto u^2_i$ is an isomorphism between $G_1$ and $G_2$. Let $G$ be the graph with
 $$V(G)=\{u^1_i,u^2_i: i\in\{1,\ldots, 12\}\}$$ 
 $$E(G)=E(G_1)\cup E(G_2)\cup \{u^1_iu^2_i: i\in\{1,\ldots, 12\}\}\cup \{u^1_iu^2_j: i,j\in\{1,\ldots, 12\}, \, u^1_iu^1_j\in E(G_1)\}.$$
 See \Cref{fig:12F}.

Our goal is to prove that $G$ is $H$-deletion-saturated. 
 
 \sta{\label{st:12ffree} $G$ is $H$-free.}
 
 Suppose for a contradiction that $G$ has an induced subgraph $H'$ isomorphic to $H$. Let $f:V(H)\rightarrow V(H')$ be an isomorphism between $H$ and $H'$ such that $f(v_1)=u^k_i$ and $f(v_2)=u^l_j$ for some $k,l\in\{1,2\}$ and $i,j\in \{1,\ldots, 12\}$. Since $v_1v_2\notin E(H)$, the vertices $u^k_i$ and $u^l_j$ are non-adjacent in $G$.  Let $I=N_G(u^k_i)\cap N_G(u^l_j)$. Since $H[N_H(v_1)
 \cap N_H(v_2)]$ is isomorphic to $2K_2$, it follows that $G[I]$ has an induced subgraph isomorphic to $2K_2$. On the other hand, since $u^k_i$ and $u^l_j$ are non-adjacent in $G$, by the construction of $G$, we have $I=(N_{G_1}(u^1_i)\cap N_{G_1}(u^1_j))\cup (N_{G_2}(u^2_i)\cap N_{G_2}(u^2_j))$. Since $G_1$ and $G_2$ are both isomorphic to the icosahedron, it follows that either $N_{G_1}(u^1_i)\cap N_{G_1}(u^1_j)=N_{G_2}(u^2_i)\cap N_{G_2}(u^2_j)=\varnothing$, or both $N_{G_1}(u^1_i)\cap N_{G_1}(u^1_j)$ and $N_{G_2}(u^2_i)\cap N_{G_2}(u^2_j)$ are $2$-vertex cliques in $G$. Thus, by the construction of $G$, either $I=\varnothing$ or $I$ is a $4$-vertex clique in $G$. But then $G[I]$ is $2K_2$-free, a contradiction. This proves \eqref{st:12ffree}.
\medskip

\sta{\label{st:12fcritical} For every $e\in E(G)$, the graph $G-e$ has an induced subgraph isomorphic to $\ol{H}$.}

Let $e=u^k_iu^l_j\in E(G)$, where $k,l\in\{1,2\}$ and $i,j\in \{1,\ldots, 12\}$. Once again, let $I=N_G(u^k_i)\cap N_G(u^l_j)$. By the choice of $H$, it suffices to show that there exists $S\subseteq I$ such that $G[S]$ is isomorphic to $2K_2$ (because then the subgraph of $G-e$ induced by $S\cup \{u^k_i,u^l_j\}$ is isomorphic to $H$). By the construction of $G$, there are now two cases:
\begin{itemize}
    \item $i=j$ and $k\neq l$. So we have $e=u^1_iu^2_i$. Since $G_1$ is isomorphic to the icosahedron, $G_1[N_{G_1}(u^1_i)]$ is $5$-cycle, and therefore we may choose two non-adjacent vertices $u^1_s,u^1_t\in N_{G_1}(u^1_i)$ where $s,t\in \{1,\ldots, 12\}\setminus \{i\}$. But now $S=\{u^1_s,u^2_s,u^1_t,u^2_t\}\subseteq I$, and $E(G[S])=\{u^1_su^2_s,u^1_tu^2_t\}$; that is, $G[S]$ is isomorphic to $2K_2$.
    
    \item $i\neq j$ and $u^1_iu^1_j \in E(G_1)$. Since $G_1$ is isomorphic to the icosahedron, it follows that $N_{G_1}(u^1_i)\cap N_{G_1}(u^1_j)$ is a stable set on two vertices; say $N_{G_1}(u^1_i)\cap N_{G_1}(u^1_j)=\{u^1_s,u^1_t\}$. But now $S=\{u^1_s,u^2_s,u^1_t,u^2_t\}\subseteq I$, and $E(G[S])=\{u^1_su^2_s,u^1_tu^2_t\}$; that is, $G[S]$ is isomorphic to $2K_2$.
\end{itemize}
This proves \eqref{st:12fcritical}.
\medskip

The result is now immediate from \eqref{st:12ffree} and \eqref{st:12fcritical}. This completes the proof of \Cref{thm:12F}.
\end{proof}
\begin{figure}[t!]
    \centering   \includegraphics[scale=0.7]{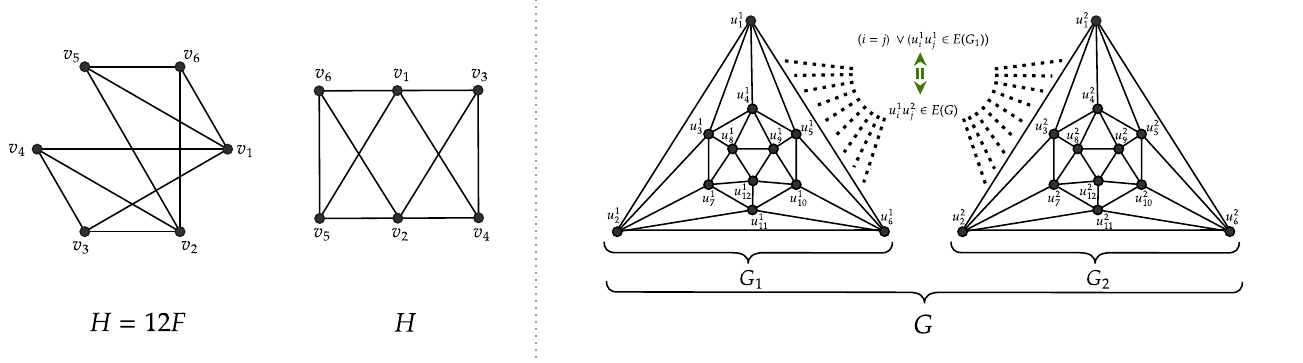}
    \caption{Proof of \Cref{thm:12F}: Two drawings of the graph $H$ (left) and the graph $G$ (right).}
    \label{fig:12F}
\end{figure}

At long last, we are now left with the 8 graphs in \textbf{Group 3}, for which the proof of deletion-normality requires completely different methods.

\begin{theorem}\label{thm:3B} Entry {\rm 3B} in \Cref{fig:allsix} is deletion-normal.
\end{theorem}
\begin{figure}[t!]
    \centering   \includegraphics[scale=0.7]{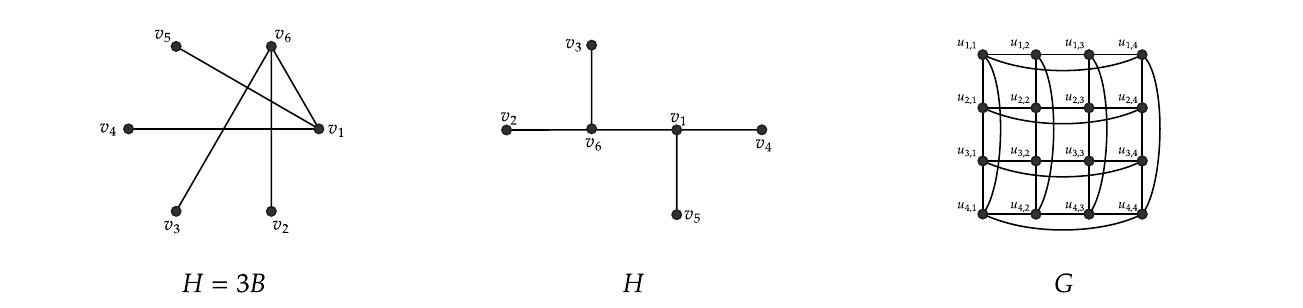}
    \caption{Proof of \Cref{thm:3B}. From left to right: Two drawings of the graph $H$, and the graph $G$.}
    \label{fig:3B}
\end{figure}

\begin{proof}
Let $H$ be entry 3B in \Cref{fig:allsix}. Let $G$ be the graph whose vertex set is $V(G)=\{u_{i,j}: 1\leq i,j\leq 4\}$ and whose edge set is 
$$E(G)=\left(\bigcup_{i=1}^4\{u_{i,1}u_{i,2},u_{i,2}u_{i,3},u_{i,3}u_{i,4},u_{i,4}u_{i,1}\}\right)\cup\left(\bigcup_{j=1}^4\{u_{1,j}u_{2,j},u_{2,j}u_{3,j},u_{3,j}u_{4,j},u_{4,j}u_{1,j}\}\right).$$
See \Cref{fig:3B}. In particular, $G$ is the Cartesian product of two $4$-cycles. 

We prove that $G$ is $H$-deletion-saturated. First, we show that $G$ is $H$-free. Suppose not.  According to the labeling given in \Cref{fig:3B}, we have $(N_H(v_1)\cup N_H(v_6))\setminus \{v_1,v_6\}=\{v_2,v_3,v_4,v_5\}$, which is a four-vertex stable set in $H$. Therefore, since $G$ has an induced subgraph isomorphic to $H$, it follows that there is an edge $uu'\in E(G)$ for which $G[(N_G(u)\cup N_G(u'))\setminus \{u,u'\}]$ contains a four-vertex stable set. However, since $G$ is edge-transitive, it is straightforward to check that for every edge $uu'\in E(G)$, the graph $G[(N_G(u)\cup N_G(u'))\setminus \{u,u'\}]$ is indeed isomorphic to $3K_2$, which does not contain a four-vertex stable set, a contradiction. We deduce that $G$ is $H$-free.
\medskip

It remains to show that for every $e\in E(G)$, the graph $G-e$ has an induced subgraph isomorphic to $H$. Since $G$ is edge-transitive, it suffices to prove the latter statement for only one edge $e\in E(G)$. Let $e=u_{1,1}u_{1,2}$. Let $S=\{u_{1,1},u_{1,2},u_{1,3},u_{1,4},u_{2,3},u_{4,4}\}$. Then one can verify that the map $f:V(H)\rightarrow S$ with
$$f(v_1)=u_{1,3},\quad  f(v_2)=u_{1,1},\quad  f(v_3)=u_{4,4},\quad  f(v_4)=u_{1,2},\quad  f(v_5)=u_{2,3}, \quad f(v_6)=u_{1,4}$$
is an isomorphism between $H$ and $G[S]-e$. This completes the proof of \Cref{thm:3B}.
\end{proof}

\begin{theorem}\label{thm:5E} Entry {\rm 5E} in \Cref{fig:allsix} is deletion-normal.
\end{theorem}
\begin{figure}[t!]
    \centering   \includegraphics[scale=0.7]{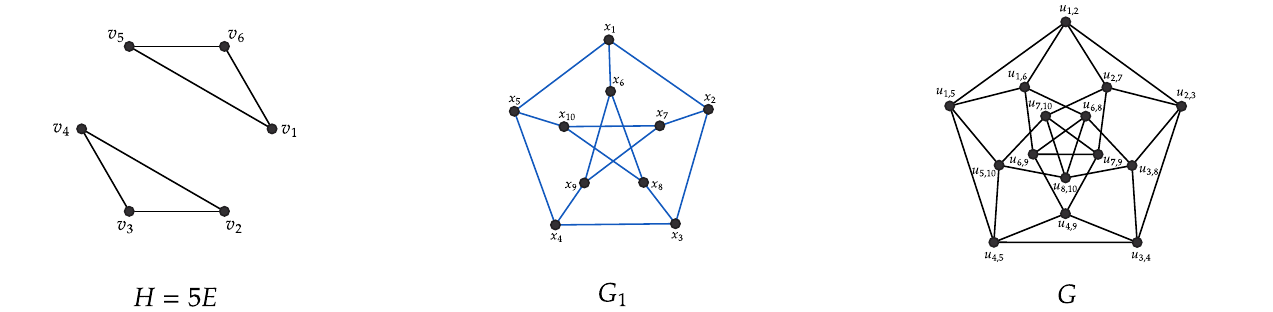}
    \caption{Proof of \Cref{thm:5E}. From left to right: Graphs $H, G_1$ and $G$.}
    \label{fig:5E}
\end{figure}

\begin{proof}
Let $H$ be entry 5E in \Cref{fig:allsix}. Then $H$ is isomorphic to $2K_3$. Let $G_1$ be the Petersen graph with $V(G_1)=\{x_1,\ldots, x_{10}\}$ and the adjacency as depicted in \Cref{fig:5E}, and let $G=\lng(G_1)$. For each edge $e=x_ix_j\in E(G_1)$, $i,j\in \{1,\ldots, 10\}$, we label the vertex of $G$ corresponding to $e$ with $u_{i,j}$; see \Cref{fig:5E}.

Our goal is to show that $G$ is $2K_3$-deletion-saturated. Note that the only graphs of which the line graph is isomorphic to $2K_3$ are $2K_3, K_3+K_{1,3}$ and $2K_{1,3}$. Since $G_1$ is triangle-free and $3$-regular, and every two non-adjacent vertices in $G_1$ have exactly one common neighbor, $G_1$ has no subgraph isomorphic to $H'$ for any $H'\in \{2K_3, K_3+K_{1,3},2K_{1,3}\}$, and thus $G$ is $2K_3$-free.

It remains to show that for every $e\in E(G)$, the graph $G-e$ has an induced subgraph isomorphic to $2K_3$. By symmetry, we may assume that $e=u_{1,2}u_{2,3}\in E(G)$. Let $T_1=\{u_{1,2},u_{1,5},u_{1,6}\}$ and let $T_2=\{u_{2,3},u_{3,4},u_{3,8}\}$. Then $T_1$ and $T_2$ are disjoint triangles in $G$, and $e$ is the only edge of $G$ with an end in $T_1$ and an end in $T_2$. Hence $G[T_1\cup T_2]-e$ is isomorphic to $2K_3$. This completes the proof of \Cref{thm:5E}. 
\end{proof}

\begin{theorem}\label{thm:5H}Entry {\rm 5H} in \Cref{fig:allsix} is deletion-normal.
\end{theorem}
\begin{figure}[t!]
    \centering   \includegraphics[scale=0.7]{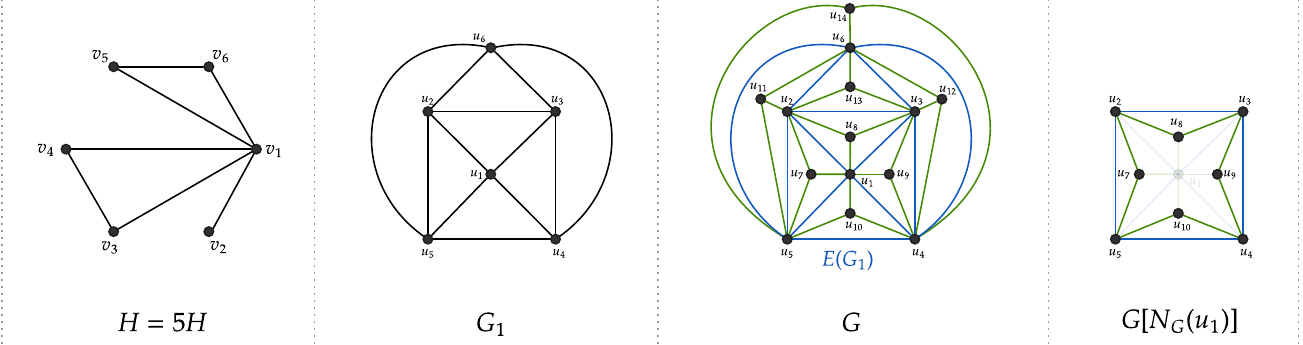}
    \caption{Proof of \Cref{thm:5H}. From left to right: Graphs $H$, $G_1$ and $G$, and the subgraph of $G$ induced by the neighborhood of $u_1$.}
    \label{fig:5H}
\end{figure}
\begin{proof}
Let $H$ be entry 5H in \Cref{fig:allsix}. Construct a graph $G$ as follows. Let $G_1$ be an isomorphic copy of $\lng(K_4)$ with $V(G_1)=\{u_1,\ldots, u_6\}$ and the adjacency as depicted in \Cref{fig:5H}. Note that $G_1$ has exactly eight triangles. Let $G$ be the graph obtained from $G_1$ by adding eight pairwise adjacent new vertices $\{u_7,\ldots, u_{14}\}$ -- one for each triangle of $G_1$ -- and making each vertex in $\{u_7,\ldots, u_{14}\}$ adjacent to the three vertices in its corresponding triangle of $G_1$, and to no other vertex in $G$, with the adjacency as depicted in \Cref{fig:5H}. (Note, in particular, that $\{u_7,\ldots, u_{14}\}$ remains a stable set in $G$.) 
It is straightforward to observe that the natural action of the automorphism group of $G$ on $V(G)$ has exactly two orbits, namely $V(G_1)=\{u_1,\ldots, u_6\}$ and $V(G)\setminus V(G_1)=\{u_7,\ldots, u_{14}\}$. Similarly, the natural action of the automorphism group of $G$ on $E(G)$ has exactly two orbits, namely $E(G_1)$ and $E(G)\setminus E(G_1)$. 

Our goal is to show that $G$ is $H$-deletion-saturated. First, we show that $G$ is $H$-free. Suppose for a contradiction that $G$ has an induced subgraph $H'$ isomorphic to $H$ with an isomorphism $f:V(H)\to V(H')\subseteq V(G)$. Let $u=f(v_1)$. Since $N_H(v_1)$ is isomorphic to $2K_2+K_1$, it follows that $G[N_G(u)]$ has an induced subgraph isomorphic to $2K_2+K_1$. In particular, since the neighborhood of every vertex $u_i$ for $i\in \{7,\ldots, 14\}$ is a triangle in $G$, it follows that $u=u_i$ for some $i\in \{1,\ldots, 6\}$. By symmetry, we may assume that $u=u_1$. But then 
$$G[N_G(u)]= G[N_G(u_1)]=G[\{u_2,u_3,u_4,u_5,u_7,u_8,u_9,u_{10}\}]$$
is the graph on the right in \Cref{fig:5H}, which is easily observed to be $(2K_2+K_1)$-free, a contradiction. We deduce that $G$ is $H$-free.
\medskip

It remains to show that for every $e\in E(G)$, the graph $G-e$ has an induced subgraph isomorphic to $H$. There are now two cases:
\begin{itemize}
    \item $e\in E(G_1)$. We may assume, by symmetry, that $e=u_4u_5$. Let 
    $$S=\{u_1,u_4,u_5,u_7,u_8,u_9\}.$$ Then one can check that the map $f:V(H)\rightarrow S$ with
$$f(v_1)=u_{1},\quad  f(v_2)=u_{8},\quad  f(v_3)=u_{9},\quad  f(v_4)=u_{4},\quad  f(v_5)=u_{5}, \quad f(v_6)=u_{7}$$
is an isomorphism between $H$ and $G[S]-e$.
\item  $e\in E(G)\setminus E(G_1)$. By symmetry, we may assume that $e=u_2u_7$. Let 
$$S=\{u_1,u_2,u_4,u_7,u_8, u_9\}.$$ Then one can check that the map $f:V(H)\rightarrow S$ with
$$f(v_1)=u_{1},\quad  f(v_2)=u_{7},\quad  f(v_3)=u_{2},\quad  f(v_4)=u_{8},\quad  f(v_5)=u_{9}, \quad f(v_6)=u_{4}$$
is an isomorphism between $H$ and $G[S]-e$.
\end{itemize}
This completes the proof of \Cref{thm:5H}.
\end{proof}

\begin{theorem}\label{thm:10E} Entry {\rm 10E} in \Cref{fig:allsix} is deletion-normal.
\end{theorem}
\begin{figure}[t!]
    \centering   \includegraphics[scale=0.7]{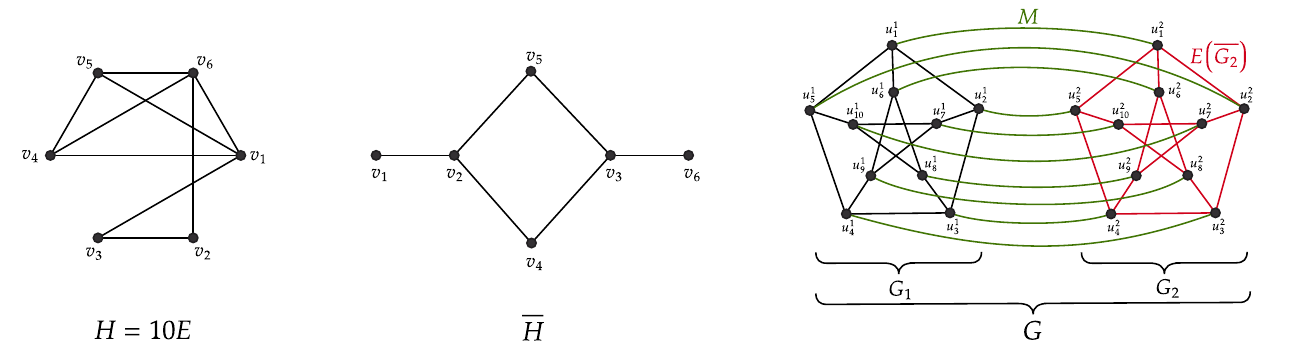}
    \caption{Proof of \Cref{thm:10E}. From left to right: The graph $H$ and its complement, and the graph $G$ with a depiction of $E(G_1)$, $E(\ol{G_2})$, and $M$.}
    \label{fig:10E}
\end{figure}

\begin{proof}
Let $H$ be entry 10E in \Cref{fig:10E}. Construct a graph $G$ as follows. Let $G_1$ and $G_2$ be disjoint graphs such that $G_1$ is isomorphic to the Petersen graph with $V(G_1)=\{u^1_1, \dots, u^1_{10}\}$ and $E(G_1)$ as shown in \Cref{fig:10E}, and $\ol{G_2}$ is isomorphic to the Petersen graph with $V(G_2)= \{u^2_1, \dots, u^2_{10}\}$ and $E(\ol{G_2})$ as shown in \Cref{fig:10E}. Let $G$ be the graph obtained from the (disjoint) union of $G_1$ and $G_2$ by adding the matching $M=\{u^1_iu^2_i : i\in \{1,\ldots, 10\}\}$. See \Cref{fig:10E}.

To prove \Cref{thm:10E}, it suffices to show that $G$ is $\ol{H}$-addition-saturated.

\sta{\label{st:10efree}$G$ is $\ol{H}$-free.}

Suppose for a contradiction that $G$ has an induced subgraph $H'$ isomorphic to $\ol{H}$ where $V(H')=\{v'_1,\ldots, v'_6\}\subseteq V(G)$, and $f:v_i\mapsto v'_i$ is an isomorphism between $\ol{H}$ and $H'$. Let $I'=\{v'_2,v'_3,v'_4,v'_5\}$. Then $G[I']$ is isomorphic to $C_4$. Assume that $I'\cap V(G_1)\neq \varnothing$. Since $G_1$ is $C_4$-free, it follows that $I'\cap V(G_2)\neq \varnothing$. Since, in the graph $G$, every vertex in $V(G_1)$ has exactly one neighbor in $V(G_2)$ (and vice versa), it follows that $|I'\cap V(G_1)|=|I'\cap V(G_2)|=2$ and $|E(G[I'])\cap M|=2$. In particular, we have $I'\cap V(G_1)=\{u^1_i,u^1_j\}$ and $I'\cap V(G_2)=\{u^2_i,u^2_j\}$
for some $i,j\in \{1,\ldots, 10\}$. But then exactly one of $u^1_iu^1_j$ and $u^2_iu^2_j$ is an edge in $G$, and so $|E(G[I'])|=3$, a contradiction. 

We deduce that $I'\cap V(G_1)=\varnothing$. Thus, we may write 
$$v'_2=u^2_{i_1},\quad  v'_4=u^2_{i_2},\quad  v'_3=u^2_{i_3},\quad  v'_5=u^2_{i_4}$$
for some distinct $i_1,i_2,i_3,i_4\in \{1,\ldots, 10\}$. In particular, we have
$$E(G[I'])=E(G_2[I'])=\{u^2_{i_1}u^2_{i_{2}}, u^2_{i_2}u^2_{i_{3}}, u^2_{i_3}u^2_{i_{4}}, u^2_{i_4}u^2_{i_1}\}.$$ 
There are now two cases:
\begin{itemize}
    \item $\{v'_1,v'_6\}\subseteq V(G_1)$. Then $v'_1=u^1_{i_1}$ and $v'_6=u^1_{i_3}$. But now since $u^2_{i_1}u^2_{i_3}\notin E(G_2)$, it follows that $v'_1v'_6=u^1_{i_1}u^1_{i_3}\in E(G_1)\subseteq E(G)$, which in turn implies that $v'_1v'_6\in E(H')$, a contradiction. 
    \item $\{v'_1,v'_6\}\cap V(G_2)\neq \varnothing$. By symmetry, we may assume that $v'_1\in V(G_2)$. It follows that $\{v'_1,u^2_{i_2},u^2_{i_4}\}$ is a stable set of cardinality three in $G_2$. But then $\ol{G_2}$, and thus the Petersen graph, contains a triangle, a contradiction.
\end{itemize}
This proves \eqref{st:10efree}.

\sta{\label{st:10ecritical}For every $e\in E(\ol{G})$, the graph $G+e$ has an induced subgraph isomorphic to $\ol{H}$.}

It suffices to show that there exists $S\subseteq V(G)$ containing both endpoints of $e$ such that $G[S]+e$ is isomorphic to $\ol{H}$. There are now three cases:

\begin{itemize}
    \item $e=u^1_{i_1}u^1_{i_2}$ for some $i_1,i_2\in \{1,\ldots, 10\}$. 
Then, it is readily observed that, since $G_1$ is isomorphic to the Petersen graph, there is a path $u^1_{i_1}\dd u^1_{i_2}\dd u^1_{i_3}\dd u^1_{i_4}$ of length three in $G_1$ from $u^1_{i_1}$ to $u^1_{i_2}$. Thus, $C=u^1_{i_1}\dd u^1_{i_2}\dd u^1_{i_3}\dd u^1_{i_4}\dd u^1_{i_1}$ is an induced $4$-cycle in $G+e$. Let $N_{G_1}(u^1_{i_3})\setminus\{u^1_{i_2},u^1_{i_4}\}=\{u^1_{i_5}\}$, where $i_5\in \{1,\ldots, 10\}\setminus \{i_1, i_2, i_3, i_4\}$. Since the Petersen graph, and so $G_1$, has girth five, it follows that $u^1_{i_5}$ is anticomplete to $\{u^1_{i_1},u^1_{i_2},u^1_{i_4}\}$ in $G_1$. Let $S=\{u^1_{i_1}, u^1_{i_2}, u^1_{i_3}, u^1_{i_4}, u^1_{i_5}, u^2_{i_1}\}$. Then, it is easy to check that the map $f:V(\ol{H})\rightarrow S$ with 
$$f(v_1)=u^2_{i_1},\quad  f(v_2)=u^1_{i_1},\quad  f(v_3)=u^1_{i_3},\quad  f(v_4)=u^1_{i_2},\quad  f(v_5)=u^1_{i_4},\quad f(v_6)=u^1_{i_5}$$
is an isomorphism between $\ol{H}$ and $G[S]+e$.

\item $e=u^1_{i_1}u^2_{i_2}$ for some $i_1,i_2\in \{1,\ldots, 10\}$. Since $G_1$ has girth five, every two vertices in $G_1$ have at most one common neighbor. Thus, since $u^1_{i_2}$ has degree three in $G_1$, we may choose a vertex $u^1_{i_3}\in N_{G_1}(u^1_{i_2})\setminus N_{G_1}[u^1_{i_1}]$ where $i_3\in \{1,\ldots, 10\}\setminus \{i_1,i_2\}$. Also, it is easy to check that in $G_1$ (as in the Petersen graph), for every edge $e'\in E(G_1)$, there is a (unique) induced subgraph $M_{e'}$ of $G_1$ isomorphic to $3K_2$ where $e'\in E(M_{e'})$.  In particular, there are $i_4,i_5\in \{1,\ldots, 10\}\setminus \{i_1,i_2,i_3\}$ for which $u^1_{i_4}u^1_{i_5}\in E(M_{u^1_{i_2}u^1_{i_3}})$. It follows that $C=u^2_{i_2}\dd u^2_{i_4}\dd u^2_{i_3}\dd u^2_{i_5}\dd u^2_{i_2}$ is an induced $4$-cycle in $G_2$ (and so in $G$). Let $S=\{u^1_{i_1}, u^2_{i_2}, u^2_{i_3}, u^2_{i_4}, u^2_{i_5}, u^1_{i_3}\}$. Then, it is easy to check that the map $f:V(\ol{H})\rightarrow S$ with
$$f(v_1)=u^1_{i_1},\quad  f(v_2)=u^2_{i_2},\quad  f(v_3)=u^2_{i_3},\quad  f(v_4)=u^2_{i_4},\quad  f(v_5)=u^2_{i_5},\quad f(v_6)=u^1_{i_3}$$
is an isomorphism between $\ol{H}$ and $G[S]+e$.

\item $e=u^2_{i_1}u^2_{i_2}$ for some $i_1,i_2\in \{1,\ldots, 10\}$. Then we have $u^1_{i_1}u^1_{i_2}\in E(G_1)$, and so $C=u^1_{i_1}\dd u^2_{i_1}\dd u^2_{i_2}\dd u^1_{i_2}\dd u^1_{i_1}$ is an induced $4$-cycle in $G+e$. Let $N_{G_1}(u^1_{i_1})=\{u^1_{i_2}, u^1_{i_3}, u^1_{i_4}\}$ where $i_3,i_4\in \{1,\ldots, 10\}\setminus \{i_1,i_2\}$. Since $N_{G_1}(u^1_{i_1})$ is a stable set in $G_1$, it follows that $u^2_{i_2}u^2_{i_4}\in E(G_2)$.  Let $S=\{u^1_{i_1}, u^1_{i_2}, u^1_{i_3}, u^2_{i_1}, u^2_{i_2}, u^2_{i_4}\}$. Then, it is easy to check that the map $f:V(\ol{H})\rightarrow S$ with 
$$f(v_1)=u^1_{i_3},\quad  f(v_2)=u^1_{i_1},\quad  f(v_3)=u^2_{i_2},\quad  f(v_4)=u^2_{i_1},\quad  f(v_5)=u^1_{i_2},\quad f(v_6)=u^2_{i_4}$$
is an isomorphism between $\ol{H}$ and $G[S]+e$.
\end{itemize}
 This proves \eqref{st:10ecritical}.
\medskip

Now, the result follows from \eqref{st:10efree} and \eqref{st:10ecritical}. This completes the proof of \Cref{thm:10E}.

\end{proof}

\begin{theorem}\label{thm:10K} Entry {\rm 10K} in \Cref{fig:allsix} is deletion-normal.
\end{theorem}
\begin{figure}[t!]
    \centering   \includegraphics[scale=0.7]{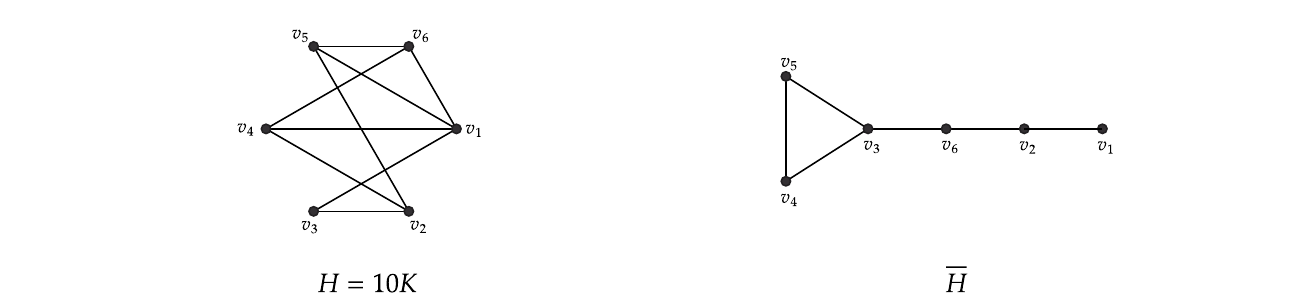}
    \caption{Proof of \Cref{thm:10K}. The graph $H$ (left) and its complement (right).}
    \label{fig:10K}
\end{figure}

\begin{proof}
Let $H$ be entry 10K in \Cref{fig:allsix}. We will show that $\ol{H}$ is addition-normal (see \Cref{fig:10K}). 

By \Cref{thm:hypergirthchi} applied to $c=5$, $g=3$ and $k=2$, there exists a triangle-free graph of chromatic number larger than $5$. Let $G$ be a triangle-free graph with $\chi(G)>5$ and with $E(G)$ maximal with respect to inclusion. It follows in particular that $G$ is connected. 

Our goal is to show that $G$ is $\ol{H}$-addition-saturated\footnote{We remark that the idea in the proof of \Cref{thm:10K} can be adapted to prove the addition-normality of any graph that is obtained from a triangle and path on two or more vertices by identifying an end of the path with a vertex of the triangle.}. In fact, $G$ is $\ol{H}$-free because $G$ is triangle-free whereas $\ol{H}$ is not. It remains to show that for every $e\in E(\ol{G})$, the graph $G^+=G+e$ has an induced subgraph isomorphic to $\ol{H}$.

By the maximality of $E(G)$, there is a triangle $T$ in $G^+$ containing both ends of $e$, and in fact every triangle in $G^+$ contains both ends of $e$. Let $N$ be the set of all vertices of $G^+$ with at least one neighbor in $T$ in $G^+$. Then $T\subseteq N$, and since $G$ is triangle-free, we have $\chi(G^+[N])\leq 3$. Thus, $\chi(G^+\setminus N)\geq \chi(G)-3>2$.

Let $C_1$ be the component of $G^+\setminus N$ with maximum chromatic number. Then $\chi(C_1)>2$, and since $G$ is connected, some vertex in $V(C_1)$ has a neighbor $x_0\in N\setminus T$ in $G^+$. Let $N_0=N_{G^+}(x_0)\cap V(C_1)$. Then $N_0$ is a stable set in $C_1$, and so $\chi(C_1\setminus N_0)\geq \chi(C_1)-1>1$.

Let $C_2$ be the component of $C_1\setminus N_0$ with maximum chromatic number. Then $\chi(C_2)>1$, and since $C_1$ is connected, some vertex in $V(C_2)$ has a neighbor $x_1\in N_0$ in $C_1$. Let $N_1=N_{G^+}(x_1)\cap V(C_2)$. Then $N_1$ is a stable set in $C_2$, and so $\chi(C_2\setminus N_1)\geq \chi(C_2)-1>0$. In particular, $V(C_2)\setminus N_1\neq \varnothing$, and since $C_2$ is connected, it follows that some vertex $x_3\in V(C_2)\setminus N_1$ has a neighbor $x_2\in N_1$ in $C_2$.

We deduce that $x_0\dd x_1\dd x_2\dd x_3$ is a four-vertex path in $G^+$ and $x_0$ is the only vertex of this path with a neighbor in $T$ in $G^+$. If $x_0$ has exactly one neighbor in $T$, then $G^+[T\cup \{x_0,x_1,x_2\}]$ is isomorphic to $\ol{H}$. If $x_0$ has two distinct neighbors $y,y'\in T$, then $G^+[\{y,y',x_0,x_1,x_2,x_3\}]$ is isomorphic to $\ol{H}$. This completes the proof of \Cref{thm:10K}.
\end{proof}

\begin{theorem}\label{thm:12H} Entry ${\rm 12H}$ in \Cref{fig:allsix} is deletion-normal.
\end{theorem}
\begin{proof}
    Let $H$ be entry 12H in \Cref{fig:allsix} (see also \Cref{fig:schlafli}). Let $A=\{1,2,3\}\times \{1,2,3\}$. Let $G$ be the graph with 
    $V(G)=\{u^1_{i,j}, u^2_{i,j},u^3_{i,j}: (i,j)\in A\}$
    and such that:
    \begin{itemize}
        \item For each $k\in \{1,2,3\}$ and all distinct $(i,j),(i',j')\in A$, we have $u^k_{i,j}u^k_{i',j'}\in E(G)$ if and only if $i=i'$ or $j=j'$.
        \item For every $(k,l)\in \{(1,2),(2,3),(3,1)\}$ and all distinct  $(i,j),(i',j')\in A$, we have $u^k_{i,j}u^l_{i',j'}\in E(G)$ if and only if $j\neq i'$.
    \end{itemize}
    See \Cref{fig:schlafli}. 
    
    This $16$-regular $27$-vertex graph, known as the \textit{Schl\"afli graph}, is both vertex-transitive and edge-transitive, and much more. For an integer $k\in \poi$, we say that a graph $F$ is \textit{$k$-ultrahomogeneous} if for every two subsets $X_1,X_2$ of $V(F)$ with $|X_1|=|X_2|\leq k$ such that $F[X_1]$ and $F[X_2]$ are isomorphic, and every isomorphism $\phi:X_1\to X_2$ between $F[X_1]$ and $F[X_2]$, there is an automorphism $f$ of $F$ such that $f|_{X_1}=\phi$. For instance, $F$ is $1$-ultrahomogeneous if and only if $F$ is vertex-transitive, and $2$-ultrahomogeneous if and only if $F$ is arc-transitive. It is known \cite{cameron, devillers} that the Schl\"afli graph is $4$-ultrahomogeneous.\footnote{In fact, something quite curious is true here: every graph that is $5$-ultrahomogeneous is also $k$-ultrahomogeneous for all $k\in \poi$, and the Schl\"afli graph and its complement are the \textit{only} graphs that are $4$-ultrahomogeneous, but not $5$-ultrahomogeneous! See \cite{cameron,devillers}.}

 Our goal is now to prove that the Schl\"afli graph $G$ is $H$-deletion-saturated. To see that $G$ is $H$-free, suppose for a contradiction that $G$ has an induced subgraph isomorphic to $H$. It follows in particular that there is a $4$-vertex clique $\{x_1,x_2,y_1,y_2\}$ in $G$ and a $2$-vertex set $Z$ in $G$ disjoint from $\{x_1,x_2,y_1,y_2\}$ such that $Z$ is anticomplete to $\{x_1,x_2\}$ and complete to $\{y_1,y_2\}$ in $G$. Since $G$ is $4$-ultrahomogeneous, we may assume due to symmetry that 
   $$x_1=u^{1}_{1,1},\quad x_2=u^{1}_{2,1},\quad y_1=u^{2}_{2,1},\quad y_2=u^{2}_{2,2}.$$
   But then by the definition of $G$, the set of all vertices that are anticomplete to $\{x_1,x_2\}=\{u^{1}_{1,1},u^{1}_{2,1}\}$ in $G$ is $\{u^1_{3,2},u^1_{3,3},u^2_{1,1},u^2_{1,2},u^2_{1,3}\}$, and exactly one of these vertices, namely $u^1_{3,3}$, is complete to $\{y_1,y_2\}=\{u^{2}_{2,1},u^{2}_{2,2}\}$ in $G$, as well. This is a contradiction to $|Z|=2$.
   
   It remains to show that for every edge $e\in E(G)$, the graph $G-e$ has an induced subgraph isomorphic to $H$. Again, since $G$ is also $4$-ultrahomogeneous (and in particular arc-transitive), we may assume by symmetry that $e=u^1_{1,1}u^1_{1,3}$. Let 
   $$S=\{u^{1}_{1,1},u^{1}_{2,1},u^{1}_{1,3},u^{1}_{3,3},u^2_{2,1},u^2_{2,2}\}.$$
  Using the definition of $G$, it is straightforward to check that the map $f:V(H)\rightarrow S$ with
   $$f(v_1)=u^2_{2,1},\quad f(v_2)=u^1_{1,1},\quad f(v_3)=u^1_{2,1},\quad f(v_4)=u^1_{1,3},\quad f(v_5)=u^1_{3,3},\quad f(v_6)=u^2_{2,2}$$
   is an isomorphism between $H$ and $G[S]-e$. This completes the proof of \Cref{thm:12H}.
\end{proof}
\begin{figure}[t!]
    \centering   \includegraphics[scale=0.7]{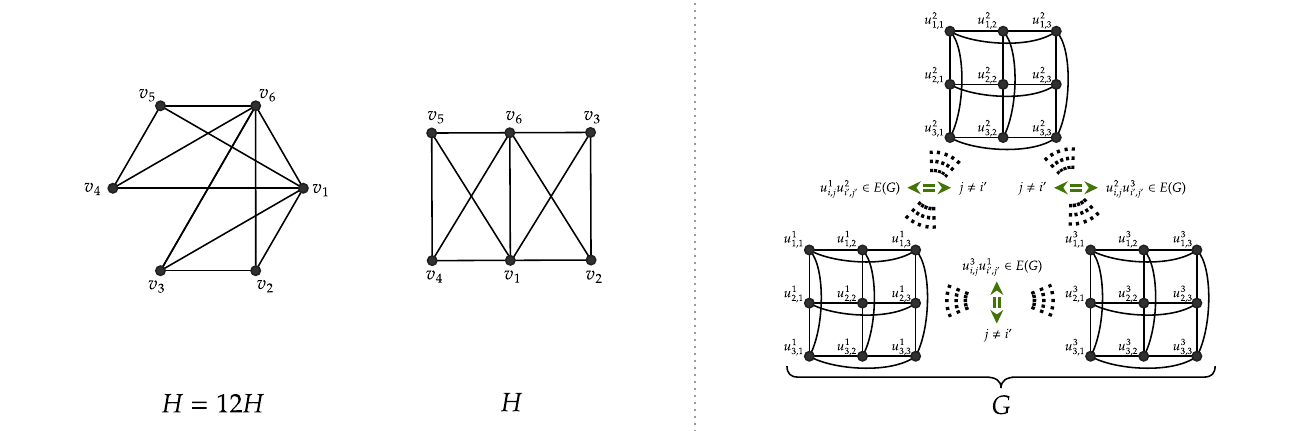}
    \caption{Proof of \Cref{thm:12H}. The graph $H$ (left) and the Schl\"afli graph $G$ (right).}
    \label{fig:schlafli}
\end{figure}

\begin{theorem}\label{thm:13B} Entry ${\rm 13B}$ in \Cref{fig:allsix} is deletion-normal.
\end{theorem}
\begin{figure}[t!]
    \centering   \includegraphics[scale=0.7]{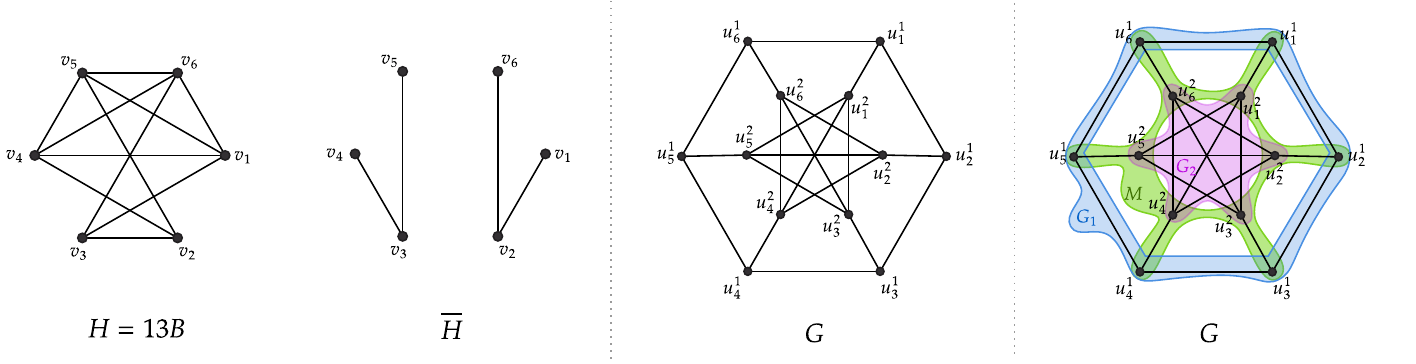}
    \caption{Proof of \Cref{thm:13B}. From left to right: The graph $H$ and its complement, the graph $G$ and its construction.}
    \label{fig:13B}
\end{figure}

\begin{proof}
Let $H$ be entry 13B in \Cref{fig:allsix}. Construct a graph $G$ as follows. Let $G_1$ and $G_2$ be disjoint graphs such that $G_1$ is a $6$-cycle with 
$$V(G_1)=\{u^1_1, \dots, u^1_6\}\quad  \text{and}\quad  E(G_1)=\{u^1_1u^1_{2},u^1_2u^1_{3},u^1_3u^1_{4},u^1_4u^1_{5},u^1_5u^1_{6},u^1_6u^1_1\}$$ 
and $\ol{G_2}$ is a $6$-cycle with $$V(\ol{G_2})= \{u^2_1, \dots, u^2_6\} \quad  \text{and}\quad  E(\ol{G_2})=\{u^2_1u^2_{2},u^2_2u^2_{3},u^2_3u^2_{4},u^2_4u^2_{5},u^2_5u^2_{6},u^2_6u^2_1\}.$$
Let $G$ be the graph obtained from the (disjoint) union of $G_1$ and $G_2$ by adding the matching $M=\{u^1_1u^2_1,\dots, u^1_6u^2_6\}$. See \Cref{fig:13B}.

Since $\ol{H}$ is isomorphic to $2P_3$, in order to show that $H$ is deletion-normal, it suffices to show that $G$ is $2P_3$-addition-saturated.

\sta{\label{st:13bfree}$G$ is $2P_3$-free.}

Suppose for a contradiction that there are two anticomplete induced copies $P^1$ and $P^2$ of $P_3$ in $G$. Let $I_1=V(G)\setminus N[V(P^1)]$ and let $I_2=V(G)\setminus N[V(P^2)]$. Then $P^1$ is an induced copy of $P_3$ in $G[I_2]$ and $P^2$ is an induced copy of $P_3$ in $G[I_1]$. If $V(P^i) \subseteq V(G_1)$ for some $i\in \{1,2\}$, then $G[I_i]$ is isomorphic to $2P_2$, and so $G[I_i]$ is $P_3$-free, a contradiction. It follows that $V(P^i)\cap V(G_2)\neq \varnothing$ for every $i\in \{1,2\}$. Moreover, since $G_2$, which is the complement of a $6$-cycle, is clearly $2P_3$-free, it follows that for some $i\in \{1,2\}$, we have $V(P^i)\cap V(G_1)\neq \varnothing$. Thus, without loss of generality, we may assume that $P^1$ has vertices from both $G_1$ and $G_2$. In particular, we have $E(P^1)\cap M\neq \varnothing$; say $u^1_1u^2_1\in E(P^1)$.  There are now three cases up to symmetry:
    \begin{itemize}
        \item $P^1=u^1_2\dd u^1_1\dd u^2_1$. Then $I_1=\{u^1_4, u^1_5, u^2_6\}$, and $G[I_1]$ is isomorphic to $\ol{P_3}$.
        \item $P^1=u^1_1\dd u^2_1\dd u^2_3$. Then $I_1=\{u^1_4, u^1_5, u^2_2\}$, and $G[I_1]$ is isomorphic to $\ol{P_3}$.
        \item $P^1=u^1_1\dd u^2_1\dd u^2_4$. Then $I_1=\{u^1_3, u^1_5\}$ is a stable set.
    \end{itemize}
    But in all three cases, $G[I_1]$ is $P_3$-free, a contradiction. This proves \eqref{st:13bfree}.

\sta{\label{st:13bcritical} For every $e\in E(\ol{G})$, the graph $G+e$ has an induced subgraph isomorphic to $2P_3$.}

We need to prove that there are two anticomplete induced copies $P^1$ and $P^2$ of $P_3$ in $G+e$. First, assume that $ e\in E(\ol{G_1})$. Then there are two cases up to symmetry:
    \begin{itemize}
        \item  $e=u^1_1u^1_3$. In this case, $P^1=u^1_6\dd u^1_1\dd u^1_3$ and $P^2=u^2_4\dd u^2_2\dd u^2_5$ work.
        \item $e=u^1_1u^1_4$. In this case, $P^1=u^1_1\dd u^1_4\dd u^1_5$ and $P^2=u^2_2\dd u^2_6\dd u^2_3$ work.
    \end{itemize}
    Next, assume that $ e\in E(\ol{G_2})$. Then, by symmetry, we may assume that $e=u^2_1u^2_6$, in which case $P^1=u^1_2\dd u^1_3\dd u^1_4$ and $P^2=u^2_5\dd u^2_1\dd u^2_6$ work. Finally, assume that $e=u^1_iu^2_j$ for some $i,j\in \{1,\ldots, 6\}$ with $i\neq j$. Then there are three cases up to symmetry:
    \begin{itemize}
        \item $e=u^1_1u^2_2$.  In this case, $P^1=u^1_6\dd u^1_1\dd u^2_2$ and $P^2=u^1_4\dd u^1_3\dd u^2_3$ work.
        \item $e=u^1_1u^2_3$. In this case, $P^1=u^1_6\dd u^1_1\dd u^2_3$ and $P^2=u^1_4\dd u^2_4\dd u^2_2$ work.
        \item $e=u^1_1u^2_4$. In this case, $P^1=u^1_6\dd u^1_1\dd u^2_4$ and $P^2=u^1_3\dd u^2_3\dd u^2_5$ work.
    \end{itemize}
    This proves \eqref{st:13bcritical}.
    \medskip

The result now follows from \eqref{st:13bfree} and \eqref{st:13bcritical}. This completes the proof of \Cref{thm:13B}.
\end{proof}

\begin{theorem}\label{thm:15I} Entry {\rm 15I} in \Cref{fig:allsix} is deletion-normal.
\end{theorem}
\begin{figure}[t!]
    \centering   \includegraphics[scale=0.7]{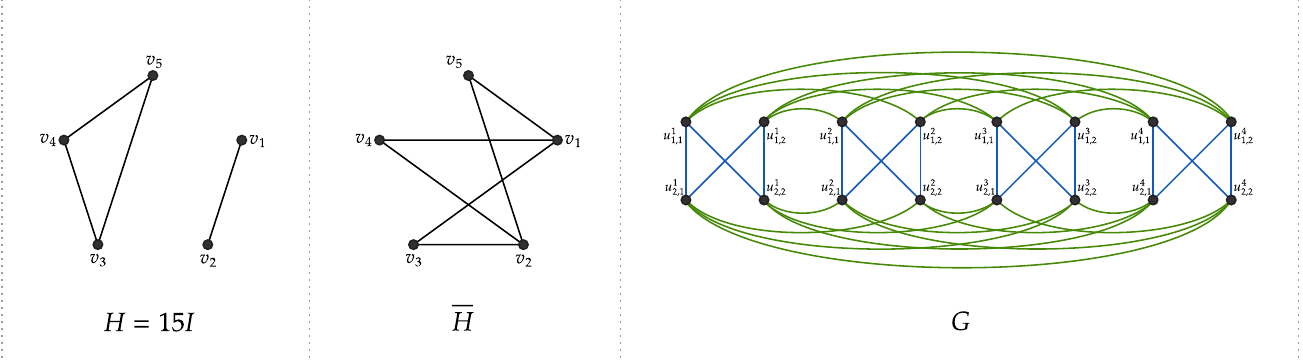}
    \caption{Proof of \Cref{thm:15I}. From left to right: The graph $H$ and its complement, and the graph $G$.}
    \label{fig:15I}
\end{figure}

\begin{proof}
Let $H$ be entry 15I in \Cref{fig:allsix}. Let $G$ be a graph with $$V(G)=\left\{u^t_{a,b}: a,b\in\{1,2\},t\in \{1,2,3,4\}\right\}$$
and $E(G)=E_1\cup E_2$, where $$E_1=\left\{u^t_{a,b}u^t_{c,d}: a,b,c,d\in\{1,2\},\, t\in \{1,2,3,4\},\,  a\neq c\right\}$$
$$E_2=\left\{u^t_{a,b}u^r_{a,d}: a,b,d\in\{1,2\},\, r,t\in \{1,2,3,4\},\, t\neq r,\, b\neq d \right\}.$$
See \Cref{fig:15I}.

Since $\ol{H}$ is isomorphic to $K_{2,3}$, it suffices to show that $G$ is $K_{2,3}$-addition-saturated. First, we show that $G$ is $K_{2,3}$-free. Suppose not. Then there are two non-adjacent vertices $x,y\in V(G)$ such that there is a stable set $Z\subseteq N_G(x)\cap N_G(y)$ with $|Z|=3$. There are now three cases up to symmetry:
\begin{itemize}
    \item $x=u^1_{1,1}$ and $y=u^1_{1,2}$. Then $N_G(x)\cap N_G(y)=\{u^1_{2,1},u^1_{2,2}\}$.
    
    \item $x=u^1_{1,1}$ and $y=u^2_{1,1}$. Then $N_G(x)\cap N_G(y)=\{u^3_{1,2},u^4_{1,2}\}$.
    
    \item $x=u^1_{1,1}$ and $y=u^2_{2,1}$. Then $N_G(x)\cap N_G(y)=\{u^1_{2,2},u^2_{1,2}\}$.
\end{itemize}
In particular, $|N_G(x)\cap N_G(y)|=2$, a contradiction. We deduce that $G$ is $K_{2,3}$-free.
\medskip

It remains to show that for every edge $e=xy\in E(\ol{G})$, the graph $G+e$ has an induced subgraph isomorphic to $K_{2,3}$. Again, there are now three cases up to symmetry:
\begin{itemize}
    \item $x=u^1_{1,1}$ and $y=u^1_{1,2}$. Then $Z_1=\{u^1_{1,2},u^3_{1,1}\}$ and $Z_2=\{u^1_{1,1},u^2_{1,1},u^4_{1,1}\}$ are two stable sets in $G+e$ such that every vertex in $Z_1$ is adjacent in $G+e$ to every vertex in $Z_2$.
    
    \item $x=u^1_{1,1}$ and $y=u^2_{1,1}$. Then $Z_1=\{u^1_{2,2},u^2_{1,1}\}$ and $Z_2=\{u^1_{1,1},u^1_{1,2},u^2_{2,1}\}$ are two stable sets in $G+e$ such that every vertex in $Z_1$ is adjacent in $G+e$ to every vertex in $Z_2$.
    
    \item $x=u^1_{1,1}$ and $y=u^2_{2,1}$. Then $Z_1=\{u^1_{1,1},u^2_{1,1}\}$ and $Z_2=\{u^2_{2,1},u^3_{1,2},u^4_{1,2}\}$ are two stable sets in $G+e$ such that every vertex in $Z_1$ is adjacent in $G+e$ to every vertex in $Z_2$.
\end{itemize}
This completes the proof of \Cref{thm:15I}.
\end{proof}

\begin{figure}[t!]
    \centering   \includegraphics[scale=0.28]{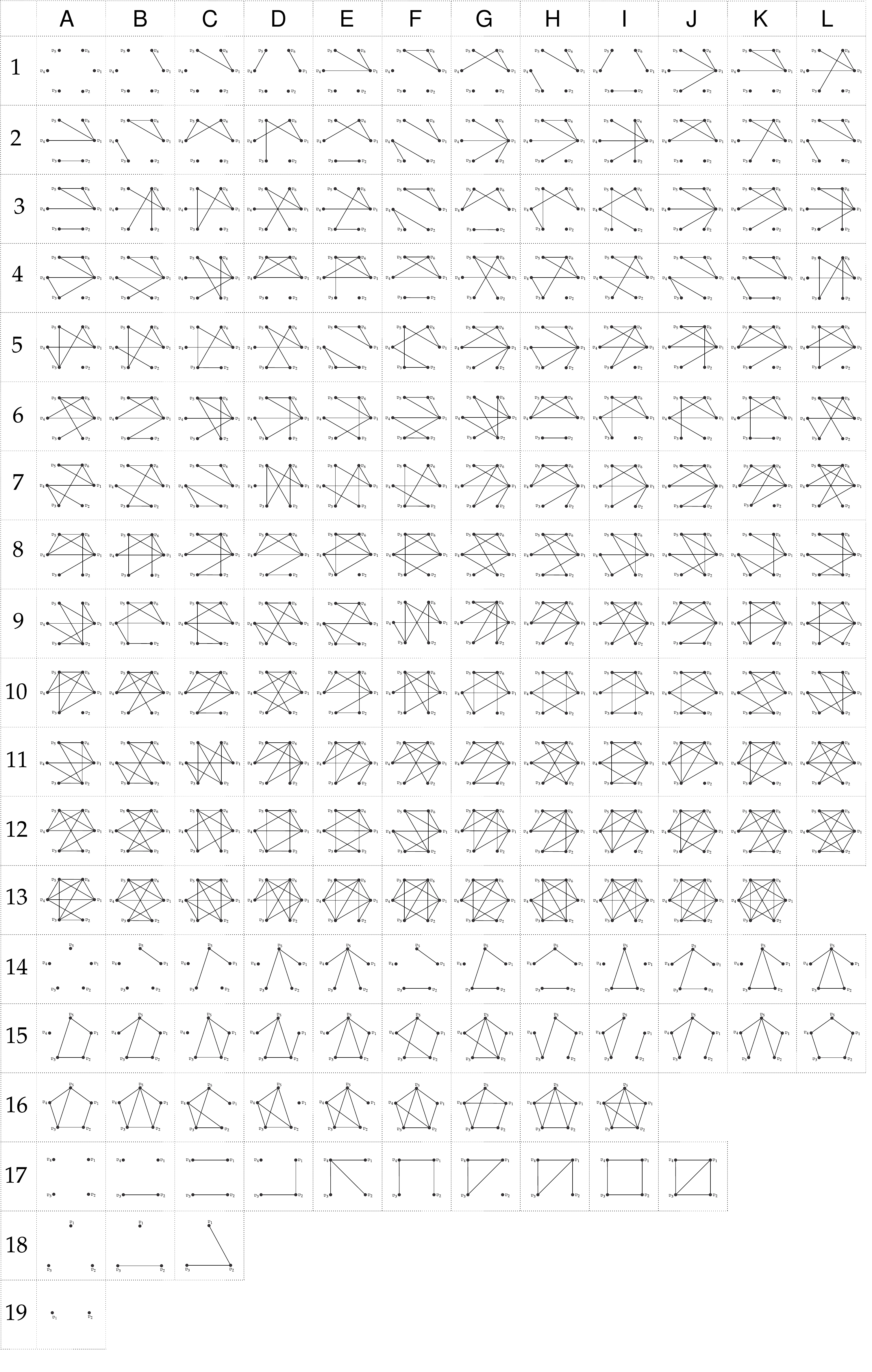}
    \caption{All pairwise non-isomorphic non-complete graphs on at most six vertices.}
    \label{fig:allsix}
\end{figure}
\newpage

\begin{figure}[t!]
    \centering   \includegraphics[width=\linewidth]{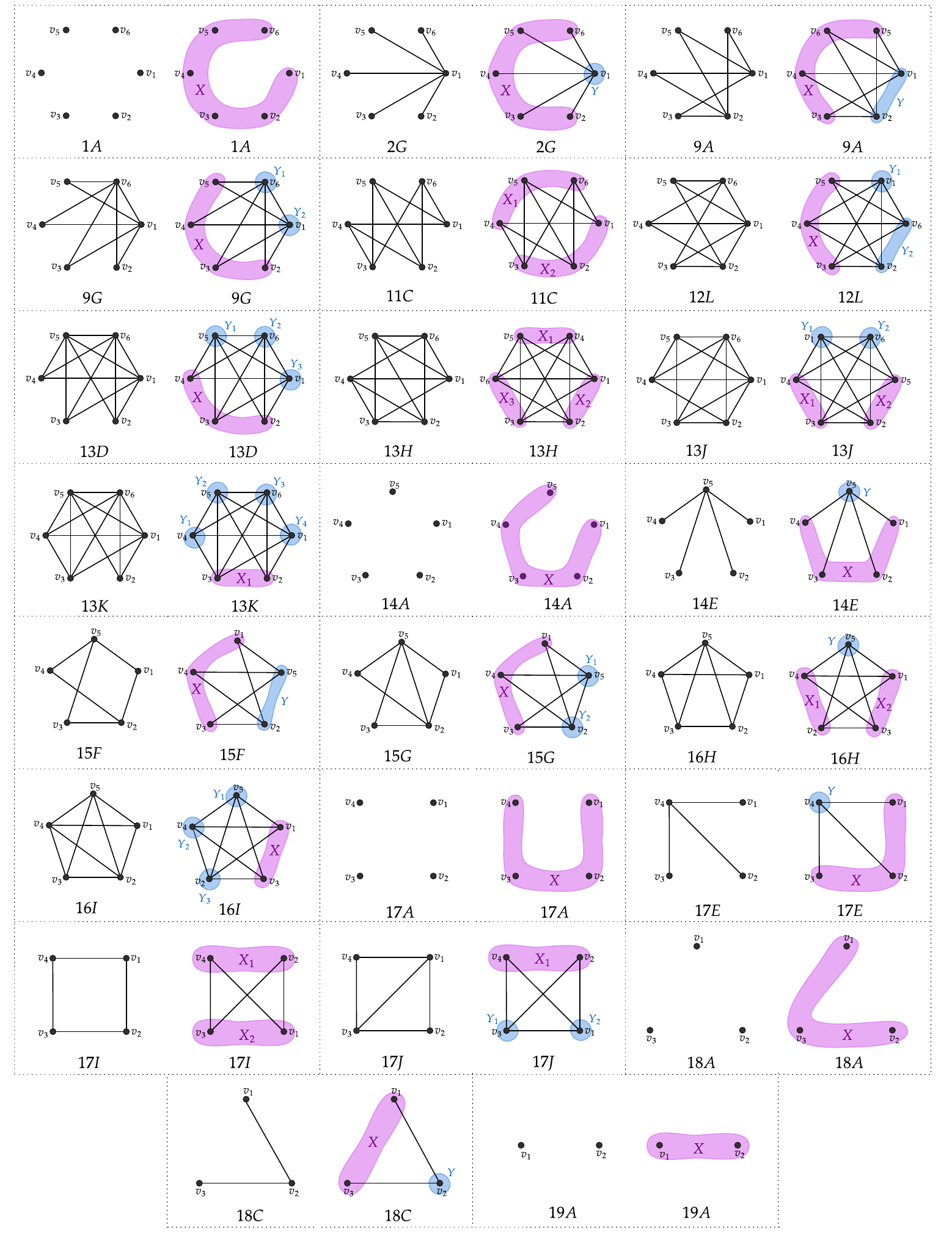}
    \caption{All non-complete complete multipartite graphs on at most six vertices.}
    \label{fig:K_(s,t,t)_Table}
\end{figure}
\newpage
\begin{figure}[t!]
    \centering   \includegraphics[width=\linewidth]{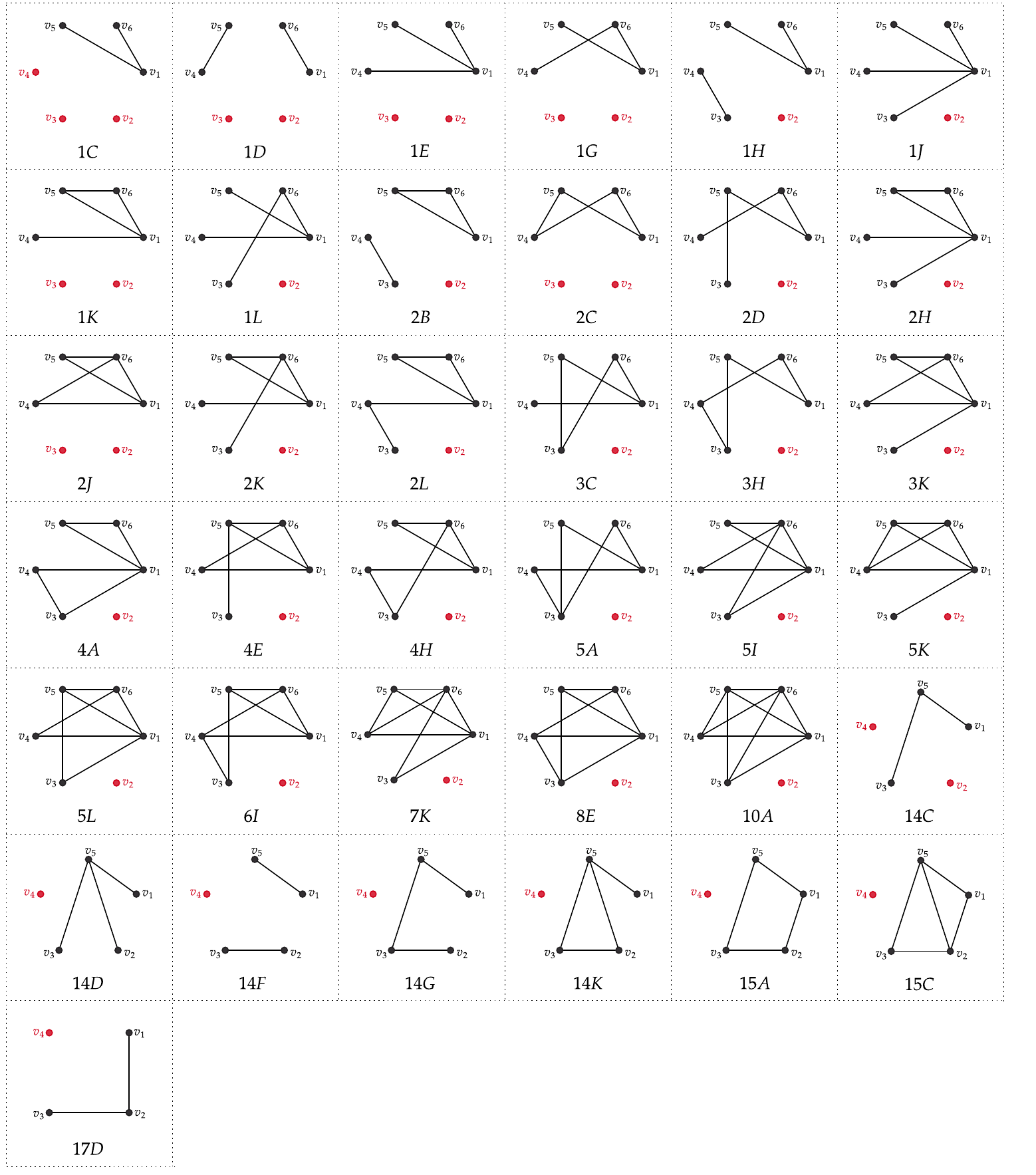}
    \caption{These graphs have at least one isolated vertex, and removing all isolated vertices from each graph leaves a non-complete graph.}
    \label{fig:isolated}
\end{figure}
\newpage

\begin{figure}[t!]
    \centering   \includegraphics[width=\linewidth]{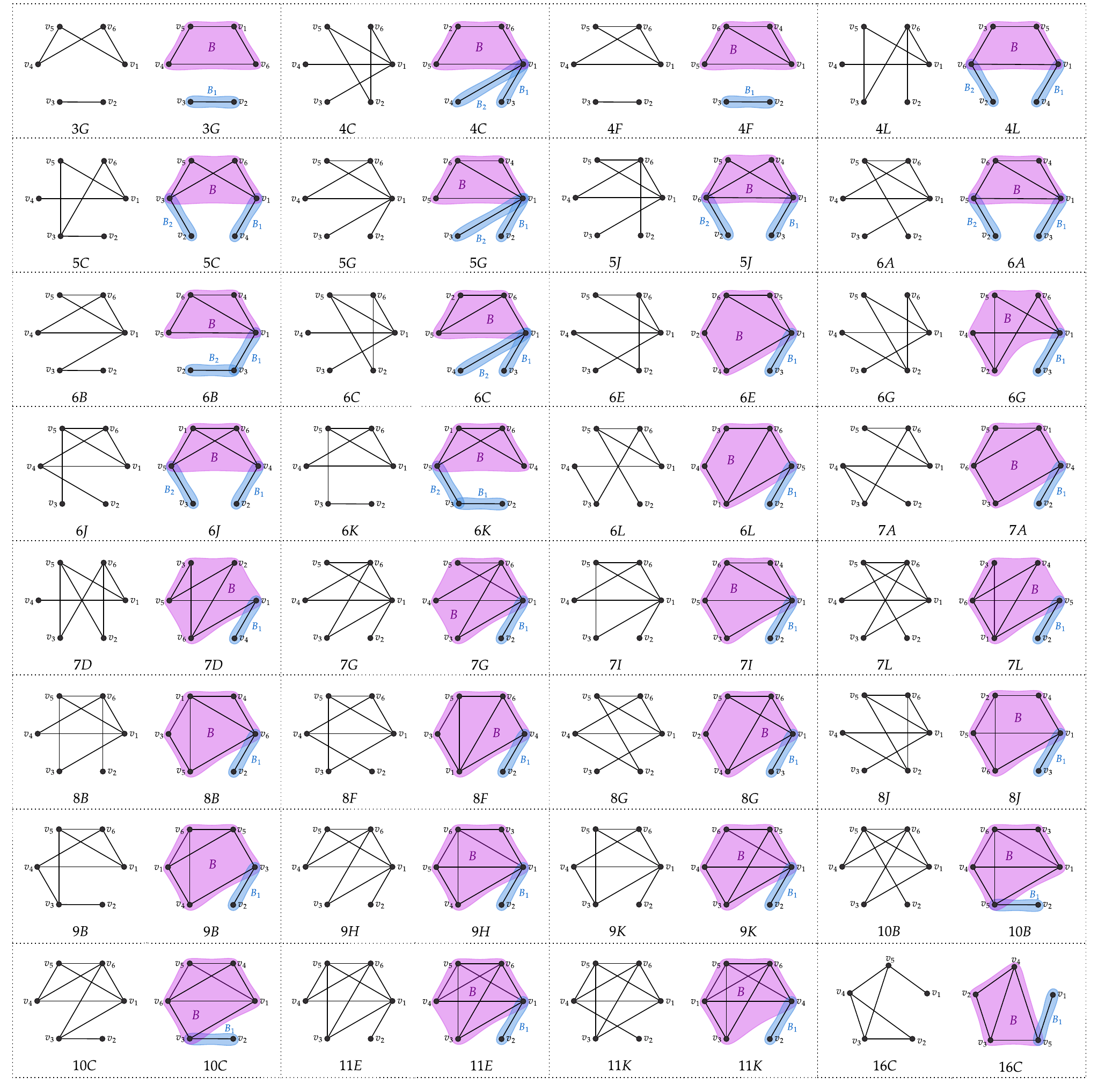}
    \caption{These graphs each have a governing block $B$, and at least one other block, and all blocks other than $B$ are isomorphic to $K_2$.}
    \label{fig:block}
\end{figure}

\begin{figure}[t!]
    \centering   \includegraphics[scale=0.55]{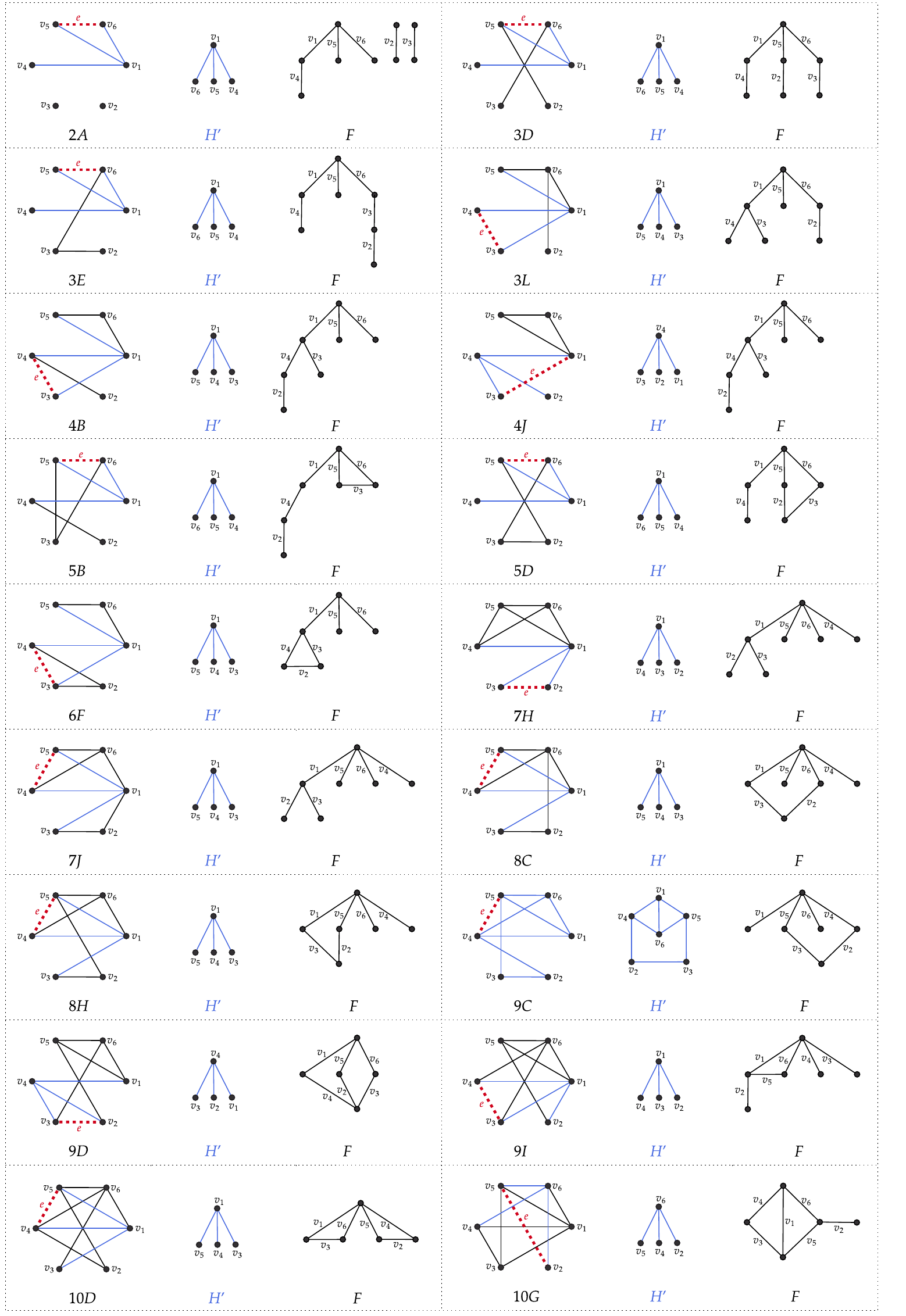}
    \caption{These graphs are all line\lm graphs. In each entry, we have a drawing of the graph $H$ itself with a specified non-edge $e$, a drawing of an induced subgraph $H'$ of $H$ that by \Cref{thm:lineobst} prevents $H$ from being a line graph, and then a drawing of the graph $F$ whose line graph is $H+e$.}
    \label{fig:lineminus1}
\end{figure}

\begin{figure}[t!]
    \centering   \includegraphics[scale=0.6]{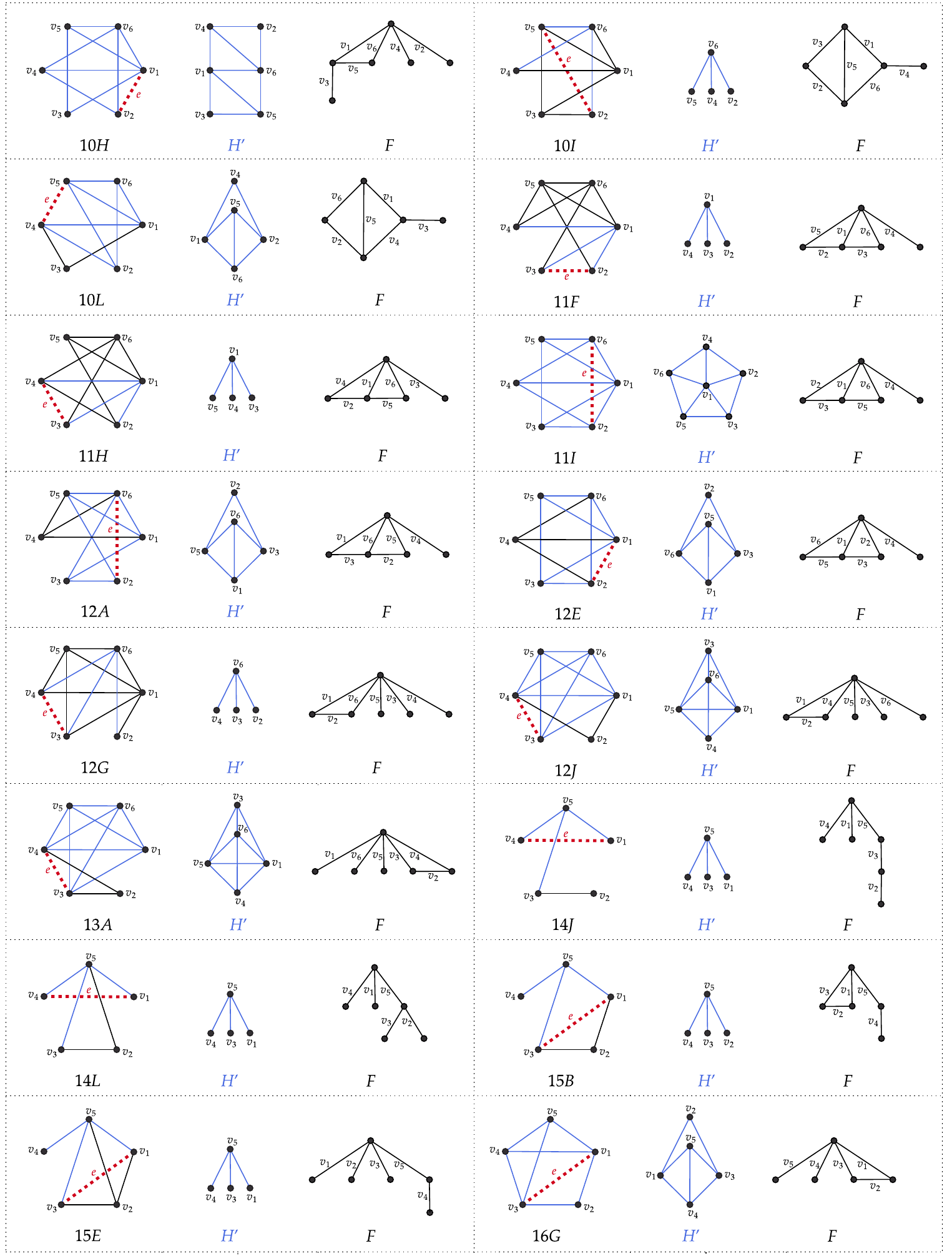}
    \caption{These graphs are all line\lm graphs. In each entry, we have a drawing of the graph $H$ itself with a specified non-edge $e$, a drawing of an induced subgraph $H'$ of $H$ that by \Cref{thm:lineobst} prevents $H$ from being a line graph, and then a drawing of the graph $F$ whose line graph is $H+e$.}
    \label{fig:lineminus2}
\end{figure}

\begin{figure}[t!]
    \centering   \includegraphics[scale=0.6]{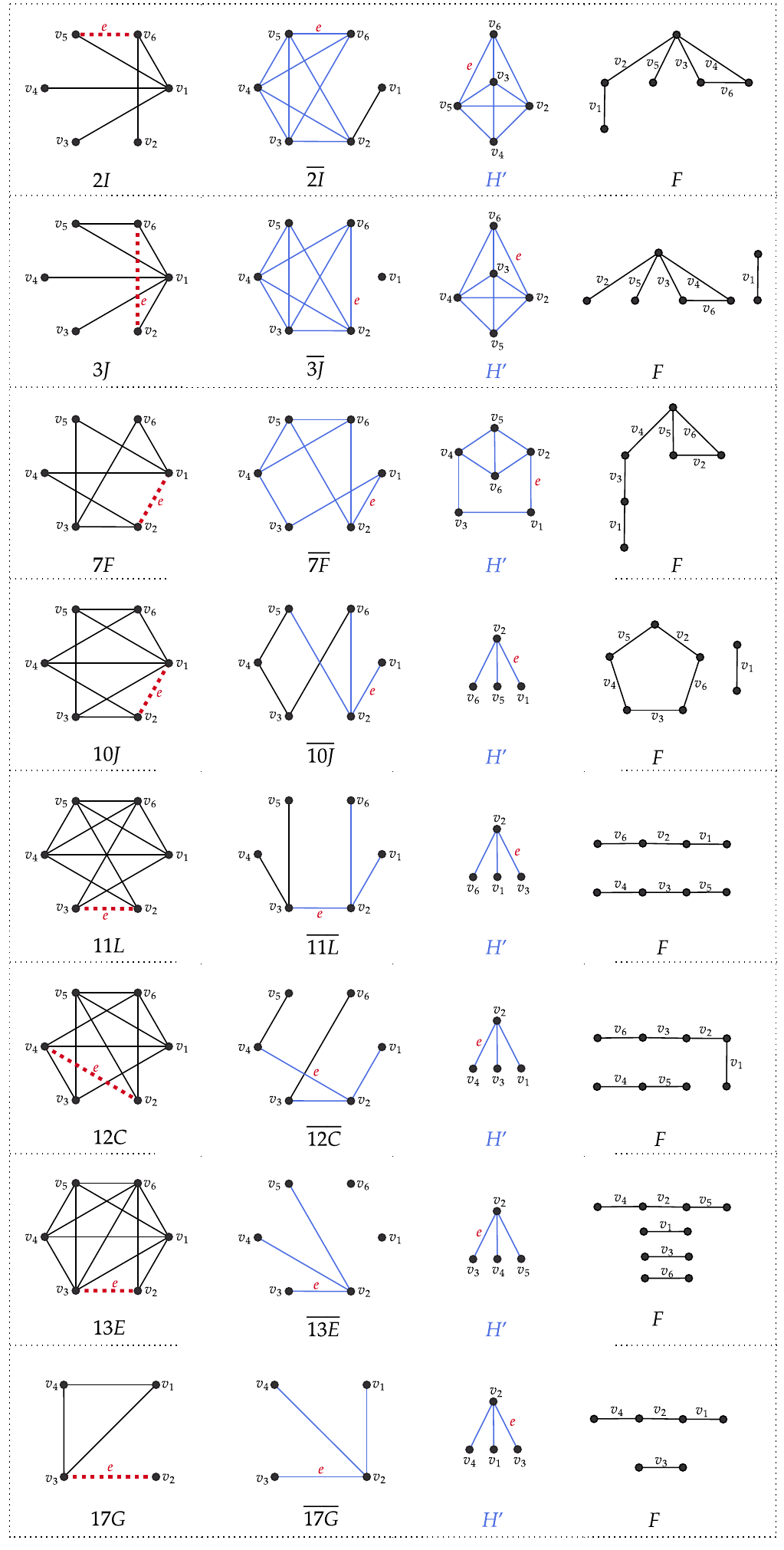}
    \caption{These graphs are all complements of line\lp graphs. In each entry, we have a drawing of the graph $H$ itself, a drawing of $\ol{H}$ with a specified edge $e$ of $\ol{H}$, a drawing of an induced subgraph $H'$ of $\ol{H}$ that by \Cref{thm:lineobst} prevents $\ol{H}$ from being a line graph, and then a drawing of the graph $F$ whose line graph is $\ol{H}-e$.}
    \label{fig:lineplus}
\end{figure}

\begin{figure}[t!]
    \centering   \includegraphics[width=\linewidth]{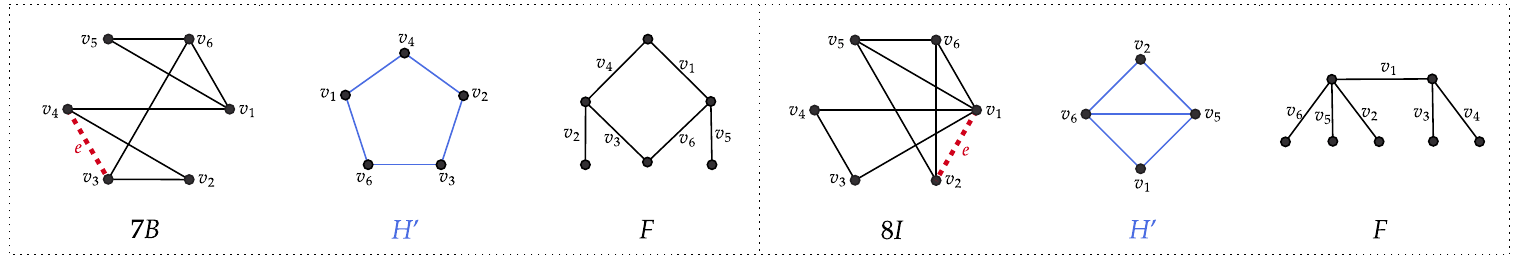}
    \caption{These graphs are all biline\lm graphs. In each entry, we have a drawing of the graph $H$ itself with a specified non-edge $e$, a drawing of an induced subgraph $H'$ of $H$ that by \Cref{thm:bilineobst} prevents $H$ from being a biline graph, and then a drawing of the bipartite graph $F$ whose line graph is $H+e$.}
    \label{fig:bilineminus}
\end{figure}
\medskip
\medskip
\medskip

\begin{figure}[t!]
    \centering   \includegraphics[width=\linewidth]{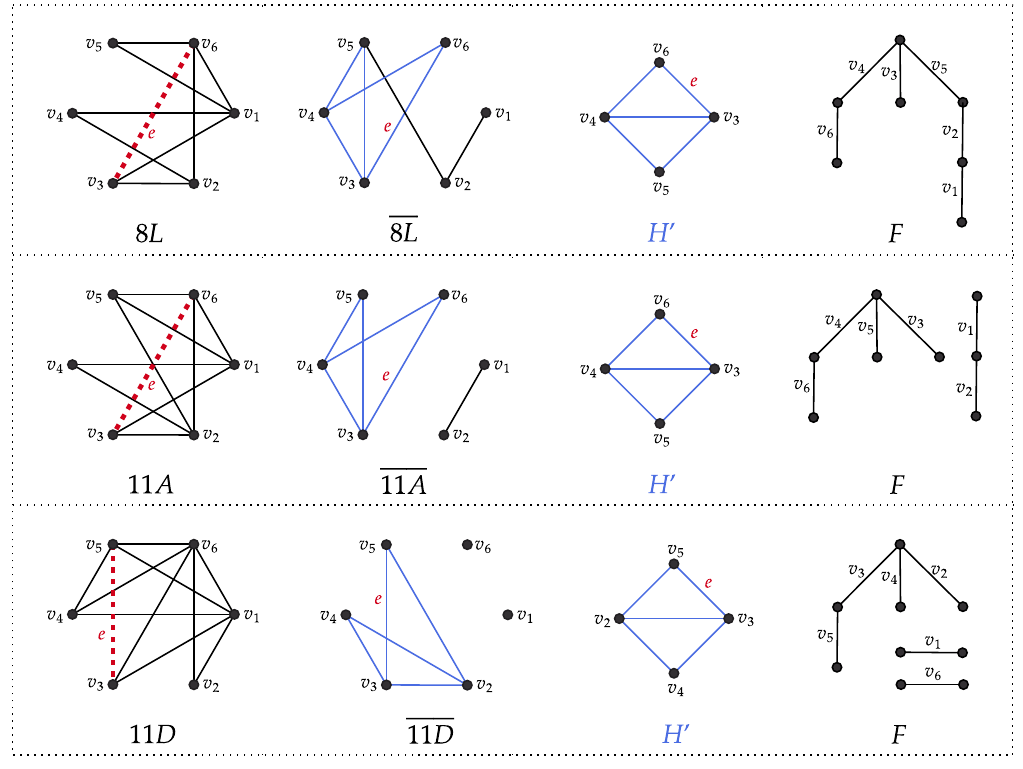}
    \caption{These graphs are all complements of biline\lp graphs. In each entry, we have a drawing of the graph $H$ itself, a drawing of $\ol{H}$ with a specified edge $e$ of $\ol{H}$, a drawing of an induced subgraph $H'$ of $\ol{H}$ that by \Cref{thm:bilineobst} prevents $\ol{H}$ from being a biline graph, and then a drawing of the bipartite graph $F$ whose line graph is $\ol{H}-e$.}
    \label{fig:bilineplus}
\end{figure}

\begin{figure}[t!]
    \centering   \includegraphics[scale=0.4]{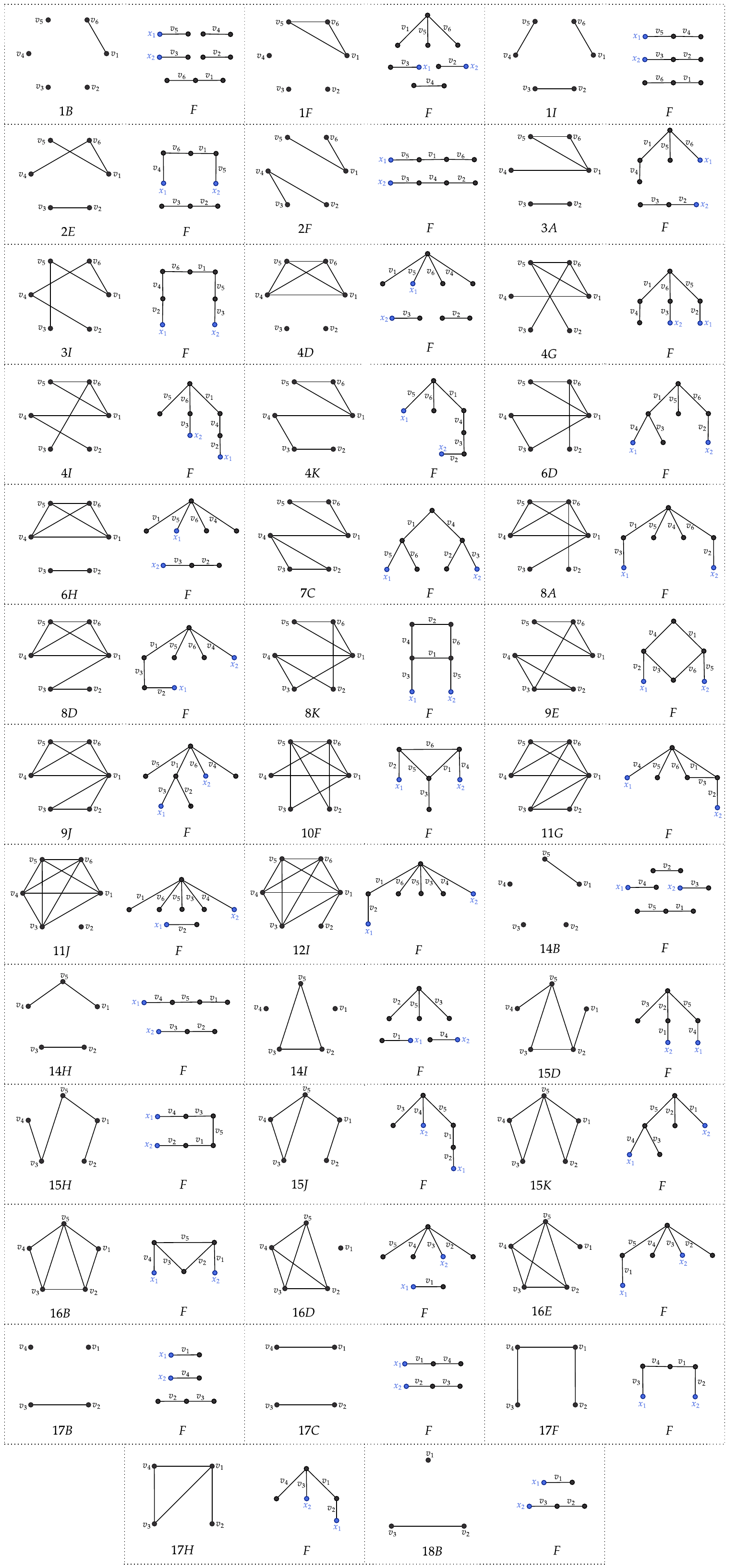}
    \caption{Every graph here is the line graph of a graph $F$ with a severed pair. In each entry, we have a drawing of the graph $H$ itself, and then a drawing of the graph $F$ with a severed pair $(x_1,x_2)$ such that $L(F)=H$.}
    \label{fig:severed}
\end{figure}

\end{document}